\documentclass[11pt]{article}

\usepackage{tikz}
\usepackage{graphicx,wrapfig,lipsum}   

\usepackage{xspace}

\usepackage{amsmath,amssymb}
\usepackage{amsthm}
\usepackage{mathtools}
\usepackage[noend]{algorithmic}
\usepackage[ruled,vlined]{algorithm2e}
\usepackage{url}
\usepackage{fullpage}
\usepackage{makeidx}
\usepackage{enumerate}
\usepackage[top=1in, bottom=1.25in, left=1in, right=1in]{geometry}
\usepackage{graphicx,float,psfrag,epsfig,caption}

\usepackage{hyperref}
\usepackage{epstopdf}
\usepackage{color}
\usepackage{xr}
\usepackage{subfig}
\usepackage{caption}
\usepackage[utf8]{inputenc}

\usepackage{scalerel,stackengine}

\usepackage{enumitem}

\stackMath
\newcommand\reallywidehat[1]{%
\savestack{\tmpbox}{\stretchto{%
  \scaleto{%
    \scalerel*[\widthof{\ensuremath{#1}}]{\kern.1pt\mathchar"0362\kern.1pt}%
    {\rule{0ex}{\textheight}}
  }{\textheight}%
}{2.4ex}}%
\stackon[-6.9pt]{#1}{\tmpbox}%
}
\usepackage[mathscr]{euscript}

\DeclareSymbolFont{rsfs}{U}{rsfs}{m}{n}
\DeclareSymbolFontAlphabet{\mathscrsfs}{rsfs}

\numberwithin{equation}{section}

\newtheoremstyle{myexample} 
    {\topsep}                    
    {\topsep}                    
    {\rm }                   
    {}                           
    {\bf }                   
    {.}                          
    {.5em}                       
    {}  

\newtheoremstyle{myremark} 
    {\topsep}                    
    {\topsep}                    
    {\rm}                        
    {}                           
    {\bf}                        
    {.}                          
    {.5em}                       
    {}  

\newtheorem{claim}{Claim}[section]
\newtheorem{lemma}[claim]{Lemma}

\newtheorem{assumption}{Assumption}[]

\newtheorem{theorem}{Theorem}
\newtheorem{proposition}[claim]{Proposition}
\newtheorem{corollary}[claim]{Corollary}
\newtheorem{definition}[claim]{Definition}

\theoremstyle{myremark}
\newtheorem{remark}{Remark}[section]

\theoremstyle{myremark}

\theoremstyle{myexample}

\definecolor{darkgreen}{rgb}{0.0, 0.5, 0.0}

\newcommand{\bea}{\begin{eqnarray}}
\newcommand{\eea}{\end{eqnarray}}
\newcommand{\<}{\langle}
\renewcommand{\>}{\rangle}
\newcommand{\E}{{\mathbb E}}

\def\sTV{\mbox{\tiny \rm TV}}

\def\salg{\mbox{\tiny \rm alg}}
\def\sspher{\mbox{\tiny \rm spher}}

\def\oS{{\overline S}}
\def\eps{{\varepsilon}}

\def\id{{\boldsymbol{I}}}

\def\cuE{\mathscrsfs{E}}
\def\cuU{\mathscrsfs{U}}
\def\cuL{\mathscrsfs{L}}

\def\ind{{\mathbb I}}

\def\toLx{\stackrel{L^1_{\xi}}{\longrightarrow}}

\def\SF{{\sf SF}}

\def\bsigma{{\boldsymbol{\sigma}}}

\def\bfe{{\boldsymbol{e}}}

\def\bZ{{\boldsymbol{Z}}}

\def\bP{{\boldsymbol{P}}}

\def\bC{{\boldsymbol{C}}}
\def\bQ{{\boldsymbol{Q}}}
\def\bM{{\boldsymbol{M}}}
\def\bT{{\boldsymbol{T}}}

\def\bV{{\boldsymbol{V}}}

\def\bm{{\boldsymbol m}}

\def\bg{{\boldsymbol{g}}}

\def\cF{{\mathcal F}}

\def\cT{{\mathcal T}}

\def\cW{{\mathcal W}}

\def\blambda{{\boldsymbol \lambda}}

\def\tH{\tilde{H}}

\def\op{\mbox{\tiny\rm op}}

\def\naturals{{\mathbb N}}
\def\reals{{\mathbb R}}

\def\integers{{\mathbb Z}}

\def\normal{{\sf N}}

\def\sT{{\sf T}}

\def\bv{{\boldsymbol{v}}}
\def\bz{{\boldsymbol{z}}}
\def\bx{{\boldsymbol{x}}}
\def\bq{{\boldsymbol{q}}}
\def\ba{{\boldsymbol{a}}}
\def\bb{{\boldsymbol{b}}}

\def\bA{\boldsymbol{A}}
\def\bB{\boldsymbol{B}}

\def\Par{{\sf P}}
\def\de{{\rm d}}

\def\bX{\boldsymbol{X}}
\def\bY{\boldsymbol{Y}}
\def\bW{\boldsymbol{W}}
\def\prob{{\mathbb P}}
\def\E{{\mathbb E}}

\def\<{\langle}
\def\>{\rangle}

\def\sign{{\rm sign}}

\def\ed{\stackrel{{\rm d}}{=}}

\def\cN{{\cal N}}

\def\cV{{\cal V}}
\def\cL{{\cal L}}

\def\by{{\boldsymbol{y}}}

\def\P{\mathbb{P}}
\def\cP{{\mathcal P}}

\def\toP{\stackrel{p}{\longrightarrow}}

\def\blambda{{\boldsymbol{\lambda}}}
\def\bD{{\boldsymbol{D}}}

\def\bu{{\boldsymbol{u}}}
\def\b0{{\boldsymbol{0}}}

\def\Var{{\rm Var}}

\def\bfone{{\boldsymbol 1}}

\def\bF{{\boldsymbol F}}
\def\bG{{\boldsymbol G}}

\DeclareMathOperator*{\plim}{p-lim}

\def\SOS{{\sf SOS}}
\def\OPT{{\sf OPT}}
\def\VAL{{\small \sf VAL}}

\def\cU{{{\mathcal U}}}
\def\tcL{{\tilde {\mathcal L}}}
\def\tcT{{\tilde {\mathcal T}}}

\def\cS{{\mathcal S}}

\def\cP{{\mathcal P}}

\def\balpha{{\boldsymbol \alpha}}

\def\bfzero{\boldsymbol{0}}

\def\ons{\mathbf{ons}}
\def\ONS{\mathbf{ONS}}
\def\X{{\boldsymbol{X}}}
\def\ZZ{{\boldsymbol{Z}}}
\def\W{{\boldsymbol{W}}}
\def\A{{\boldsymbol{A}}}
\def\F{{\boldsymbol{F}}}
\def\AMP{{\sf AMP}}
\def\LAMP{{\sf LAMP}}

\def\qq{{\boldsymbol{q}}}

\def\VV{{\boldsymbol{V}}}
\def\vv{{\boldsymbol{v}}}

\def\M{{\boldsymbol{m}}}

\newcommand{\x}{{\boldsymbol x}}

\newcommand{\z}{{\boldsymbol z}}

\newcommand{\f}{{\boldsymbol f}}
\renewcommand{\u}{{\boldsymbol u}}
\renewcommand{\v}{{\boldsymbol v}}
\renewcommand{\b}{{\boldsymbol b}}
\newcommand{\m}{{\boldsymbol m}}

\def\beps{{\boldsymbol \eps}}

\def\psim{\stackrel{p}{\simeq}}

\def\bcG{{\boldsymbol \cal G}}
\def\Sym{{\rm Sym}}
\def\VF{{\mathcal J}}
\def\sp{{\sf p}}
\def\sP{{\sf P}}
\def\sE{{\sf E}}
\def\sVar{{\sf Var}}
\def\hbm{{\hat{\boldsymbol m}}}
\def\tbm{{\tilde{\boldsymbol m}}}
\def\hm{\hat{m}}
\def\tH{\tilde{H}}
\def\bLambda{{\boldsymbol \Lambda}}
\def\spn{{\rm span}}

\def\km{D}

\def\hZZ{\hat{\boldsymbol{Z}}}

\newcommand{\R}{\mathbb{R}}

\def\hbA{\hat{\boldsymbol A}}
\def\tbA{\tilde{\boldsymbol A}}
\def\hbz{\hat{\boldsymbol z}}

\def\bfone{{\boldsymbol 1}}

\title{Optimization of Mean-field Spin Glasses}

\author{Ahmed El Alaoui\thanks{Department of Electrical Engineering, Stanford University}, \;\; Andrea Montanari\thanks{Department of Electrical Engineering and
  Department of Statistics, Stanford University}, \;\; Mark Sellke\thanks{Department of Mathematics, Stanford University}}

\date{}

\begin{document}

\maketitle

\begin{abstract}
Mean-field spin glasses are families of random energy functions (Hamiltonians) on high-dimensional product spaces. In this paper we consider the case
of Ising mixed $p$-spin models, namely Hamiltonians $H_N:\Sigma_N\to \reals$ on the Hamming hypercube $\Sigma_N = \{\pm 1\}^N$, which are defined by the property that
$\{H_N(\bsigma)\}_{\bsigma\in \Sigma_N}$ is a centered Gaussian process with covariance $\E\{H_N(\bsigma_1)H_N(\bsigma_2)\}$ depending only on the scalar 
product $\<\bsigma_1,\bsigma_2\>$. 

The asymptotic value of the optimum $\max_{\bsigma \in \Sigma_N}H_N(\bsigma)$ was characterized in terms of a variational principle known as the Parisi formula, first proved by 
Talagrand and, in a more general setting, by Panchenko. The structure of superlevel sets is extremely rich and has been studied by a number of authors. Here we ask 
 whether a near optimal configuration $\bsigma$ can be computed in polynomial time.

 We develop a message passing algorithm whose complexity per-iteration is of the same order as the complexity
 of evaluating the gradient of $H_N$, and  characterize the typical energy value it achieves.
 When the $p$-spin model $H_N$ satisfies a certain no-overlap gap assumption, for any $\eps>0$,
 the algorithm outputs $\bsigma\in\Sigma_N$ such that $H_N(\bsigma)\ge (1-\eps)\max_{\bsigma'} H_N(\bsigma')$, with high probability.
 The number of iterations is bounded in $N$ and depends uniquely on $\eps$. More generally, regardless of whether the no-overlap gap assumption holds, the energy achieved is given by an extended  variational principle, which generalizes the Parisi formula.
\end{abstract}

\section{Introduction}
\label{sec:Introduction}

Let $\bW^{(k)} \in (\reals^N)^{\otimes k}$, $k\ge 2$,  be an standard symmetric Gaussian tensor of order $k$
with entries $\bW^{(k)} \equiv  (W^{(k)}_{1\le i_1,\cdots,i_k})_{i_1,\cdots,i_k\le N}$. 
Namely, if $\{G^{(k)}_{i_1,\dots, i_k}\}_{k\ge 2; 1 \le  i_1,\cdots,i_k\le N}\sim_{iid}\normal(0,1)$ is a collection of standard normal random variables, we set
$\bW^{(k)} \equiv N^{-(k-1)/2} \sum_{\pi \in S_k} (\bG^{(k)})^{\pi}$
where the sum is over the group of permutations of $k$ objects, and $(\bG^{(k)})^{\pi}$ is obtained by permuting the indices of  $\bG^{(k)}$ according to $\pi$.
In particular, $\E\{W^{(k)}_{i_1,\dots,i_k})^2 \} = k!/N^{k-1}$ for $i_1<i_2<\dots< i_k$.

We consider the problem of optimizing a polynomial with coefficients given by the tensors $\bW^{(k)}$ over the hypercube $\Sigma_N=\{-1,+1\}^N$:
\begin{align}\label{eq:hamiltonian}
\OPT_N & = \frac{1}{N}\max\Big\{ H_N(\bsigma) \, : \;\;\bsigma\in \Sigma_N \Big\}\, ,\\
H_N(\bsigma) &=  \sum_{k=2}^{\infty} \frac{c_k}{k!}  \<\bW^{(k)},\bsigma^{\otimes k}\>\, ,\;\;\;\;\;\;
\<\bW^{(k)},\bsigma^{\otimes k}\>\equiv
\sum_{1\le i_1 ,\cdots , i_k\le N} W_{i_1,\cdots,i_p}\sigma_{i_1}\cdots\sigma_{i_k}\, .
\end{align}
The parameters $(c_k)_{k\ge 2}$ are customarily encoded in the function  $\xi(x) \equiv \sum_{k\ge 2}c_k^2x^k$ which we henceforth call the \emph{mixture} of the model. We will assume throughout that $\xi(1+\eps)<\infty$ for some $\eps>0$. This implies $|c_k|\le c_*\alpha^k$ for some $c_*>0, \alpha\in (0,1)$, so that the sum defining $H_{N}$ is almost surely finite.
(In fact, there is very little loss of generality in assuming $c_k=0$ for all $k$ larger than some absolute constant  $k_M$.) 

We would like to develop an algorithm that accepts as input the tensors $(\bW^{(k)})_{k\ge 2}$ and returns a vector $\bsigma^*\in \Sigma_N$
such that, with high-probability,  $H_N(\bsigma^{\salg})/N\ge\rho\cdot\OPT_N$ for an approximation factor $\rho\in [0,1]$ as close to one as possible.
From a worst case point of view, this objective is hopeless: achieving any $\rho>1/(\log N)^c$ (for $c$ a small constant) 
is NP-hard already in the case of quadratic polynomials \cite{arora2005non}. 
For higher-order polynomials, the task is known to be even more difficult. For instance,  \cite{barak2012hypercontractivity} 
proves that obtaining $\rho> \exp(-(\log N)^c\}$ is hard already for the case in which a single term $c_k$, $k\ge 3$ is non-vanishing
and the combinatorial constraint  $\bsigma \in \Sigma_N$ is relaxed to $\|\bsigma\|_2^2\le N$. 

Worst-case hardness results do not have direct implications on random instances, as described above. 
However, standard optimization methods based on semidefinite programming (SDP) relaxations appear to fail 
on such random instances. These methods typically produce an efficiently computable upper bound on $\OPT_N$. For the case
of an homogeneous polynomial (i.e., $c_k=1$ for some $k\ge 3$, and $c_{k'}=0$ for all $k'\neq k$), \cite{bhattiprolu2016sum}
shows that a level-$k$ sum-of-squares relaxation produces an upper bound that is polynomially larger  than
$\OPT_N$: $\SOS_N(k) \gtrsim N^{(k-2)/4}\cdot \OPT_N$.
In contrast, significant progress has been achieved recently for \emph{search} algorithms, i.e., algorithms that produce a feasible solution $\bsigma^{\salg}$ 
but not a certificate of (near-)optimality. In particular, Subag \cite{subag2018following} developed an algorithm for the 
spherical mixed $p$-spin model in which the constraint
$\bsigma\in\Sigma_N$ is replaced by $\|\bsigma\|_2^2 = N$,
and proved that  it achieves any approximation factor $\rho = (1-\eps)$, $\eps>0$, provided $t\mapsto \xi''(t)^{-1/2}$ is concave.
In \cite{montanari2019optimization}, one of the authors developed an algorithm for the Sherrington-Kirkpatrick model, 
which corresponds to the quadratic case ($c_k=0$ for $k\ge 3$), with $\bsigma\in \Sigma_N$. Under a widely believed conjecture about the so-called Parisi formula,
 the algorithm of \cite{montanari2019optimization} also achieves a $(1-\eps)$-approximation for any $\eps>0$.

The main result of this paper is a characterization of the optimal value achieved by a class of low-complexity message passing algorithms 
that generalize the approach of \cite{montanari2019optimization}. As special cases, we recover the results of \cite{subag2018following} and
\cite{montanari2019optimization}. For a given approximation error $\eps>0$, the algorithm complexity is of the same order as evaluating
the gradient $\nabla H_N(\bx)$ at a constant number $C(\eps)$ of points.
 Its output  $\bsigma^{\salg}$ satisfies $H_N(\bsigma^{\salg})/N\ge(1-\eps)\cdot\OPT_N$ with high probability
whenever the corresponding Parisi formula satisfies a certain `no-overlap gap' condition.
Even more interestingly, we characterize the optimal value achieved by message passing algorithms in terms of an extended variational principle which generalizes
the Parisi formula. This points at a possible general picture for the optimal approximation ratio in ensembles of random optimization problems.

The random energy function $H_N$  has been studied for over forty years in statistical physics and probability theory, and is known 
as the Hamiltonian of the mixed $p$-spin model \cite{sherrington1975solvable,SpinGlass,TalagrandVolI,panchenko2013sherrington}. 
With the above definitions, it is easy to see that $\{H(\bsigma)\}_{\bsigma\in \Sigma_N}$ is a centered Gaussian process on the hypercube,
with covariance
\begin{align}
\E\big\{H_N(\bsigma)H_N(\bsigma')\big\} = N \xi\big(\langle\bsigma,\bsigma'\rangle/N\big)\, .
\end{align}

The asymptotic value of $\OPT_N$ was first derived by physicists using the non-rigorous replica method \cite{parisi1979infinite} and
subsequently established by Talagrand \cite{talagrand2006parisi} and
Panchenko \cite{panchenko2013parisi,panchenko2013sherrington}. 
This asymptotic value is characterized 
in terms of a variational principle known as the `Parisi formula.' While the Parisi formula allows to compute the asymptotic
free energy associated to the Hamiltonian $H_N$, it can be specialized to the zero temperature case, to compute the asymptotics 
of $\OPT_N$. The resulting characterization was established by
Auffinger and Chen in \cite{auffinger2017parisi} and it is useful to recall it for the reader's convenience.

Let $\cuU$ be the following subset of functions $\gamma:[0,1)\to \reals_{\ge 0}$:
\begin{align}
\cuU \equiv\Big\{\gamma:[0,1) \to\reals_{\ge 0}: \;\; \gamma\mbox{ non-decreasing }, \int_{0}^1\gamma(t)\,\de t<\infty\Big\} \, .
\end{align}
For $\gamma\in \cuU$, let $\Phi_{\gamma}:[0,1]\times \reals\to \reals$ be the solution of the following PDE, known as \emph{the Parisi PDE},
with terminal condition at $t=1$:
\begin{align}\label{eq:PDEFirst}
\begin{split}
\partial_t \Phi_{\gamma}(t,x)+\frac{1}{2}\xi''(t) \Big(\partial_x^2\Phi_{\gamma}(t,x)+\gamma(t) (\partial_x\Phi_{\gamma}(t,x))^2\Big) = 0\, ,\\
\Phi_{\gamma}(1,x) = |x|\, .
\end{split}
\end{align}
We refer to Section \ref{sec:ExtendedPrinciple} and Section \ref{sec:PropertiesVariational} for a construction of solutions of this
PDE. 

The Parisi functional $\Par:\cuU\to \reals$ is then defined by
\begin{align}
\Par(\gamma) \equiv \Phi_{\gamma}(0,0)-\frac{1}{2}\int_{0}^1 t\xi''(t)\gamma(t)\, \de t\, .
\end{align}
\begin{theorem}[\cite{auffinger2017parisi}]
The following limit holds almost surely
\begin{align}
\lim_{N\to\infty}\OPT_N = \inf_{\gamma\in \cuU}\Par(\gamma)\, .
\end{align}
\end{theorem}

The optimization problem on the right-hand side of the last formula is expected to be achieved\footnote{Existence has been proved in \cite{auffinger2017parisi}, but uniqueness is only known for positive temperature (finite $\beta$) \cite{auffinger2015parisi,jagannath2016dynamic}.} at a unique function $\gamma_P\in \cuU$
\cite{auffinger2017parisi}, which has a physical interpretation  \cite{SpinGlass}. Consider the (random) Boltzmann distribution 
$p_{\beta}(\bsigma)\propto \exp\{\beta H_N(\bsigma)\}$ at temperature $1/\beta$, and let $\bsigma_1,\bsigma_2\sim p_{\beta}$ be two
independent samples from this distribution, i.e., $(\bsigma_1,\bsigma_2)\sim \E p^{\otimes 2}_{\beta}$. 
Then $\beta^{-1}\gamma_P(t)$ is the asymptotic probability of the event $\{(\bsigma_1,\bsigma_2): |\<\bsigma_1,\bsigma_2\>|/N\le t\}$
(when the limit $\beta\to\infty$ is taken \emph{after} $N\to\infty$.) 
Given this interpretation, the non-decreasing constraint in the definition of $\cuU$ is very natural: it follows from $\gamma_P$ 
being the limit of a sequence of cumulative distribution functions (rescaled by the factor $\beta$). 

As mentioned above, in this paper we describe and analyze a class of algorithms that
aim at finding near-optima, i.e., configurations  $\bsigma^{\salg}\in\Sigma_N$ with $H_N(\bsigma^{\salg})/N$ as close as possible to
$\OPT_N$ (or to its asymptotic value $\inf_{\gamma\in \cuU}\Par(\gamma)$). Our main results can be summarized as follows:
\begin{enumerate}
\item If the infimum in the Parisi formula is achieved at $\gamma_P$ which is strictly increasing over the interval $[0,1)$,
then we provide an efficient algorithm that returns a $(1-\eps)$-optimizer.
This condition correspond to the `no-overlap gap' scenario mentioned above.
\item More generally, we introduce a new extended variational principle, which prescribes to minimize the Parisi functional
$\Par(\gamma)$ of a larger space $\cuL$ of functions $\gamma$ which are \emph{not} necessarily monotone. 
We present an algorithm that achieves $H(\bsigma^{\salg})/N\ge (1-\eps) \inf_{\gamma\in \cuL}\Par(\gamma)$, provided the infimum on the
right-hand side is achieved at some $\gamma_*\in \cuL$. Since $\cuU\subseteq \cuL$, this value is of course no larger
than the value of the global optimum. 

Moreover, under the `no-overlap gap' scenario,  we have $\inf_{\gamma\in \cuU}\Par(\gamma) = \inf_{\gamma\in \cuL}\Par(\gamma)$ and therefore we recover 
the result at the previous point.
\item We show, by a duality argument, that no algorithm in the class of message passing algorithms that we introduce 
can overcome the value $\inf_{\gamma\in \cuL}\Par(\gamma)$. This appears to be an interesting computational threshold,
whose importance warrants further exploration.
\end{enumerate}

\subsection{Further background}

Understanding the average case hardness of random computational problems is an outstanding challenge with numerous ramifications.
The use of spin glass concepts in this context has
a long history, which is impossible to review here. A few  pointers include
\cite{SpinGlass,monasson1999determining,mezard2002analytic,krzakala2007gibbs,MezardMontanari}.
Spin glass theory allows to derive a detailed picture of the structure of superlevel sets of random optimization problems,
or the corresponding Boltzmann distribution $p_{\beta}(\bsigma)\propto\exp\{\beta H_N(\bsigma)\}$.
A central challenge in this area is to understand the connection between this picture and computational tractability.
Which features of the energy landscape $H_N$ are connected to intractability?

Of course, the answer depends on the precise formulation of the question. In this paper we consider the specific problem
of achieving the best approximation factor $\rho$ so that a polynomial-time algorithm can output a feasible solution $\bsigma^{\salg}$
such that $H_N(\bsigma^{\salg})/N\ge \rho \, \OPT_N$ with high probability. This question was addressed in the physics literature from at least
two points of view:
\begin{itemize}
\item Significant effort has been devoted to computing
  the number (and energy) of local optima that are separated by large energy barriers: the energy of the most numerous such local optima is sometimes
  used as a proxy for the algorithmic threshold.
  The exponential growth-rate of the number of such optima is computed  using non-rigorous methods in
  \cite{crisanti2003complexity,crisanti2005complexity,parisi2006computing}.
\item An equally large amount of work was devoted to the study of Glauber or Langevin dynamics, which can be interpreted as greedy optimization algorithms.  In particular \cite{cugliandolo1994out,bouchaud1998out} and follow-up work study the $N\to\infty$ asymptotics of these dynamics, for a fixed time horizon.
\end{itemize}
These two approaches produced an impressive amount of (mostly non-rigorous) information. Despite these advances, no clear picture has been put forward
for the optimum approximation factor $\rho$ (the `algorithmic threshold'), except in particularly simple
cases, such as the pure $p$-spin spherical model. We refer to \cite{folena2019memories} for a recent illustration of the outstanding challenges.

Over the last two years, significant progress was achieved on this question. Apart from \cite{subag2018following,montanari2019optimization}
mentioned above, Addario-Berry and Maillard \cite{addario2018algorithmic} studied this question within the generalized random energy model,
which can be viewed as a stylized model for the energy landscape of mean field spin glasses. They prove that a variant of greedy search
achieves a $(1-\eps)$-approximation of $\OPT$ under a suitable variant of the no-overlap gap assumption.

In a different direction, Gamarnik and co-authors showed in several examples that the existence of an overlap gap
rules out a $(1-\eps)$-approximation for certain classes of polynomial time algorithms \cite{gamarnik2014limits,gamarnik2017performance,chen2019suboptimality}.
In particular, the recent paper \cite{gamarnik2019overlap} proves that approximate message passing algorithms (of the type studied in this paper)
cannot achieve a $(1-\eps)$-approximation of the optimum in pure $p$-spin Ising models, under the assumption that these exhibit an overlap gap.
However \cite{gamarnik2019overlap} does not characterize optimal approximation ratio, which
we instead do here, as a special case of our results.

Finally, two recent papers \cite{kunisky2019tight,mohanty2019lifting} study degree-$4$ sum-of-squares relaxations
for the Sherrington-Kirkpatrick model, and show that they fail at producing a tighter upper bound on $\OPT$ than
simple spectral methods. In conjunction with \cite{montanari2019optimization} these results suggest that
--in the context of spin glass problems-- computing a certifiable upper bound on $\OPT$ is fundamentally
harder than searching for an approximate optimizer.
  
Our approach is based on the construction and analysis of a class of approximate message passing (AMP) algorithms.
Following \cite{montanari2019optimization}, we refer to this family of algorithms as incremental approximate message passing (IAMP).
  AMP algorithms admit an exact asymptotic characterization in terms of a limiting Gaussian process, which is known
  as state evolution. This characterization was first established rigorously by  Bolthausen \cite{bolthausen2014iterative} for a special case,
and subsequently generalized in several papers \cite{BM-MPCS-2011,javanmard2013state,bayati2015universality, berthier2019state}. 
Here we will follow the proof scheme of \cite{berthier2019state} to generalize state evolution to the case of tensors.

\subsection{Notations}

We will typically use lower-case for scalars (e.g., $x,y,\dots$), bold lower-case for vectors (e.g., $\bx,\by,\dots$),
and bold upper case for matrices (e.g., $\bX,\bY,\cdots$). The ordinary scalar product in $\reals^d$ is denoted by
$\<\bx,\by\> = \sum_{i\le d}x_iy_i$, and the corresponding norm by $\|\bx\| = \<\bx,\bx\>^{1/2}$.
Given two vectors $\ba,\bb\in\reals^N$, we will often consider the normalized scalar product
$\<\ba,\bb\>_N=\sum_{i\le N}a_ib_i/N$, and the norm $\|\ba\|_N = \<\ba,\ba\>_N^{1/2}$.
There will be no confusion between this and $\ell_p$ norms which will be rarely used in $\reals^d$.

We will use standard notations for functional spaces, in particular spaces of differentiable functions (e.g. $C^k(\Omega)$,
$C^k_c(\Omega)$, and so on), and spaces of integrable functions (e.g., $L^p(\Omega)$). We refer --for
instance-- to \cite{evans2009partial} for definitions.

Given a sequence of random variables $(Y_n)_{n\ge1}$, and
$Y_{\infty}$, we write $Y_n\toP Y_{\infty}$, or $\plim_{n\to\infty} Y_n = Y_{\infty}$
if $Y_n$ converges in probability to $Y_{\infty}$.

For a function $f:\reals \to \reals$, we denote by $\|f\|_{\sTV(J)}$ the total variation of $f$ on the interval $J$: 
\begin{equation}\label{eq:total_var}
\|f\|_{\sTV(J)} := \sup_n\sup_{t_0<t_1<\dots<t_n, t_i\in J} \sum_{i=1}^n |f(t_i) - f(t_{i-1})|\, ,
\end{equation}
(i.e., the supremum is taken over all partitions of the interval $J$.)
We say that a function $\psi:\reals^d\to\reals$ is \emph{pseudo-Lipschitz} if there exists a constant $L<\infty$
such that, for every $\bx,\by\in\reals^d$, $|\psi(\bx)-\psi(\by)|\le L(1+\|\bx\|+\|\by\|)\|\bx-\by\|$.

Throughout the paper, we write that an event holds \emph{with high probability}, if its probability converges to one as
$N\to\infty$. We use $C$ to denote various constants, whose value can change from line to line.

\section{Achievability}

\subsection{Value achieved by message passing algorithms}

\begin{definition}\label{def:strong_tv}
We say that a function $f : [a,b] \times \reals \to \reals$ has bounded \emph{strong} total variation if there exists $C <\infty$ such that
\begin{equation}\label{eq:strong_tv} 
\sup_n\sup_{a\le t_0<\dots<t_n\le b} \sup_{x_1,\dots,x_n\in\reals} \sum_{i=1}^n |f(t_i,x_i) - f(t_{i-1},x_i)|\le C,
\end{equation}
(the supremum is over all partitions $(t_i)$ of the interval $[a,b]$ and all sequences $(x_i)$ in $\reals$). 
\end{definition}
\begin{assumption}\label{ass:uv}
Let $u, v:[0,1] \times \reals\to \reals$ be two measurable functions, with $u$ non-vanishing, and assume that the following 
holds for some constant $C<\infty$:
\begin{itemize}
\item[\bf{(A1)}] $u$ and $v$ are uniformly bounded: $\sup_{t,x} |u(t,x)| \vee |v(t,x)|\le C$.
\item[\bf{(A2)}] $u$ and $v$ are Lipschitz continuous in space, with uniform (in time) Lipschitz constant: $|u(t,x_1)-u(t,x_2)| \vee |v(t,x_1)-v(t,x_2)|\le C|x_1-x_2|$ for all $x_1,x_2\in\reals$ and $t\in[0,1]$. 
\item[\bf{(A3)}] $u(\, \cdot\, ,x)$ is continuous for all $x \in \reals$.
\item[\bf{(A4)}] $u$ and $v$ have bounded strong total variation. 
\end{itemize}
\end{assumption}
Consider the following stochastic differential equation
\begin{equation}
\label{eq:firstSDE}
\de X_t = v(t,X_t)\,\de t+\sqrt{\xi''(t)}\de B_t\, ,~~~ \mbox{with}~X_0 = 0\, ,\\
\end{equation}
where $(B_t)_{t \in [0,1]}$ is a standard Brownian motion. Under conditions {\bf(A1)} and {\bf(A2)} (pertaining to $v$), the above SDE has a unique strong solution which we denote by $(X_t)_{t\in [0,1]}$~\cite{oksendal2013stochastic}. We define the martingale 
\begin{equation}
\label{eq:magnetization}
M_t \equiv \int_0^t \sqrt{\xi''(s)}\,  u(s,X_s)\, \de B_s \, .
\end{equation}

Finally, it is useful to introduce a slight modification of the Hamiltonian \eqref{eq:hamiltonian}.
Namely, we denote by $\tH_N(\bsigma)$ the function that is obtained by restricting the sums in
$H_N(\bsigma)$ to sets of distinct indices $i_1,\dots,i_k$. (Notice that $\tH_N(\bsigma) = H_N(\bsigma) +o(N)$,
cf. Section \ref{sec:Rounding}.)
\begin{theorem}
\label{thm:value_message_passing}
Let Assumption \ref{ass:uv} hold, and further assume that  $M_{t_*}
\in [-1,1]$ almost surely  and $\E[M_t^2]=t$ for all $t\in[0,t_*]$, for some $t_*\in [0,1]$.

Further denote by $\chi$ the computational complexity of evaluating $\nabla H_N(\bm)$  at a point $\bm\in [-1,1]^N$,
and by $\chi_1$  the complexity of evaluating one coordinate of $\nabla\tH_N(\bm)$ at a point $\bm\in[-1,+1]^N$.

Then for any $\eps>0$ there  exists an randomized algorithm, with complexity $(C/\eps^2)\cdot(\chi+N)+N\chi_1$
which outputs $\bsigma^{\salg}\in\Sigma_N$ such that
\begin{align}
\frac{1}{N}H_N(\bsigma^{\salg}) \ge \int_{0}^{t_*}\xi''(t)\E \{u(t,X_t)\}\, \de t - \eps\, ,
\end{align}
with probability converging to one as $N\to\infty$.
\end{theorem}
The proof of this theorem is deferred to Section \ref{sec:PropertiesVariational}.

\begin{remark}
The stated complexity holds in a simplified model of computation whereby real sums and multiplications
have complexity of order one. However, we do not anticipate any difficulty to arise from passing to a finite model.

Typically, computing each gradient has complexity that is linear in the input size, and $N\chi_1$ is of the same order as $\chi$.
For instance,  if the coefficients $c_k$ vanish for $k> k_M$, it is easy to see that $\chi = O(N^{k_M})$,
and $\chi_1=O(N^{k_M-1})$. As a consequence, the dominant term in the complexity is $(C/\eps^2)\chi$.
In words, the algorithm's complexity is of the same order as computing the gradient of the cost function
$C/\eps^2$ times.
We further note that this constant $C$ depends on the regularity constants in Assumption \ref{ass:uv}.
\end{remark}

\begin{remark}\label{rmk:Spherical}
The proof of Theorem \ref{thm:value_message_passing} also establishes a similar result for the \emph{spherical} mixed $p$-spin 
model, where the constraint $\bsigma\in\Sigma_N$ is replaced by $\|\bsigma\|_2^2=N$. The same conclusion of the above theorem holds, whereby the
condition $M_{t_*} \in [-1,1]$ is no longer required.

In this case the choice of the functions $u$ and $v$ is straightforward.  Simply set $u(t,x) = \xi''(t)^{-1/2}$: since this is independent of $x$, the 
choice of $v$ is immaterial. The value achieved in this case is
\begin{align}
\frac{1}{N}H_N(\bsigma^{\sspher}) \ge \int_{0}^{1}\sqrt{\xi''(t)}\, \de t - \eps\, ,  \;\;\; \;\;\;\;\; \|\bsigma^{\sspher}\|_2^2=N\, .
\end{align}
In this case, we recover the energy achieved by the algorithm of Subag \cite{subag2018following}.
\end{remark}

\subsection{The extended variational principle}
\label{sec:ExtendedPrinciple}

For a function $\gamma:[0,1)\to \reals$, we write $\xi''\gamma$ for the pointwise multiplication of $\xi''$ and $\gamma$: $\xi''\gamma(t) = \xi''(t)\gamma(t)$.
We consider the extended space of order parameters
\begin{align}
\cuL \equiv \Big\{\gamma:[0,1)\to \reals_{\ge 0}: \;\; \|\xi''\gamma\|_{\sTV[0,t]}<\infty~ \forall t\in [0,1), \int_0^1\!\xi''\gamma(t)\,\de t < \infty\Big\} \, .\label{eq:LDef}
\end{align}
We metrize this space with the weighted $L^1$ metric
$\|\gamma_1-\gamma_2\|_{1,\xi''} := \|\xi''(\gamma_1-\gamma_2)\|_1 =\int_0^1\xi''(t)|\gamma_1(t)-\gamma_2(t)|\de t$, hence implicitly identifying $\gamma_1$ and $\gamma_2$ if they coincide for
almost every $t\in [0,1)$.
The notation $\| \cdot \|_{\sTV[0,t]}$ for total variation norm is defined in Eq.~\eqref{eq:total_var}.
It follows from the definition that, if   for $\gamma\in \cuL$, $\xi''\gamma(t) = \nu([0,t])$ where $\nu$ is a signed 
measure\footnote{This identification holds possibly apart from a set of values of $t$ of vanishing Lebesgue measure, which will be irrelevant here.}
of bounded total variation on intervals
$[0,1-\eps]$, $\eps>0$.

 It is obvious that the space $\cuL$ is a strict superset of $\cuU$: most crucially, it includes non-monotone functions.
As shown in Section \ref{sec:PropertiesVariational}, the Parisi functional 
$\gamma\mapsto \Par(\gamma)$ can be defined on this larger space. 
\begin{theorem}\label{thm:VarPrinciple}
Assume that the infimum $\inf_{\gamma\in \cuL} \Par(\gamma)$ is achieved at a function $\gamma_*\in\cuL$. 
Further denote by $\chi$ the computational complexity of evaluating $\nabla H_N(\bm)$  at a point $\bm\in [-1,1]^N$,
and by $\chi_1$  the complexity of evaluating one coordinate of $\nabla\tH_N(\bm)$ at a point $\bm\in[-1,+1]^N$.

Then for every $\eps>0$ there exists an algorithm with complexity at most $C(\eps)\cdot (\chi+N)+N\chi_1$ which outputs
$\bsigma^{\salg}\in\Sigma_N$ such that
\begin{align}
\frac{1}{N}H_N(\bsigma^{\salg}) \ge \inf_{\gamma\in\cuL}\Par(\gamma)- \eps\, ,
\end{align}
with probability converging to one as $N\to\infty$.
\end{theorem}

As an important consequence of Theorem \ref{thm:VarPrinciple} we obtain a $(1-\eps)$-approximation of the optimum 
whenever $\inf_{\gamma\in \cuU} \Par(\gamma)$ is achieved on a strictly increasing function. For future reference,
we introduce the following `no-overlap gap' assumption.
\begin{assumption}[No overlap gap at zero temperature]
A mixed $p$-spin model with mixture $\xi$ is said to satisfy the no-overlap gap assumption at zero-temperature
if there exists $\gamma_*\in\cuU$ strictly increasing in $[0,1)$ such that  $\Par(\gamma_*) = \inf_{\gamma\in \cuU} \Par(\gamma)$.
\end{assumption}
The no-overlap gap assumption is expected to hold for some choices of the mixture $\xi$ but not for others.
In particular, it is believed to hold for the Sherrington-Kirkpatrick model, which corresponds to the special case $\xi(t) = c_2^2t^2$, 
but not for the pure $p$-spin model, i.e., $\xi(t) = c_p^2t^p$, $p\ge 3$.
\begin{corollary}\label{coro:NoGap}
Assume the no-overlap gap assumption to hold for the mixture $\xi$. 
Then for every $\eps>0$ there exists an algorithm with the same complexity as in Theorem \ref{thm:VarPrinciple} which outputs
$\bsigma^{\salg}\in \Sigma_N$ such that
\begin{align}
\frac{1}{N}H_N(\bsigma^{\salg}) \ge \OPT_N- \eps\, ,
\end{align}
with probability converging to one as $N\to\infty$.
\end{corollary}

\begin{remark}
Continuing from Remark \ref{rmk:Spherical}, Theorem  \ref{thm:VarPrinciple} has an analogue for the spherical model $\|\bsigma\|_2^2=N$. In this case, the 
variational principle takes a more explicit form \cite{crisanti1992sphericalp,chen2013aizenman}:
\begin{align}
\Par^{\sspher}(\gamma) = \frac{1}{2}\int_0^1\left(\xi''(t)\Gamma(t) +\frac{1}{\Gamma(t)}\right)\de t\,,\;\;\;\;\;\;\;\;\;
\Gamma(t) \equiv \int_t^1 \gamma(s)\de s\,  .
\end{align}
A simple calculation shows that this is minimized in $\cuL$ at $\gamma_*(t) = -\frac{\de\phantom{t}}{\de t}(\xi''(t)^{-1/2})$. This leads
to the optimal value $\Par^{\sspher}(\gamma_*) = \int_0^1\sqrt{\xi''(t)}\de t$, which we anticipated  in Remark \ref{rmk:Spherical}.
The condition for the minimizer to be in $\cuU$, $\gamma_*\in\cuU$, coincides with the condition that $t\mapsto \xi''(t)^{-1/2}$ is concave.
This is the condition for no-overlap gap in the spherical model, and is also the condition under which the algorithm of  \cite{subag2018following}
achieves a $(1-\eps)$-optimum.
\end{remark}

\section{Message passing algorithms}
\label{sec:message_passing}

In this section we introduce a general class of message passing algorithms that we use to prove Theorem \ref{thm:value_message_passing} and Theorem \ref{thm:VarPrinciple}.
These are generalizations of the algorithm introduced in \cite{montanari2019optimization} for the Sherrington-Kirkpatrick model $\xi(t) = c_2^2 t^2$.

\subsection{The general iteration}
 
For each $\ell \ge 0$, let $f_{\ell} : \reals^{\ell+1} \to \reals$ be
a real-valued Lipschitz  function, and let $f_{-1}  \equiv 0$. For a sequence of vectors
$\z^0,\cdots,\z^\ell \in \reals^N$ we use the notation
$f_{\ell}(\z^0,\cdots,\z^\ell)$ for the vector
$(f_{\ell}(z_i^0,\cdots,z_i^\ell))_{1\le i \le N}$. For a vector $\u
\in \reals^N$ we denote by $ \bW\{\u\}$ the vector $\v \in \reals^N$
with coordinates 
\[v_i = \frac{1}{(p-1)!}\sum_{1\le i_1,\cdots,i_{p-1}\le N}W_{i,i_1,\cdots,i_{p-1}}u_{i_1}\cdots u_{i_{p-1}}.\]     
We let $\langle \u\rangle_N := \frac{1}{N}\sum_{i=1}^N u_i$ and $\langle \u, \v \rangle_N := \frac{1}{N}\sum_{i=1}^N u_i v_i$. Finally, we write $\f_{\ell}$ as shorthand for the vector $f_{\ell}(\z^0,\cdots,\z^\ell) \in \reals^N$. 

Before introducing the actual message passing algorithm, we need define a Gaussian process that captures its asymptotic behavior
as $N\to\infty$. Let $p_0$ be a probability distribution on $\reals$ and let $Z^0 \sim p_0$. For each $\ell\in\integers$, let
$(Z^1,\cdots,Z^\ell)$ be a centered Gaussian vector independent of $Z^0$ with covariance $Q_{j,k} = \E[Z^jZ^k]$ defined recursively by
\begin{align}\label{eq:cov_amp}
\begin{split}
Q_{j+1,k+1} &= \xi'\Big(\E\big[f_{j}(Z^0,\cdots,Z^j)f_{k}(Z^0,\cdots,Z^k) \big]\Big),~~~ \ell,j\ge0.
\end{split}
\end{align}
The message passing algorithm starts with $\z^0$ with coordinates drawn  i.i.d.\ with distribution $p_0$ independently of everything else.       
The general message passing iteration takes the form
\begin{align}\label{eq:general_amp}
\begin{split}
\z^{\ell+1} &= \sum_{p=2}^{\infty} c_p \bW^{(p)}\{f_{\ell}(\z^0,\cdots,\z^\ell)\} - \sum_{j=0}^\ell d_{\ell, j} f_{j-1}(\z^0,\cdots,\z^{j-1}),\\  
d_{\ell,j} &= \xi''\big( \E\big[f_{\ell}(Z^0,\dots,Z^{\ell})f_{j-1}(Z^0,\dots,Z^{j-1}\big]\big) \cdot \E\Big[\frac{\partial f_{\ell}}{\partial z^j}(Z^0,\cdots,Z^\ell)\Big]\, .
\end{split}
\end{align}  

Note that the first term in the update equation is the gradient of $H_N$ at the point $f_{\ell}(\z^0,\cdots,\z^\ell)$. The joint distribution for the first $\ell$ iterates of Eq.~\eqref{eq:general_amp} can be exactly characterized in the $N\to \infty$ limit.
\begin{proposition}[State evolution]\label{prop:state_evolution}  
Assume that $p_0$ has finite second moment and let $\psi : \reals^{\ell+1} \to \reals$ be a pseudo-Lipschitz function. Then
\[\Big\langle \psi\big(\z^0,\cdots,\z^\ell\big)\Big\rangle_N ~\xrightarrow[N\to \infty]{p} ~\E \big[\psi\big(Z^0,\cdots,Z^\ell\big)\big]\, .\]
\end{proposition}
This characterization is known as state evolution \cite{bolthausen2014iterative,BM-MPCS-2011,javanmard2013state,bayati2015universality, berthier2019state}. 
The proof of Proposition~\ref{prop:state_evolution}  follows from the same technique introduced in  \cite{berthier2019state}, and we present
it in Appendix \ref{sec:ProofSE}. We note in passing that a version of this result was announced in \cite{richard2014statistical} without proof; the proof in
Appendix \ref{sec:ProofSE} fills this gap. 

\subsection{Choice of the non-linearities}

We choose the sequence of functions $f_{\ell}$ in a specific way.  Let $u, v:[0,1] \times \reals \to \reals$ be two functions satisfying the conditions of Assumption~\ref{ass:uv}.
Given $z^0,\cdots,z^\ell \in \reals$ we consider the finite difference equation
 \begin{equation}\label{eq:discrete_cavity_field}
x^{j+1} - x^{j} = v(j\delta ; x^{j}) \delta + (z^{j+1}-z^{j}),~~ 0\le j \le \ell-1, \quad \mbox{with}~x^0=0, 
\end{equation}
with driving `noise' $z^0,\cdots,z^\ell$, drift $v$ and `step size' $\delta>0$. 
 This is meant to be a discretization of the SDE~\eqref{eq:firstSDE}, provided that the sequence $z^0,\cdots,z^\ell$ `behaves' like Brownian motion.
We further let the discrete analogue of the martingale $M_t$, Eq.~\eqref{eq:magnetization}, be 
 \begin{equation}\label{eq:discrete_cavity_magnetization}
 m^{\ell} \equiv m^0 + \sum_{j=0}^{\ell-1} u_j^{\delta}(x^{j})(z^{j+1}-z^{j}),~~\mbox{for}~\ell \ge 1~ \mbox{and}~m^{0} = \sqrt{\delta},
 \end{equation}
 where $u^{\delta}_j(x) = a_j u(j \delta; x)$ with $a_j$ a bounded rescaling which will be defined in~Eq.~\eqref{eq:rescaling_u} below.
 
  Note that $x^{\ell}$ is a function of $z^0,\cdots,z^\ell$ and so is $m^{\ell}$. We define the non-linearity $f_{\ell}$ as the function mapping $z^0,\cdots,z^\ell$ to $m^{\ell}$:
 \begin{equation}
 \label{eq:non_linearity}
 f_{\ell} : (z^0,\cdots,z^\ell) \longmapsto m^{\ell} ~~~\mbox{as per Eq.~\eqref{eq:discrete_cavity_field} and Eq.~\eqref{eq:discrete_cavity_magnetization}}.
 \end{equation}

 The algorithm is completely specified by defining the functions $u,v:[0,1]\times \reals \to \reals$. For any choice of such functions,
 Theorem \ref{thm:value_message_passing} predicts what is the value achieved by the algorithm (for small $\delta$). Theorem
 \ref{thm:VarPrinciple} corresponds to a specific choice of these functions. Namely, if $\gamma_*$ minimizes
the Parisi functional over $\cuL$ (i.e. $\Par(\gamma_*)= \inf_{\gamma\in\cuL}\Par(\gamma)$), we let
$\Phi_{\gamma_*}:[0,1]\times\reals\to\reals$ denote the corresponding solution of the PDE \eqref{eq:PDEFirst}.
We  let
\begin{align}\label{eq:choice_uv2}
  v(t,x) = \xi''(t)\gamma_*(t)\partial_x\Phi_{\gamma_*}(t,x)\,,\;\;\;\;\;\; u(t,x) = \partial_x^2\Phi_{\gamma_*}(t,x)\, ,\;\;\; t\in[0,t_*]\, ,
\end{align}
and extend  them as to satisfy the assumptions of Theorem \ref{thm:value_message_passing} for $t\in (t_*,1]$. 
Theorem \label{thm:VarPrinciple} is proved by letting $t_*=t_*(\eps)\to 1$ as $\eps\to 0$.
 We prove in the next section that this choice is optimal: no pair of functions satisfying the hypotheses of 
Theorem \ref{thm:value_message_passing} can achieve a value larger than $\inf_{\gamma\in\cuL}\Par(\gamma)$.

\section{Optimality and stochastic control}
\label{sec:optimality}
In this section we show that the value given by the extended variational principle of Subsection~\ref{sec:ExtendedPrinciple} is the largest achievable by message passing algorithms of the form considered above.     
\begin{theorem}\label{thm:stoch_control}
For $u,v : [0,1]\times \reals \to \reals$ satisfying conditions of Assumption~\ref{ass:uv}, let $\m^{\ell} = f_{\ell}(\z^0,\cdots,\z^\ell)$ be the output of the message passing algorithm~\eqref{eq:general_amp} with non-linearity given by~\eqref{eq:non_linearity}. Then
\[ \lim_{\delta \to 0^+} \, \plim_{N\to \infty}  \frac{H_{N}\big(\m^{\lfloor \delta^{-1}\rfloor}\big)}{N}  \le  \inf_{\gamma \in \cuL} \Par(\gamma).\]
\end{theorem}

The proof of this theorem is deferred to Section
\ref{sec:ProofControl}. Here we outline the basic strategy which formulates the optimality question as a stochastic optimal control problem.

We will prove in Proposition~\ref{prop:value} below that the left-hand side in the above inequality is equal to
\begin{equation}\label{eq:obj_value}
\cuE(u,v) \equiv \int_0^1 \xi''(t) \E\big[u(t,X_t)\big] \de t,
\end{equation}
where $(X_t)$ solves the SDE~\eqref{eq:firstSDE}.
 We will analyze the variational problem consisting in maximizing the objective value~\eqref{eq:obj_value} given the constraints $\E[M_t^2] = t$ for all $t \in [0,1]$ and $M_1 \in (-1,1)$ over $u$ and $v$ satisfying Assumption~\ref{ass:uv}. (We recall that $M_t = \int_0^t \sqrt{\xi''(s)} u(s,X_s) \de B_s$.)
   
For $s \le t$, we define the space of admissible controls $D[s,t]$ on the interval $[s,t]$ as the collection of all stochastic processes $(u_r)_{r \in [s,t]}$  which are progressively measurable with respect to the filtration of the Brownian motion $(B_r)_{r \in [s,t]}$ and such that
\[\E\int_s^t \xi''(r) u_r^2 \,\de r \, < +\infty.\]
We are then led to consider the stochastic control problem 
\begin{align}\label{eq:stochastic_control}
\begin{split}
\VAL ~\equiv ~\sup_{u \in D[0,1]}~~& \E\Big[\int_0^1 \xi''(s) u_s \de s \Big]\\
\mbox{s.t.} ~~&~ \E[(M^{u}_t)^2] = t~~ \forall t \in [0,1], ~\mbox{and}~ M^{u}_1 \in (-1,1) ~~\mbox{a.s.,}
\end{split}
\end{align}
with $M^{u}_t := \int_0^t \sqrt{\xi''(s)} u_s \de B_s$.

Note that $D[0,1]$ is a larger space of controls than the one arising from the original algorithm, cf.
Eqs.~\eqref{eq:firstSDE}, \eqref{eq:magnetization}. Indeed, for any choice of the drift $v$,  the process $(u(t,X_t))_{t\in [0,1]}$ is in $D[0,1]$,
and hence can be encoded in the choice of a stochastic process $(u_t)_{t\in [0,1]} \in D[0,1]$.
The proof of Theorem~\ref{thm:stoch_control} consists in showing $\VAL \le \inf_{\gamma \in \cuL} \Par(\gamma)$. We achieve this by writing the Lagrangian form of the above constrained optimization problem with respect to the equality constraint $\E[(M^{u}_t)^2]=t$ for all $t$. 
We define the space of piecewise constant, or simple, functions:
\begin{align}\label{eq:simple_functions}
\SF_+\equiv\Big\{g=\sum_{i=1}^ma_i\ind_{[t_{i-1},t_i)}:\;\; 0=t_0<t_1<\cdots<t_m=1, a_i\in\reals_{\ge 0}, m\in\naturals\Big\}\, .
\end{align}

Let $\gamma: [0,1] \to \reals_+$, $\gamma \in \SF_+$ (defined by continuity at $t=1$). We claim that the following upper bound holds: 
\begin{equation}\label{eq:relaxed_control}
\VAL \le \VF_{\gamma}(0,0), 
\end{equation}
where $\VF_{\gamma}: [0,1] \times (-1,1) \to \reals$ is defined by
\begin{align}\label{eq:valueBellman}
\begin{split}
\VF_{\gamma}(t,z) := \sup_{u\in D[t,1]}~~& \E\left[ \int_t^1 \xi''(s) u_s \de s  + \frac{1}{2}  \int_t^1 \nu(s) \big( \xi''(s) u_s^2 - 1\big) \de s\right],\\
\mbox{s.t.} ~~&~ z+\int_t^1 \sqrt{\xi''(s)}u_s \de B_s  \in (-1,1) ~~\mbox{a.s.},
\end{split}
\end{align}
where $\nu(t) := \int_t^1 \xi''(s)\gamma(s) \de s$.
Indeed, we have by integration by parts, 
\[\int_0^1 \nu(s) \big( \xi''(s) u_s^2 - 1\big) \de s =  \int_0^1 \xi''(t) \gamma(t)\Big( \int_0^t \xi''(s) u_s^2 \de s - t\Big) \de t.\]
Since $\E[(M^{u}_t)^2] = \E\int_0^t \xi''(s) u_s^2 \de s$, the second term in the definition of $\VF(0,0)$ Eq.~\eqref{eq:valueBellman}
vanishes for any control $(u_s)$ that satisfies the constraints of the problem \eqref{eq:stochastic_control}, thus proving Eq.~\eqref{eq:relaxed_control}.
In  other words, $\VF_{\gamma}(0,0)$ is the Lagrangian associated to the optimization problem~\eqref{eq:stochastic_control} with dual variable $\frac{1}{2}\xi''\gamma$.  

We are now left with the task of relating the value function $\VF$ to the Parisi functional $\Par(\gamma)$: 
\begin{proposition}\label{prop:value_at_zero}
For $\gamma \in \SF_+$, $\VF_{\gamma}(0,0) = \Par(\gamma)$.
\end{proposition}

 The bound~\eqref{eq:relaxed_control} then implies that 
 \[\VAL \le \inf_{\gamma \in \SF_+} \Par(\gamma),\]
 Since any function in the class $\cuL$ can be approximated with a piecewise constant function with respect to the $L^1$ norm, and $\gamma \mapsto \Par(\gamma)$ is continuous in this norm (see Section~\ref{sec:PropertiesVariational}), the above infimum is no larger than $\inf_{\gamma \in \cuL} \Par(\gamma)$. 
 
We now sketch the first steps in establishing Proposition~\ref{prop:value_at_zero}, relegating a full proof to Section~\ref{sec:ProofControl}. The value function~\eqref{eq:valueBellman} can be (formally) computed by dynamic programming where we search for solutions to the equation 
\begin{equation}\label{eq:bellman}
V(t,z) = \sup_{u \in D[t,\theta]} \E\left[ \int_t^\theta \xi''(s) u_s \de s  + \frac{1}{2}  \int_t^\theta \nu(s) \big(\xi''(s) u_s^2 - 1\big) \de s + V\Big(\theta, z+ \int_t^\theta \sqrt{\xi''(s)}u_s \de B_s\Big)\right],
\end{equation}
valid for all $\theta \in [t,1]$ and $z \in (-1,1)$, with terminal condition $V(1,z) = 0$ for $|z|< 1$. 
 The associated Hamilton-Jacobi-Bellman (HJB) equation, which can be formally obtained from~\eqref{eq:bellman} by letting $\theta \to t^+$ and applying It\^o's formula, is 
\begin{align}\label{eq:HJB}
\begin{split}
\partial_tV(t,z) + \xi''(t) \sup_{\lambda \in \reals}\Big\{  \lambda + \frac{\lambda^2}{2} \big(\nu(t)+\partial_z^2V(t,z)\big) \Big\} - \frac{1}{2} \nu(t) &= 0,~~~ (t,z) \in [0,1)\times (-1,1),\\
V(1,z) &= 0,~~~ z \in (-1,1).
\end{split}
\end{align}
Note that it is a priori unclear whether Eq.~\eqref{eq:bellman} and Eq.~\eqref{eq:HJB} have (classical) solutions and whether they are at all related Eq.~\eqref{eq:valueBellman}: $\VF_{\gamma}$ is not known a priori to be smooth, hence the above derivation is not rigorously justified; it is not clear that the right-hand side of~\eqref{eq:bellman} is even measurable.    
To circumvent this issue, we will ``guess" a solution $V$  to~\eqref{eq:HJB} and use the so-called ``verification argument" to certify that the guessed solution is equal to $\VF_{\gamma}$ as defined in Eq.~\eqref{eq:valueBellman}. En route, we establish that the optimal control process in the stochastic control problem~\eqref{eq:stochastic_control} is given by
\[u^*_t =  \partial_{x}^2\Phi_{\gamma}(t,X_t),\]
where $(X_t)$ solves the SDE~\eqref{eq:firstSDE} with drift $v(t,x) = \xi''(t)\gamma(t) \partial_x\Phi_{\gamma}(t,x)$ and $\Phi_{\gamma}$ solves the Parisi PDE. This confirms in hindsight our choice of the functions $u$ and $v$ used in the message passing algorithm, Eq.~\eqref{eq:choice_uv2}. (See also proof of Theorem~\ref{thm:VarPrinciple}.)

\section{Proof of Theorem \ref{thm:value_message_passing}}

\subsection{The scaling limit} 
Consider the message passing iteration~\eqref{eq:general_amp} with non-linearities $f_{\ell}$ given by~\eqref{eq:non_linearity} and iterate sequence $(\z^0,\z^1\cdots)$ starting from $\z^0 = \bf{0}$. We denote by $(\x^0,\x^1\cdots)$ and  $(\m^0,\m^1\cdots)$ the  two auxiliary sequences obtained from the finite difference equation~\eqref{eq:discrete_cavity_field} and the relation~\eqref{eq:discrete_cavity_magnetization}, respectively. (All maps are applied independently to every coordinate $i \in [N]$.) 
It is clear from Eq.~\eqref{eq:discrete_cavity_magnetization} that $f_{\ell}$ is Lipschitz continuous for each $\ell$, with a Lipschitz constant depending on $\ell$ and $C$ (the uniform bound on $u$), and therefore the conclusion of Proposition~\ref{prop:state_evolution} applies. 
Let $(Z^{\delta}_{\ell})_{\ell\ge 0}$ be the limit of the sequence $(\z^0,\z^1\cdots)$. Since $u,v$ are uniformly Lipschitz in $x$, then $(\x^0,\x^1\cdots)$ and  $(\m^0,\m^1\cdots)$ converge as well in the sense of Proposition~\ref{prop:state_evolution} to stochastic processes $(X^{\delta}_{\ell})_{\ell\ge 0 }$ and $(M^{\delta}_{\ell})_{\ell\ge 0 }$, defined respectively via the formulas~\eqref{eq:discrete_cavity_field} and~\eqref{eq:discrete_cavity_magnetization} by replacing every occurrence of $z^j$ by $Z^\delta_{j}$. 
Define for all $\ell \ge 0$, 
\[q^{\delta}_{\ell} \equiv \E\big[(M^{\delta}_{\ell})^2\big].\]
\begin{lemma}\label{lem:discrete_martingale}
The sequence $(Z^{\delta}_{\ell})_{\ell\ge 0}$ is a Gaussian process starting at $Z^{\delta}_0 =0$. Its increments $\Delta^{\delta}_\ell := Z^{\delta}_{\ell}-Z^{\delta}_{\ell-1}$ are independent, have zero mean and variance
\begin{align*}
\E\big[(\Delta^{\delta}_1)^2\big] &= \xi'(\delta),\\
\E\big[(\Delta^{\delta}_\ell)^2\big] &= \xi'(q^{\delta}_{\ell-1}) - \xi'(q^{\delta}_{\ell-2})~~~\mbox{for all}~\ell\ge 2.
\end{align*}
Furthermore, $(M^{\delta}_{\ell})_{\ell\ge 0}$ is a martingale with respect to the filtration $\big(\cF_\ell = \sigma(Z^{\delta}_0,\dots,Z^{\delta}_{\ell})\big)_{\ell \ge 0}$, and $M^{\delta}_0 = \sqrt{\delta}$.
\end{lemma}
\begin{proof}
We proceed by induction. Since $\z^0 = \mathbf{0}$ and $\m^{0} = \sqrt{\delta}\mathbf{1}$, we have $Z^{\delta}_0 = 0$ and $M^{\delta}_0 = \sqrt{\delta}$. We also have for all $j\ge 1$, $\E[Z^{\delta}_1Z^{\delta}_j] = \xi'(\E[M^{\delta}_0 M^{\delta}_{j-1}]) = \xi'(\delta)$. So $\E[\Delta^{\delta}_1\Delta^{\delta}_2] = \E[Z^{\delta}_2 Z^{\delta}_1] - \E[(Z^{\delta}_1)^2] = 0$, and $\E[(\Delta^{\delta}_1)^2] = \xi'(\delta)$.
Now we assume that the increments $(\Delta^{\delta}_j)_{j \le \ell}$ are independent. This implies that $(M^{\delta}_{j})_{j\le \ell}$ is a martingale. Appealing to the state evolution recursion, 
 \begin{align*}
 \E[\Delta^{\delta}_{\ell+1}\Delta^{\delta}_1] &= \E[Z^{\delta}_{\ell+1}Z^{\delta}_1] - \E[Z^{\delta}_{\ell}Z^{\delta}_1] \\
 &= \xi'\big(\E[M^{\delta}_{\ell}M^{\delta}_{0}]\big) - \xi'\big(\E[M^{\delta}_{\ell-1}M^{\delta}_{0}]\big)\\
 &=0,
  \end{align*}
 since $M^{\delta}_{0}=\sqrt{\delta}$ and $\E[M^{\delta}_{\ell}] = \E[M^{\delta}_{\ell-1}]$. For $ 2\le j\le \ell$,
 \begin{align*}
 \E[\Delta^{\delta}_{\ell+1}\Delta^{\delta}_j] &= \xi'\big(\E[M^{\delta}_{\ell}M^{\delta}_{j-1}]\big) - \xi'\big(\E[M^{\delta}_{\ell-1}M^{\delta}_{j-1}]\big) - \xi'\big(\E[M^{\delta}_{\ell}M^{\delta}_{j-2}]\big)+\xi'\big(\E[M^{\delta}_{\ell-1}M^{\delta}_{j-2}]\big)\\
 &= 0
 \end{align*}
 since $(M^{\delta}_{j})_{j\le \ell}$ has independent increments. So $\Delta^{\delta}_{\ell+1}$ is independent from $(\Delta^{\delta}_j)_{j \le \ell}$. This ends the induction argument. The variance identity follows straightforwardly. 
\end{proof}
We define the functions $u^{\delta}$ by the relations 
 \begin{align}\label{eq:rescaling_u}
 \begin{split}
 u^{\delta}_0 \equiv \Big(\frac{\delta}{\xi'(\delta)}\Big)^{1/2}, &\qquad
 u_\ell^{\delta}(x) \equiv \frac{u(\ell \delta; x)}{\Sigma^{\delta}_\ell} ~\mbox{for all}~ \ell\ge 1,\\
\mbox{with}~~~ (\Sigma^{\delta}_\ell)^2 &= \delta^{-1}\big(\xi'(q^{\delta}_{\ell}) - \xi'(q^{\delta}_{\ell-1})\big) \E\big[u(\ell \delta; X^{\delta}_{\ell})^2\big].
 \end{split}
 \end{align}
 \begin{lemma}\label{lem:reparametrization}
Assume $u_{\ell}^{\delta}$ takes the form~\eqref{eq:rescaling_u} for all $\ell \ge 0$. Then $q^{\delta}_{\ell} = (\ell+1)\delta$ for all $\ell \ge 0$. 
 \end{lemma}  
\begin{proof}
First notice that $u_\ell^{\delta}$ is well defined since $\Sigma^{\delta}_\ell>0$ for all $\ell$. This can be easily shown by induction since
$\xi'$ is strictly increasing and, by the induction hypothesis $q^{\ell}_{\delta}>q^{\ell-1}_{\delta}$, and $X^{\delta}_{\ell}$ is a non-degenerate Gaussian, whence 
$\E\big[u(\ell \delta; X^{\delta}_{\ell})^2\big]>0$ (because by assumption $u$ is non-vanishing).
We have $q^{\delta}_0 = \E[(M_0^{\delta})^2] = \delta$. Let $\ell \ge 1$. Since $Z^{\delta}$ has independent increments, Eq.~\eqref{eq:discrete_cavity_magnetization} implies 
\begin{align*}
\E\big[(M^{\delta}_{\ell}-M^{\delta}_{0})^2\big] &= \sum_{j=0}^{\ell-1} \E\big[u_j^{\delta}(X^{\delta}_{j})^2\big] \cdot \E\big[(\Delta^{\delta}_{j+1})^2\big]\\
&=  \E\big[u_0^{\delta}(X^{\delta}_{0})^2\big] \cdot \xi'(\delta) + \sum_{j=1}^{\ell-1} \E\big[u_j^{\delta}(X^{\delta}_{j})^2\big] \cdot \big(\xi'(q^{\delta}_{j}) - \xi'(q^{\delta}_{j-1})\big)\\
& = \delta + (\ell-1)\delta.
\end{align*}
The second line follows from Lemma~\ref{lem:discrete_martingale}, the last line follows from~\eqref{eq:rescaling_u}. The fact that $M^{\delta}$ is a martingale yields the desired result. 
\end{proof}
Next, we show that under condition~\eqref{eq:non-linear},  $(Z^{\delta}_{j}, X^{\delta}_{j}, M^{\delta}_{j})_{0\le j\le \ell}$ converge to continuous-time stochastic processes $(Z_t, X_t, M_t)_{t \in [0,1]}$ on the interval $[0,1]$ as $\delta \to 0, \ell \to \infty$ and $\ell \le \delta^{-1}$, with $Z_t \equiv \int_0^t \sqrt{\xi''(s)}\de B_s$, $X_t$ is the solution to the SDE~\eqref{eq:firstSDE}  and $M_t \equiv \int_0^t \sqrt{\xi''(s)}u(s,X_s)\de B_s$.
\begin{proposition}\label{prop:convergence}
Assume
\begin{equation}\label{eq:non-linear}
\E\big[M_t^2\big] = t ~~~\mbox{for all } t \in [0,1].
\end{equation}
Then there exists a coupling between the random variables $\{(Z^{\delta}_{\ell}, X^{\delta}_{\ell}, M^{\delta}_{\ell})\}_{\ell \ge 0}$ and the stochastic process 
$\{(Z_{t}, X_{t}, M_{t})\}_{t \ge 0}$ 
such that the following holds. There exists $\delta_0 >0$ and a constant $C>0$ such that for all $\delta \le \delta_0$ and $\ell \le \delta^{-1}$,
\begin{align}
\max_{1 \le j \le \ell} \E \big[|X^{\delta}_j - X_{\delta j}|^{2} \big] &\le C \delta,\label{eq:error-X}\\
\max_{1 \le j \le \ell} \E \big[|M^{\delta}_j - M_{\delta j}|^{2} \big] &\le C \delta.\label{eq:error-M}
\end{align}
\end{proposition}
\begin{proof}
Let $(B_t)_{t \in [0,1]}$ be a standard Brownian motion. We couple the increments of $Z^{\delta}$ with $(B_t)$ via the relation  
\begin{equation}\label{eq:coupling}
Z^{\delta}_{\ell}  - Z^{\delta}_{\ell-1} = \int_{\delta(\ell-1)}^{\delta \ell} \sqrt{\xi''(s)}\de B_s ~~\mbox{for all}~\ell \ge 1.
\end{equation}
It\^o's isometry implies $\E\big[(Z^{\delta}_{\ell}  - Z^{\delta}_{\ell-1})^2\big] = \xi'(\delta \ell) - \xi'(\delta(\ell-1))$. By Lemma~\ref{lem:reparametrization}, this is in accordance with the characterization of the law of $Z^{\delta}$ obtained in Lemma~\ref{lem:discrete_martingale}. Moreover we have $Z^{\delta}_{\ell} = Z_{\delta \ell}$ for all $\ell\ge 0$. 
We now show~\eqref{eq:error-X}. Let $\Delta^X_{j} = X^{\delta}_j - X_{\delta j}$. Using~\eqref{eq:firstSDE} and~\eqref{eq:discrete_cavity_field} we have
 \begin{align*}
 \Delta^X_{j}  - \Delta^X_{j-1} &= \int_{(j-1)\delta}^{j \delta} \big(v((j-1)\delta; X^{\delta}_{j}) - v(t; X_{t})\big) \de t + Z^{\delta}_{j}  - Z^{\delta}_{j-1} - \int_{\delta(j-1)}^{\delta j} \sqrt{\xi''(s)}\de B_s\\
 &=  \int_{(j-1)\delta}^{j \delta} \big(v((j-1)\delta; X^{\delta}_{j}) - v(t; X_{t})\big) \de t\\
 &=   \int_{(j-1)\delta}^{j \delta} \big(v((j-1)\delta; X^{\delta}_{j}) - v((j-1)\delta; X_{t})\big) \de t
 + \int_{(j-1)\delta}^{j \delta} \big(v((j-1)\delta; X_{t}) - v(t; X_{t})\big) \de t.
 \end{align*}
The first term is the above equation is bounded in absolute value by $C \int_{(j-1)\delta}^{j \delta}|X_j^{\delta} - X_t|\de t$ since $v$ Lipschitz in space uniformly in time. As for the second term, 
\begin{align*} 
\sum_{k=1}^\ell &\int_{(k-1)\delta}^{k \delta} \big|v((k-1)\delta; X_{t}) - v(t; X_{t})\big| \de t \\
&\le \sum_{k=1}^\ell \int_{(k-1)\delta}^{k \delta} \Big\{ \big|v((k-1)\delta; X_{t}) - v(t; X_{t})\big| +\big|v(t; X_{t})-v(k\delta; X_{t})\big|\Big\} \de t\\
&\le \delta \sum_{k=1}^\ell \sup_{(k-1)\delta \le t \le k \delta}  \Big\{\big|v((k-1)\delta; X_{t}) - v(t; X_{t})\big| + \big|v(t; X_{t})-v(k\delta; X_{t})\big|\Big\}\\
&\le \delta \sup_{t_1,\cdots,t_k} \sum_{k=1}^\ell  \Big\{\big|v((k-1)\delta; X_{t_k}) - v(t_k; X_{t_k})\big| + \big|v(t_k; X_{t_k})-v(k\delta; X_{t_k})\big|\Big\}\\
 &\le C \delta,
\end{align*}
where the last inequality follows from the property of bounded strong total variation of $v$ (see Definition~\ref{def:strong_tv}).
Putting to the two bounds together, summing over $j$, and using $\Delta^X_{0}=0$, we have
\[\big|\Delta^X_{\ell}\big| \le \sum_{j=1}^{\ell} \big|\Delta^X_{j}  - \Delta^X_{j-1}\big| \le 
C  \sum_{j=1}^\ell \int_{(j-1)\delta}^{j \delta}  |X_j^{\delta} - X_t| \de t + C \delta.\]
Squaring and taking expectations,
\begin{align*}
\E\big[(\Delta^X_{\ell})^2\big] &\le 2C^2 \E\Big(\sum_{j=1}^\ell \int_{(j-1)\delta}^{j \delta}  |X_j^{\delta} - X_t| \de t\Big)^2 + 2C^2 \delta^2\\
&\le 2C^2 \ell \delta \sum_{j=1}^\ell \int_{(j-1)\delta}^{j\delta}  \E|X_j^{\delta} - X_t|^2 \de t + 2C^2 \delta^2.
\end{align*}
Furthermore, $\E|X_j^{\delta} - X_t|^2 \le 2 \E|X_j^{\delta} - X_{\delta j}|^2+2 \E|X_{\delta j} -X_t|^2$. It is easy to show that $\E|X_t -X_s|^2 \le C|t-s|$ for all $t,s$. Therefore 
\[\E\big[(\Delta^X_{\ell})^2\big] \le 4C^2 \ell \delta^2 \sum_{j=1}^\ell \E\big[(\Delta^X_{j})^2\big] + 4C^3 \ell \delta  \sum_{j=1}^\ell\int_{(j-1)\delta}^{j\delta} (t - (\ell-1)\delta)\de t + 2C^2 \delta^2.\]
The middle term is proportional to $\ell^2 \delta^3$. Using $\ell \delta \le 1$ we obtain that for $\delta$ smaller than an absolute constant, it holds that 
\[\E\big[(\Delta^X_{\ell})^2\big] \le C \delta \sum_{j=1}^{\ell-1} \E\big[(\Delta^X_{j})^2\big] + C\delta,\]
for a different absolute constant $C$. This implies $\E\big[(\Delta^X_{\ell})^2\big] \le C \delta$ as desired.

Next, we show~\eqref{eq:error-M}. Using the relation~\eqref{eq:coupling} we have 
\begin{align}\label{eq:bound_martingale_disc}
\E\big[\big(M^{\delta}_\ell - M_{\delta \ell}\big)^2\big] &= \E\Big[\Big(\sum_{j=0}^{\ell-1}  u_j^{\delta}(X^{\delta}_{j})(Z^{\delta}_{j+1}-Z^{\delta}_{j}) - \int_0^{\delta \ell} \sqrt{\xi''(t)} u(t,X_t)\de B_t\Big)^2\Big]\nonumber\\
&= \E\Big[\Big(\sum_{j=0}^{\ell-1}\int_{j\delta}^{(j+1)\delta}  \big(u_j^{\delta}(X^{\delta}_{j}) -  u(t,X_t)\big)\sqrt{\xi''(t)}\de B_t\Big)^2\Big]\nonumber\\
&= \sum_{j=0}^{\ell-1}\int_{j\delta}^{(j+1)\delta}  \E\big[\big(u_j^{\delta}(X^{\delta}_{j}) -  u(t,X_t)\big)^2 \big]\, \xi''(t) \de t.
\end{align}
Recall that $u_j^{\delta}(x) = u(\delta j; x)/\Sigma^{\delta}_j$ for $j \ge 1$ where $\Sigma^{\delta}_j$ is given in Eq.~\eqref{eq:rescaling_u}. Since we have $q^{\delta}_{j} = \delta(j+1)$, the formula for $\Sigma^{\delta}_j$ reduces to
\[(\Sigma^{\delta}_j)^2 = \frac{\xi'(\delta(j+1)) - \xi'(\delta j)}{\delta} \E[u(\delta j; X^{\delta}_{j})^2].\]
 Let us first show the bound 
\begin{equation}\label{eq:sigma_bound}
\big|(\Sigma^{\delta}_j)^2 - 1\big| \le C\sqrt{\delta}
\end{equation}
for $\delta$ small enough. 
Since $u$ is bounded and $\xi'''$ is bounded on $[0,1]$, we have
\[\big|(\Sigma^{\delta}_j)^2 - \xi''(\delta j) \E[u(\delta j; X^{\delta}_{j})^2]\big| \le C\delta.\] 
Additionally, since $u$ is Lipschitz in space (and bounded), we use the bound Eq.~\eqref{eq:error-X} to obtain
\[\big|(\Sigma^{\delta}_j)^2 - \xi''(\delta j) \E[u(\delta j; X_{\delta j})^2]\big| \le C\sqrt{\delta}.\] 
Now, since $\E[M_t^2]=t$ for all $t \in [0,1]$ and $t \mapsto u(t,X_t)$ is a.s.\ continuous, we have by Lebesgue's differentiation theorem, for all $t\in[0,1]$,
\[\xi''(t) \E[u(t; X_{t})^2]  = 1,\]
and hence $\big|(\Sigma^{\delta}_j)^2 - 1| \le C\sqrt{\delta}$ for $\delta$ smaller than some absolute constant.
This implies the bound $|u^{\delta}_j(X^{\delta}_j) - u(\delta j; X^{\delta}_j)| \le C \big|\frac{1}{\Sigma^{\delta}_j} - 1\big| \le C \sqrt{\delta}$.
Now, going back to Eq.~\eqref{eq:bound_martingale_disc}, we have
\begin{align*}
\E\big[\big(M^{\delta}_\ell - M_{\delta \ell}\big)^2\big]  &\le 2\sum_{j=0}^{\ell-1}\int_{j\delta}^{(j+1)\delta}  \E\big[\big(u_j^{\delta}(X^{\delta}_{j}) -  u(\delta j; X^{\delta}_j)\big)^2\big]\, \xi''(t) \de t\\
&~~~+2\sum_{j=0}^{\ell-1}\int_{j\delta}^{(j+1)\delta}  \E\big[\big(u(\delta j; X^{\delta}_j) -  u(t,X_t)\big)^2\big] \xi''(t) \de t
\end{align*}
The first term is bounded by $C \ell \delta^{2} \le C \delta$. As for the second term,
\begin{align*}
\sum_{j=0}^{\ell-1}\int_{j\delta}^{(j+1)\delta}\E\big[\big(u(\delta j; X^{\delta}_j) -  u(t,X_t)\big)^2\big] \xi''(t) \de t&\le 
C\sum_{j=0}^{\ell-1}\int_{j\delta}^{(j+1)\delta}\E\big[\big(u(\delta j; X^{\delta}_j) -  u(\delta j,X_{\delta j})\big)^2\big]  \de t\\
&~~+C\sum_{j=0}^{\ell-1}\int_{j\delta}^{(j+1)\delta}\E\big[\big(u(\delta j; X_{\delta j}) -  u(\delta j,X_t)\big)^2\big] \de t\\
&~~+C\sum_{j=0}^{\ell-1}\int_{j\delta}^{(j+1)\delta}\E\big[\big(u(\delta j; X_t) -  u(t,X_t)\big)^2\big]   \de t\\
&= I +II + III. 
\end{align*}
Since $u$ is Lipschitz in space, the error bound Eq.~\eqref{eq:error-X} implies $I \le C \ell \delta^2$. Further, we have the continuity bound $\E[|X_t-X_s|^2] \le C|t-s|$, therefore $II \le C\ell \delta^2$. Finally, since $u$ has bounded strong total variation (Def.~\ref{def:strong_tv}) and $\ell \delta \le 1$, it follows that $III \le C\delta$. Putting the pieces together we obtain
\[\E\big[\big(M^{\delta}_\ell - M_{\delta \ell}\big)^2\big] \le C\delta,\]
 which is the desired bound.
\end{proof}


\subsection{Value achieved by the algorithm}

Throughout this section, we denote by $\<\bA,\bB\>_N$ the normalized scalar product between tensors $\bA,\bB\in(\reals^N)^{\otimes k}$.
Namely $\<\bA,\bB\>_N = \sum_{i_1,\dots,i_k\le N}A_{i_1,\dots,i_k} B_{i_1,\dots,i_k}/N$.
\begin{proposition}\label{prop:value}
There exists $\delta_0 >0$ and a constant $C>0$ such that for all $\delta \le \delta_0$ and $\ell \le \delta^{-1}$,
\begin{align*}
\Big|\plim_{N \to \infty} \frac{H_{N}(\m^{\ell})}{N}- \int_0^{\ell \delta} \xi''(t) \E[u(t,X_t)] \de t \Big| \le C\sqrt{\delta}.
\end{align*}
\end{proposition}
\begin{proof}
In order to compute $H_{N}(\m^{\ell})$ for large $N$, we evaluate the differences $H_{N}(\m^{k}) - H_{N}(\m^{k-1})$ for $ 1\le k\le \ell$ and sum them. We have
\[N^{-1}\big(H_{N}(\m^{k}) - H_{N}(\m^{k-1})\big) = \sum_{p} \frac{c_p}{p!} \big\langle \bW^{(p)} , (\m^{k})^{\otimes p}-(\m^{k-1})^{\otimes p}\big\rangle_N,\]
where the above inner product is of tensors of order $p$, normalized by $N$. 
We want to approximate the term 
\[A^{k}_p := \big\langle \bW^{(p)} , (\m^{k})^{\otimes p}- (\m^{k-1})^{\otimes p}\big\rangle_N\] 
with 
\[ B^{k}_p :=\Big\langle \bW^{(p)} ,  \frac{p}{2}\big((\m^{k})^{\otimes (p-1)} + (\m^{k-1})^{\otimes (p-1)} \big) \otimes (\m^{k}-\m^{k-1})\Big\rangle_N,\]
which captures the first two the terms in the binomial expansion of $A^k_p$ in $\m^{k}-\m^{k-1}$.

The result follows from the next lemma.  
\begin{lemma} \label{lem:main_error_term}
There exists $\delta_0 >0$ and a constant $C>0$ such that for all $\delta \le \delta_0$ and $\ell \le \delta^{-1}$,
\begin{align}
\Big|\plim_{N \to \infty}  \sum_{k=1}^{\ell}\sum_{p \ge 2} \frac{c_p}{p!} B^k_p -  \int_0^{\ell \delta} \xi''(t) \E[u(t,X_t)] \de t\Big|&\le C\sqrt{\delta},\label{eq:main_term}\\
\mbox{and}~~~~~\Big|\sum_{k=1}^{\ell} \sum_{p \ge 3} \frac{c_p}{p!} (A^k_p-B^k_p) \Big| &\le C \sqrt{\delta},\label{eq:error_term}
\end{align}
with probability tending to one as $N \to \infty$. 
\end{lemma}
Let us first finish the proof of Proposition~\ref{prop:value}. For $\ell \ge 1$, we have
\begin{align*} 
N^{-1}\big(H_{N}(\m^{\ell}) - H_{N}(\m^{0})\big) &= \sum_{k=1}^{\ell} N^{-1}\big(H_{N}(\m^{k}) - H_{N}(\m^{k-1})\big)\\
&= \sum_{k=1}^{\ell}\sum_{p} \frac{c_p}{p!} B^k_p + \sum_{k=0}^{\ell}\sum_{p} \frac{c_p}{p!} (A^k_p- B^k_p).
\end{align*}
Since $\m^0$ is non-random, $\plim_{N} H_{N}(\m^{0})/N = 0$, and Lemma~\ref{lem:main_error_term} yields the desired result.  
\end{proof}

\begin{proof}[Proof of Lemma~\ref{lem:main_error_term}]
We prove the two statements separately:

\noindent \textit{Proof of Eq.~\eqref{eq:main_term}}. We have
\begin{align*}
\sum_{p} \frac{c_p}{p!} B^k_p  &= \frac{1}{2} \sum_{p} c_p \big\langle \bW^{(p)}\{\m^{k}\},  \m^{k}-\m^{k-1} \big\rangle_N 
+ \frac{1}{2} \sum_{p} c_p \big\langle \bW^{(p)}\{\m^{k-1}\},  \m^{k}-\m^{k-1} \big\rangle_N\\
&:= \frac{1}{2} (S_{1,N} +S_{2,N}).
\end{align*}
By taking the scalar product of all the terms in iteration~\eqref{eq:general_amp} with $\m^{k}-\m^{k-1}$, we see that
\begin{align*}
S_{1,N}&= \langle \z^{k+1},  \m^{k}-\m^{k-1}\rangle_N + \sum_{j=0}^{k} d_{k,j} \langle \m^{j-1},  \m^{k}-\m^{k-1}\rangle_N,\\
S_{2,N} &= \langle \z^{k},  \m^{k}-\m^{k-1}\rangle_N + \sum_{j=0}^{k-1} d_{k-1,j} \langle \m^{j-1},  \m^{k}-\m^{k-1}\rangle_N.
\end{align*}
Taking $N$ to infinity and invoking Proposition~\ref{prop:state_evolution}, $S_{1,N}$ and $S_{2,N}$ converge in probability to 
\begin{align*}
\plim_{N\to\infty} S_{1,N} &=\E\big[Z^{\delta}_{k+1}(M^{\delta}_{k}-M^{\delta}_{k-1})\big] + \sum_{j=0}^{k} d_{k,j} \E\big[M^{\delta}_{j-1}(M^{\delta}_{k}-M^{\delta}_{k-1})\big],\\
\plim_{N\to\infty} S_{2,N} &=\E\big[Z^{\delta}_{k}(M^{\delta}_{k}-M^{\delta}_{k-1})\big] + \sum_{j=0}^{k-1} d_{k-1 ,j} \E\big[M^{\delta}_{j-1}(M^{\delta}_{k}-M^{\delta}_{k-1})\big],
\end{align*} 
respectively. Since $M^{\delta}$ is a martingale, the right-most terms in the above expressions vanish. Next, since $Z^{\delta}$ has independent increments, the left-most terms in the above expressions are equal, and we get
\begin{align*}
\frac{1}{2} (S_{1,N} +S_{2,N}) &= \E\big[Z^{\delta}_{k}(M^{\delta}_{k}-M^{\delta}_{k-1})\big] =\E\big[(Z^{\delta}_k -Z^{\delta}_{k-1})(M^{\delta}_{k}-M^{\delta}_{k-1})\big]\\
&= \E \big[u^{\delta}_{k-1}(X^{\delta}_{k-1}) (Z^{\delta}_k -Z^{\delta}_{k-1})^2\big].
\end{align*}
Summing over $k \in \{1,...,\ell\}$, we obtain
\begin{align*}
\plim_{N \to \infty}  \sum_{k=1}^{\ell}\sum_{p} \frac{c_p}{p!} B^k_p &=   \sum_{k=1}^{\ell}  \E \big[u^{\delta}_{k-1}(X^{\delta}_{k-1}) (Z^{\delta}_k -Z^{\delta}_{k-1})^2\big]\\
&=  \sqrt{\delta\xi'(\delta)}+ \sum_{k=2}^{\ell}  \E \big[u^{\delta}_{k-1}(X^{\delta}_{k-1})\big] \big(\xi'(q^{\delta}_{k-1})-\xi'(q^{\delta}_{k-2})\big)\\
&= \sqrt{\delta\xi'(\delta)}+ \sum_{k=2}^{\ell}  \frac{\E \big[u(\delta (k-1); X^{\delta}_{k-1})\big]}{\Sigma^{\delta}_{k-1}}\big(\xi'(\delta k)-\xi'(\delta (k-1))\big).
\end{align*}
Since $\big|\frac{1}{\Sigma^{\delta}_{k}} - 1\big| \le C\sqrt{\delta}$ (this is a consequence of Eq.~\eqref{eq:sigma_bound}) and $\xi'(\delta) \le \xi''(1)\delta$, the above is equal to
\begin{align*}
\sum_{k=2}^{\ell} &\E \big[u(\delta (k-1); X^{\delta}_{k-1})\big] \big(\xi'(\delta k)-\xi'(\delta (k-1))\big) +O(\sqrt{\delta})\\
&=\sum_{k=2}^{\ell} \E \big[u(\delta (k-1); X^{\delta}_{k-1})\big] \xi''(\delta (k-1)) \delta +O(\sqrt{\delta})\\
&=\int_0^{\ell \delta} \E[u(t, X_t)] \xi''(t) \de t +O(\sqrt{\delta}).
\end{align*}
The last equality is obtained by invoking the discretization error bound Eq.~\eqref{eq:error-X} of Proposition~\ref{prop:convergence}, and using the regularity properties of $u$, exactly as done in the proof of Eq.~\eqref{eq:error-M}.

\noindent \textit{Proof of Eq.~\eqref{eq:error_term}}. We fix $k$ and write $\m = \m^{k-1}$, $\m' = \m^{k}$ and $\balpha = \m' - \m$. 
Since the tensors $\bW^{(p)}$ are symmetric the approximation error $A^{k}_p - B^{k}_p$ is
 \begin{align}\label{eq:A-B}
 A^k_p-B^k_p &= \sum_{j =3}^p {p \choose j} \Big\langle \bW^{(p)} ,  \m^{\otimes (p-j)}\otimes \balpha^{\otimes j} \Big\rangle_N - \sum_{j =2}^{p-1} {p-1 \choose j} \Big\langle \bW^{(p)} ,   \frac{p}{2} \m^{\otimes (p-j-1)} \otimes \balpha^{\otimes (j+1)} \Big\rangle_N\nonumber\\
 &=\sum_{j =3}^p {p \choose j} (1-j/2) \Big\langle \bW^{(p)} ,  \m^{\otimes (p-j)}\otimes \balpha^{\otimes j} \Big\rangle_N.
 \end{align}
We crudely bound the above inner product as 
\[\Big|\Big\langle \bW^{(p)},  \m^{\otimes (p-j)}\otimes\balpha^{\otimes j} \Big\rangle_N\Big| \le \frac{1}{N} \big\|\bW^{(p)}\big\|_{\textup{op}} \cdot \|\m\|_2^{p-j} \cdot \|\balpha\|_2^{j}.\]
Here, $\big\|\cdot\big\|_{\textup{op}}$ is the operator (or injective) norm of symmetric tensors in the $\ell_2$ norm: for a symmetric tensor $\bT \in (\reals^N)^{\otimes k}$
\[ \big\|\bT\big\|_{\textup{op}} := \sup_{\|\bu\|_2\le1} \left\langle \bT, \bu^{\otimes k} \right\rangle.\]
The operator norm of symmetric Gaussian tensors is well understood. In particular, it is known~\cite{chen2013aizenman,auffinger2013random} that there exists a $p$-dependent constant $E_p$, known as the ground state energy of the spherical
$p$-spin model, such that $\plim_{N \to \infty} ~ N^{(p-2)/2} \cdot
\big\|\bW^{(p)}\big\|_{\textup{op}} = E_p$. A simple concentration bound \cite[Lemma 2]{richard2014statistical}  yields
\begin{align}
\prob\left(N^{(p-2)/2} \big\|\bW^{(p)}\big\|_{\textup{op}}\ge p!\sqrt{p}\right)\le e^{-Np/8}.\label{eq:UB-Inj-Norm}
\end{align} 
Furthermore, by Proposition~\ref{prop:state_evolution}, 
\begin{align*}
\plim_{N\to \infty} \,  \|\m\|_2^2 /N &= \E[(M^{\delta}_{k-1})^2] = k\delta\\
\mbox{and}~~~ \plim_{N\to \infty} \,  \|\balpha\|_2^2/N  &= \E[(M^{\delta}_{k} - M^{\delta}_{k-1})^2] = \delta.
\end{align*}
Combining the above bounds, and letting $K_p=p!\sqrt{p}$, we get
\[ \Big|\Big\langle \bW^{(p)},  \m^{\otimes (p-j)}\otimes\balpha^{\otimes j} \Big\rangle_N\Big| \le K_p \, (k \delta)^{(p-j)/2} \, \delta^{j/2},\]
for all $p$, with probability tending to one as $N \to \infty$.
Bounding $k\delta$ by 1, and plugging back into expression~\eqref{eq:A-B}, we obtain 
\[|A^k_p - B^k_p| \le K_p\sum_{j= 3}^p {p \choose j} |1-j/2| \delta^{j/2},\]
with probability tending to one as $N \to \infty$. Summing over $p$ and $k$, we obtain 
\begin{align*}
\sum_{k=1}^{\ell} \sum_{p \ge 3}  \frac{c_p}{p!} |A^k_p - B^k_p| &\le \ell \sum_{p \ge 3} \frac{c_p}{p!} K_p \sum_{j= 3}^p {p \choose j} |1-j/2| \delta^{j/2}\\
&\le \sum_{p \ge 3} \frac{c_p}{p!} K_p \sum_{j= 3}^p {p \choose j} j \delta^{(j-2)/2}\\
&\le \sum_{p \ge 3} \frac{c_p}{p!}  K_p  p^3\sqrt{\delta} e^{p\sqrt{\delta}}\\
&\le  \sqrt{\delta}\sum_{p \ge 3} c_p    p^4 e^{p\sqrt{\delta}}
\end{align*}
with probability tending to one as $N \to \infty$.
By assumption $|c_p|\le c_*\alpha^k$  for some $\alpha<1$ (since $\xi(t)<\infty$ for some $t>1$). Therefore, the 
sum is finite for $\eps$ and $\delta$ small enough, and the overall upper bound is $C \sqrt{\delta}$. This concludes the proof. 
\end{proof}

\subsection{Rounding and proof of Theorem \ref{thm:value_message_passing}}

The algorithm described in the previous section returns a sequence of vectors $\bm^{\ell}\in\reals^N$. 
In this section we describe how to round these in order to construct a feasible solution $\bsigma^{\salg}\in\{-1,+1\}^N$,
and bound the rounding error.

Fix $t_*\in [0,1]$, and let $\ell_* = \lfloor t_*/\delta\rfloor$. The rounding procedure consists in two steps:
$(i)$~Threshold the coordinates of $\bm^{\ell_*}$ to construct a vector $\hbm\in[-1,+1]^N$;
$(ii)$~Round the entries of $\hbm$ in a sequential fashion, to obtain a vector $\bsigma^{\salg}\in\{-1,+1\}^N$.

\subsubsection{Thresholding}

We define $\hbm\in [-1,+1]^N$ by thresholding entrywise $\bm^{\ell_*}$:
\begin{align*}
  \hm_i \equiv\begin{cases}
    m^{\ell_*}_i & \mbox{ if $|m^{\ell_*}_i|\le 1$,}\\
    \sign(m^{\ell_*}_i) & \mbox{ otherwise,}
    \end{cases}
\end{align*}
\begin{lemma}\label{lemma:BoundGradient}
  There exists constants $C,\eps_0>0$ such that, with high probability
  \begin{align}
    \sup \big\{\|\nabla H_N(\bx)\|_{N} :\; \|\bx\|_N\le 1+\eps_0\big\} \le C\, .\label{eq:SupGrad}
  \end{align}
\end{lemma}
\begin{proof}
  Denoting by $B_N(\eps_0)$ the supremum on the left hand side of Eq.~\eqref{eq:SupGrad}, we have
  \begin{align*}
    B_N(\eps_0) &= \sup_{\|\by\|_N\le 1, \|\bx\|_N\le 1+\eps_0}\<\by , \nabla H_N(\bx)\>_N\\
                &\le \sup_{\|\by\|_N\le 1, \|\bx\|_N\le 1+\eps_0}\sum_{p\ge 2}\frac{c_p }{p! N} p\<\bW^{(p)},\bx^{\otimes(p-1)}\otimes \by\>_N\\
                &\le \sum_{p\ge 2}\frac{c_p N^{(p-2)/2}}{p! } p\|\bW^{(p)}\|_{\op}(1+\eps_0)^{p-1}\\
                &\stackrel{(a)}{\le } \sum_{p\ge 2}c_p p^{3/2}(1+\eps_0)^{p-1}\stackrel{(b)}{\le } C\, .
  \end{align*}
  Here, the inequality $(a)$ holds by Eq.~\eqref{eq:UB-Inj-Norm}, and $(b)$
  since $|c_p|\le c_*\alpha^k$  for some $\alpha<1$ (recall that
 $\xi(t)<\infty$ for some $t>1$).
  \end{proof}
\begin{lemma}\label{lemma:Clipping}
  There exists a constant $C$ such that
  \begin{align}
    \plim_{N\to\infty}\left|\frac{1}{N}H_N(\bm^{\ell*})-\frac{1}{N}H_N(\hbm)\right|\le C\sqrt{\delta}\,.
    \end{align}
  \end{lemma}
  \begin{proof}
 Define the test function $\psi:\reals\to\reals$,
$\psi(x) \equiv \min_{z\in [-1,+1]} (x-z)^{2}$, i.e.
\begin{align*}
  \psi(x)=\begin{cases}
    (|x|-1)^2 & \mbox{ if $|x|>1$,}\\
    0 & \mbox{ if $|x|\le 1$.}
  \end{cases}
\end{align*}
Proposition \ref{prop:state_evolution} implies
\begin{align*}
  \plim_{N\to\infty}\frac{1}{N}\sum_{i=1}^N\psi(\m_i^{\ell_*}) = \E\psi(M^{\delta}_{\ell_*})\, .
\end{align*}
On the other hand, Proposition \ref{prop:convergence} yields
\begin{align*}
  \E\psi(M^{\delta}_{\ell_*})\le \E\psi(M_{t_*}) + C\delta \le C\delta\, ,
\end{align*}
where the second inequality follows because $M_{t_*}\in [-1,+1]$ almost surely.
Note that $\|\bm^{\ell_*}-\hbm\|_N^2 = \sum_{i=1}^N\psi(\m_i^{\ell_*}) /N$, and therefore we conclude
\begin{align}
  \plim_{N\to\infty}\|\bm^{\ell_*}-\hbm\|_N \le C\sqrt{\delta}\,. \label{eq:Clipping}
 \end{align}
 Now, by the intermediate value theorem, there exists $s\in [0,1]$ such that, for
 $\tbm= (1-s)\bm^{\ell_*}+s\hbm$, 
 \begin{align*}
   \left|\frac{1}{N}H_N(\bm^{\ell*})-\frac{1}{N}H_N(\hbm)\right| &= \frac{1}{N}
                                                                   \big|\<\nabla H_N(\tbm), \bm^{\ell_*}-\hbm\>_N\big|\\
                                                                 &\le \sup_{\|\bx\|_N\le 1+C\sqrt{\delta}}\|\nabla H_N(\bx)\|_N \cdot 
                                                                   \|\bm^{\ell_*}-\hbm\|_N\\
   &\le C\sqrt{\delta}\, ,
 \end{align*}
 where we used Eq.~\eqref{eq:Clipping} and Lemma \ref{lemma:BoundGradient}.
\end{proof}

\subsubsection{Rounding}
\label{sec:Rounding}

We next round $\hbm\in [-1,+1]^N$ to $\bsigma^{\salg}\in\{-1,+1\}^N$.
In order to define the rounding, we introduce the modified Hamiltonian
\begin{align*}
  \tH_N(\bsigma) :=\sum_{k=2}^\infty c_k\sum_{i_1<\dots<i_k}W^{(k)}_{i_1,\dots,i_k}\sigma_{i_1}\cdots \sigma_{i_k}\, .
\end{align*}
\begin{lemma}\label{lemma:Multilinear}
  There exist a constant $C>0$ such that, with high probability,
  \begin{align}
    \max_{\bx\in [-1,1]^N}|H_N(\bx)-\tH_N(\bx)|\le C\sqrt{N\log N}\, .
  \end{align}
\end{lemma}
\begin{proof}
  Note that $\tH_N(\bx)$ is obtained from $H_N(\bx)$ by restricting the sum in Eq.~\eqref{eq:hamiltonian}
  to terms with distinct indices.
  As a consequence,  $G_N(\bx)=H_N(\bx)-\tH_N(\bx)$ is a Gaussian process independent of $\tH_N(\bx)$.
  We therefore have
  \begin{align*}
    \E\{G_N(\bx)^2\} &= \E\{H_N(\bx)^2\}-\E\{\tH_N(\bx)^2\}\\
                     & = N\xi(\|\bx\|_N^2) -\sum_{k=2}^{\infty}c_k^2\sum_{i_1<\dots<i_k}\E\{(W^{(k)}_{i_1,\dots,i_k})^2\}\, x^2_{i_1}\cdots x^2_{i_k}\\
& = N \sum_{k=2}^{\infty}c_k^2 \frac{1}{N^k}\sum_{i_1,\dots,i_k\in D^c(N,k)}x^2_{i_1}\cdots x^2_{i_k}\, ,
  \end{align*}
  where $D^c(N,k)$ is the subset of $[N]^k$ consisting of $k$-uples that are not distinct.
  A union bound yields $|D^c(N,k)|\le N^{k-1}k(k-1)/2$, whence
  \begin{align*}
    \E\{G_N(\bx)^2\} & \le N \sum_{k=2}^{\infty}c_k^2 \frac{|D^c(N,k)|}{N^k}\le \sum_{k=2}^{\infty}c_k^2 k^2\le C\, .
  \end{align*}
  Note that, with high probability, 
  $\|\nabla G_N(\bx)\|=\|\nabla H_N(\bx)\|+\|\nabla\tH_N(\bx)\|\le C_*\sqrt{N}$ for all $\bx\in [-1,+1]^N$
  (the bound for $\nabla H_N(\bx)$ is proven in Lemma \ref{lemma:BoundGradient}, and the one for
  $\nabla \tH_N(\bx)$ follows analogously).
  Let $\cN_N(\eps)$ be an $\eps$-net (with respect the ordinary Euclidean distance) of $[-1,1]^N$.
  Then, for $\eps< t/(2C_*\sqrt{N})$
  \begin{align*}
    \prob\Big\{  \max_{\bx\in [-1,1]^N}|H_N(\bx)-\tH_N(\bx)|\ge t\Big\}&\le
    \prob\Big\{  \max_{\bx\in \cN_N(\eps)}|G_N(\bx)|\ge \frac{t}{2}\Big\} + \prob\big\{\sup_{\bx\in[-1,+1]^N}
                                                                         \|\nabla G_N(\bx)\| > C_*\sqrt{N}\big\}\\
                                                                       &\le 2|\cN_N(\eps)|\, e^{-t^2/2C} +o(1)\\
    &\le 2\left(\frac{\sqrt{N}}{\eps}\right)^N e^{-t^2/2C} +o(1)\, .
  \end{align*}
  The proof is completed by taking $\eps=1$ and $t=C_0\sqrt{N\log N}$ with $C_0$ a large enough constant.
\end{proof}

We are now in position to complete our description of the rounding procedure. Notice that
$\tH_N(\bx)$ is linear in each coordinate of $\bx$. Therefore, viewed as a function of $x_i$,
it is maximized over $[-1,+1]$ at $x_i\in\{-1,+1\}$. We starts from $\hbm$ and sequentially maximize
$\tH_N$ over each coordinate.

Explicitly, we can write $\tH_N(\bx) = \tH_N^{(-i)}(\bx_{-i})+x_i\, \Delta_i \tH_{N}(\bx_{-i})$,
where $\bx_{-i}\equiv (x_j)_{j\in [N]\setminus i}$. We then define $\bx^{(j)}$, $j\in\{0,\dots,N\}$ by letting
$\bx^{(0)} = \hbm$ and, for $j\ge 1$
\begin{align*}
  x^{(j)}_i = \begin{cases}
    x^{(j-1)}_i &\mbox{ if $i\neq j$,}\\
    \sign(\Delta_i \tH_{N}(\bx^{(j)}_{-i})) &\mbox{ if $i= j$.}
    \end{cases}
\end{align*}
We then return the last vector $\bsigma^{\salg} \equiv\bx^{(N)}$.

The proof of  Theorem \ref{thm:value_message_passing} is completed by noting that
the following inequalities hold with high probability,
\begin{align*}
  \frac{1}{N}H_N(\bsigma^{\salg}) &\stackrel{(a)}{\ge} \frac{1}{N}\tH_N(\bsigma^{\salg}) -C\sqrt{\frac{\log N}{N}}\\
                                  & \stackrel{(b)}{\ge} \frac{1}{N}\tH_N(\hbm) -C\sqrt{\frac{\log N}{N}}\\
                                  & \stackrel{(c)}{\ge} \frac{1}{N}H_N(\hbm) -2C\sqrt{\frac{\log N}{N}}\\
                                  & \stackrel{(d)}{\ge} \frac{1}{N}H_N(\bm^{\ell_*}) -C\sqrt{\delta}-2C\sqrt{\frac{\log N}{N}}\, .
\end{align*}
Here $(a)$ and $(c)$ follow from Lemma \ref{lemma:Multilinear}, $(b)$ from the fact that the
$\tH_N$ is non-decreasing along the rounding procedure, and
$(d)$ from Lemma  \ref{lemma:Clipping}. Finally, the value $H_N(\bm^{\ell_*}) /N$ is lower bounded using 
Proposition \ref{prop:value}.

\section{Analysis of the variational principle and proof of Theorem \ref{thm:VarPrinciple}}
\label{sec:PropertiesVariational}

\subsection{Properties of the variational principle}

In this section we consider the function space $\cuL$ from \eqref{eq:LDef}, which we endow with the weighted  $L^1$ distance
$\|\gamma_1-\gamma_2\|_{1,\xi''} = \|\xi''(\gamma_1-\gamma_2)\|_1 =\int_0^1\xi''(t)|\gamma_1(t)-\gamma_2(t)|\de t$. 
We will write $\gamma_n\toLx \gamma$, whenever $\|\gamma_n-\gamma\|\to 0$ as $n\to\infty$.
We recall  the space of piecewise constant functions
\begin{align}
\SF_+ = \Big\{g=\sum_{i=1}^ma_i\ind_{[t_{i-1},t_i)}:\;\; 0=t_0<t_1<\cdots<t_m=1, a_i\in\reals_{\ge 0}, m\in\naturals\Big\}\, .
\end{align}

We study the PDE \eqref{eq:PDEFirst}, with a slightly more general initial 
condition
\begin{align}
\begin{split}
\partial_t \Phi(t,x)+\frac{1}{2}\xi''(t) \Big(\partial_x^2\Phi(t,x)+\gamma(t) (\partial_x\Phi(t,x))^2\Big) &= 0\, ,\label{eq:GeneralPDE}\\
\Phi(1,x) &= f_0(x) \, .
\end{split}
\end{align}
Throughout we assume $f_0$ to be convex, continuous, non-negative, with $f_0(-x) = f_0(x)\ge 0$, and differentiable for $x\neq 0$, with $0\le f_0'(x)\le 1$ for all $x>0$.
We will write $f'_0(x)$ for the weak derivative of $f_0$ (the right and left derivatives exist but are potentially different at $x=0$).
Associated to the above PDE, we consider the following stochastic differential equation driven by Brownian motion $(B_t)_{t\ge 0}$:
\begin{align}
\de X_t = \xi''(t) \gamma(t)\partial_x\Phi(t,X_t)\, \de t+ \sqrt{\xi''(t)}\, \de B_t\, ,\;\;\;\;\; X_0=0\ . \label{eq:SDE}
\end{align}
In the following we will also write $\Phi_x$, $\Phi_{xx}$ and so on for the partial derivatives of $\Phi$, and $\Phi^{\gamma}$ whenever we want to emphasize the 
dependence of $\Phi$ on $\gamma$. We write $\partial^{\pm}_t\Phi$ for the left and right derivatives of $\Phi$.

We first collect a few properties of $\Phi(t,x)$ when $\gamma\in \SF_+$.
\begin{proposition}\label{propo:SF}
\begin{itemize}
\item[$(a)$] For any $\gamma\in\SF_+$ the solution $\Phi:[0,1]\times\reals\to \reals$ of Eq.~\eqref{eq:GeneralPDE}
exists uniquely in the classical sense and is smooth for $t\in [0,1)$. Namely, 
for any $j>0$,  $\|\partial_x^j\Phi\|_{L^\infty([0,1-\eps)\times \reals)}\le C(\gamma,\eps)$, and
$\|\partial^{\pm}_t\partial_x^j\Phi\|_{L^\infty([0,1-\eps)\times \reals)}\le C(\gamma,\eps)$, with $\partial^{+}_t\partial_x^j\Phi(t,x) =\partial^{-}_t\partial_x^j\Phi(t,x)$
whenever $t$ is a continuity point of $\gamma$.
\item[$(b)$] For any $\gamma\in\SF_+$  the solution  $\Phi$ of Eq.~\eqref{eq:GeneralPDE} is 
such that $x\mapsto \partial_x\Phi(t,\,\cdot\,)$ is non-decreasing for all $t\in [0,1]$, with
$|\partial_x\Phi(t,x)|\le 1$ for all $x\in\reals$. 
\item[$(c)$] If $\gamma_1,\gamma_2\in \SF_+$ and $\Phi^{\gamma_1}$, $\Phi^{\gamma_2}$ are the corresponding solutions, then
\begin{align*}
\|\Phi^{\gamma_1}-\Phi^{\gamma_2}\|_{\infty}\le \|\xi''(\gamma_1-\gamma_2)\|_1\, .
\end{align*}
\end{itemize}
\end{proposition}
\begin{proof}
Point $(a)$ follows from the Cole-Hopf representation which allows to write an explicit form of the solution for $\gamma\in \SF_+$ \cite{guerra2001sum,auffinger2017parisi}.
This solution is $C^{\infty}$ except (possibly) when
$t\in \{t_1,\dots,t_{m-1}\}$, the set of discontinuity points of $\gamma$. 
As a consequence of point $(a)$, the SDE  \eqref{eq:SDE} is well defined, with unique strong solution on $[0,1]$.
Further, $\Phi$ satisfies the following representation,
for $\gamma\in\SF_+$ \cite{jagannath2016dynamic}:
\begin{align*}
\partial_x\Phi(t,x) = \E\big[f'_0(X_1) |X_t =x\big]\, .
\end{align*}
Since $\|f'_0\|_{\infty}\le 1$, this implies $|\partial_x\Phi(t,x)|\le 1$. The non-decreasing property also  follows 
again by the Cole-Hopf representation.

Finally, point $(c)$ is identical to Lemma 14 in \cite{jagannath2016dynamic} (the assumption that $\gamma$ is non-decreasing is never used there). 
\end{proof}

As a consequence of Proposition \ref{propo:SF}, we can define $\Phi^{\gamma}$ by continuity for any $\gamma\in \cuL$. Namely,
we construct a sequence $\gamma_n\in \SF_+$, $\gamma_n\toLx \gamma$ and 
\begin{align*}
\Phi^{\gamma}(t,x)  = \lim_{n\to\infty}\Phi^{\gamma_n}(t,x)\, .
\end{align*}
\begin{lemma}\label{lemma:FirstDerivativePhi}
For any $\gamma\in\cuL$,  $\Phi^{\gamma}$ constructed above is such that $\partial_x\Phi^{\gamma}$ exists in weak sense, is non-decreasing,
and $|\partial_x\Phi^{\gamma}(t,x)|\le 1$ for all $t\in[0,1]$, $x\in\reals$. Further, if $\gamma_n\in \SF_+$, $\gamma_n\toLx \gamma$,
for any $t\in [0,1]$, we have $\partial_x\Phi^{\gamma_n}(t,x) \to \partial_x\Phi^{\gamma}(t,x)$ for almost every $x$.

Finally, $\Phi = \Phi^{\gamma}$ is a weak solution of the PDE \eqref{eq:GeneralPDE}.
Namely, for any $h\in C^{\infty}_c((0,1]\times\reals)$, we have
\begin{align}
0 = \int_{(0,1]}\int_{\reals} \left\{-\Phi\partial_t h +\frac{1}{2}\xi''(t) \Big(\Phi\partial_{x}^2h+\gamma(t)(\partial_x\Phi)^2 h \Big)\right\} \de x\,\de t +
\int_{\reals}\Phi(1,x)\, f_0(x)\,\de x\, .\label{eq:Weak}
\end{align}
\end{lemma}
\begin{proof}
Since $\Phi^{\gamma}(t,\,\cdot\,)$ is the uniform limit of convex $1$-Lipschitz functions, it is also convex $1$-Lipschitz. Hence its weak derivative
exists, is non-decreasing and is bounded as claimed. The claim $\partial_x\Phi^{\gamma_n}(t,x) \to \partial_x\Phi^{\gamma}(t,x)$ follows by dominated convergence.
 
In order to show that $\Phi$ is a weak solution, let $\Phi^n = \Phi^{\gamma_n}$ for $\gamma_n\in \SF_+$, $\gamma_n\toLx \gamma$
(hence $\|\Phi^n-\Phi\|_{\infty}\to 0$).
Since $\Phi^n$ is a classical solution corresponding to $\gamma_n$, we have
\begin{align*}
0 = \int_{(0,1]}\int_{\reals} \left\{-\Phi^n\partial_t h +\frac{1}{2}\xi''(t) \Big(\Phi^n\partial_{x}^2h+\gamma_n(t)(\partial_x\Phi^n)^2 h \Big)\right\} \de x\,\de t +
\int_{\reals}\Phi^n(1,x)\, f_0(x)\,\de x\, .
\end{align*}
Letting $\Delta$ denote the right-hand side of Eq.~\eqref{eq:Weak}, we have (since $\Phi^n(1,x)=\Phi(1,x)$ is independent of $n$)
\begin{align*}
\Delta = &\int_{(0,1]}\int_{\reals} \left\{(\Phi^n-\Phi)\partial_t h -\frac{1}{2}\xi''(t)  (\Phi^n-\Phi)\partial_{x}^2h\right\} \de x\, \de t\\
&-\int_{(0,1]}\int_{\reals} \frac{1}{2}\xi''(t) 
\Big(\gamma_n(t)(\partial_x\Phi^n)^2 -\gamma(t)(\partial_x\Phi)^2 \Big)h \,  \de x\,\de t \, .
\end{align*}
The first term vanishes as $n\to\infty$ by dominated convergence. For the second term, by the bound on $\partial_x\Phi$, $\partial_x\Phi^n$, 
we have
\begin{align*}
|\Delta| \le   \frac{1}{2}\int_{(0,1]}\int_{\reals} \xi''(t)\, |\gamma_n(t)-\gamma(t)|\, |h| \de x\, \de t+
\frac{1}{2}\int_{(0,1]}\int_{\reals} \xi''\gamma(t) \, \left|(\partial_x\Phi^n)^2- (\partial_x\Phi)^2\right|\, |h| \de x\, \de t\, .
\end{align*}
The first term vanishes as $n\to\infty$ since $\gamma_n\toLx\gamma$, and the  second vanishes by dominated convergence, using the fact that $\|\xi''\gamma\|_1<\infty$.
\end{proof}

\begin{lemma}\label{lemma:SecondDerivativePhi}
For $\gamma\in\cuL$ and any $t\in [0,1)$, the second derivative $\partial_{x}^2\Phi(t,\,\cdot\,)$ exists in weak sense, with 
$\sup_{0\le t\le 1-\eps}\|\partial_{x}^2\Phi(t,\,\cdot\,)\|_{L^2(\reals)} <\infty$ for any $\eps>0$.
\end{lemma}
\begin{proof}
Following \cite{jagannath2016dynamic}, it is useful to introduce the the smooth time change $\theta(t) = (\xi'(1)-\xi'(t))/2$, and define $u:[0,\theta_M]\times\reals$, $\theta_M = \xi'(1)/2$,
via $u(\theta(t),x) = \Phi(t,x)$. By a simple change of variables, $u$ is a weak solution of the PDE
\begin{align*}
\partial_{\theta} u-\Delta u = m(\theta) u_x^2\, ,\;\;\;  u(0,x) = f_0(x)\, ,
\end{align*}
where $m(s) = \gamma(\theta^{-1}(s))$. The desired claim is implied by showing that the partial derivative $\partial_{x}^2u$ exists in weak sense and is bounded uniformly  over $\theta>\eps$ (for any $\eps>0$).

Again, as in \cite{jagannath2016dynamic} the fact that $u$ is a weak solution implies the Duhamel principle
\begin{align}
\begin{split}
u(\theta) &= G_{\theta}*f_0 +\int_{0}^{\theta} m(s)\, G_{\theta-s}*u_x(s)^2 \de s\, , \label{eq:Duhamel}\\
G_t(x) & \equiv \frac{1}{\sqrt{4\pi t}}e^{-x^2/4t}\, .
\end{split}
\end{align}
(Here $*$ denotes convolution and this equation  is to be interpreted in weak sense, namely, for any $h\in C_c^{\infty}(\reals)$, $\int h(x)  u(\theta,x)\, \de x$ is 
given by the convolution with $h$ of the right hand side.)
Note that by Lemma \ref{lemma:FirstDerivativePhi},  $x\mapsto u_x(s,x)^2$ is bounded between $0$ and $1$,  non-increasing in $(-\infty,0]$, non-decreasing in $[0,\infty)$
and symmetric (the value at $x=0$ is immaterial). Hence, there exists a  measure $\nu_s$ on $[0,\infty)$, with total mass $\nu_s([0,\infty))\le 1$, such that
\begin{align*}
 u_x(s,x)^2 = \nu_s([0,x))\, \ind_{x>0} +\nu_s([0,-x))\, \ind_{x<0}\, .
\end{align*}
We then obtain, from Eq.~(\ref{eq:Duhamel})
\begin{align}
u_{xx}(\theta) &= G'_{\theta}*f'_0 +\int_{0}^{\theta} m(s)\, \int_{\reals_{\ge 0}} [G'_{\theta-s}(\, \cdot\, -x) + G'_{\theta-s}(\, \cdot\, +x)] \de\nu_s(x) \, \de s\, .
\label{eq:Duhamel2}
\end{align}

 The claim follows by showing that each of the two terms on the right hand side of Eq.~\eqref{eq:Duhamel2} is a well defined function, bounded in $L^2(\reals)$.
For the first term, notice that $f'_0$ is bounded and non-decreasing. Hence there exists a  measure $\omega_0$  on $\reals$ with $\omega_0(\reals)\le 2$,
such that $G'_{\theta}*f'_0 = G_{\theta}*\de\omega_0$, whence
\begin{align*}
\|G'_{\theta}*f'_0\|_{2} = \left\|\int G_{\theta}(\,\cdot\, -x) \, \de\omega_0( x) \right\|_2 \le 2 \|G_{\theta}\|_2\le \frac{C}{\theta^{1/4}}\, ,
\end{align*}
where the upper bound follows from Jensen's inequality.
The second term  on the right-hand side of \eqref{eq:Duhamel2} can be treated analogously. Denoting it by $w(\theta)$, we have, again by Jensen with $\theta =\theta(1-\eps)$,
\begin{align*}
\|w(\theta)\|_2&\le \int_{0}^{\theta} m(s)\, \int_{\reals_{\ge 0}} \|G'_{\theta-s}(\, \cdot\, -x) + G'_{\theta-s}(\, \cdot\, +x)\|_2\de\nu_s(x) \, \de s\\
& \le C\int_{0}^{\theta} m(s)\, \frac{1}{(\theta-s)^{3/4}} \, \de s \le C'\int_{1-\eps}^1 \frac{\xi''\gamma(s)}{(\xi'(s)-\xi'(1-\eps))^{3/4}} \, \de s\, ,
\end{align*}
where the second inequality follows by $\|G'_{t} \|_2\le C\, t^{-3/4}$. 
Decomposing the last integral, we get
\begin{align*}
\|w(\theta)\|_2&\le C'\int_{1-\eps}^{1-\eps/2} \frac{\xi''\gamma(s)}{(\xi'(s)-\xi'(1-\eps))^{3/4}} \, \de s +C'\int_{1-\eps/2}^1 \frac{\xi''\gamma(s)}{(\xi'(s)-\xi'(1-\eps))^{3/4}} \, \de s\\
&\le C'\xi''\gamma(1-\eps/2)\int_{1-\eps}^{1-\eps/2} \frac{1}{(\xi'(s)-\xi'(1-\eps))^{3/4}} \, \de s +\frac{C'}{(\xi'(1-\eps/2)-\xi'(1-\eps))^{3/4}}\int_{1-\eps/2}^1 
\xi''\gamma(s)\, \de s\\
&\le C''\|\xi''\gamma\|_{\sTV[0,1-\eps/2]}  + C'' \eps^{-3/4}\, \|\xi''\gamma\|_1\, .
\end{align*}
The last expression is bounded by some $C(\eps)<\infty$ since $\gamma\in\cuL$. 
\end{proof}

\begin{lemma}\label{lemma:Smoothness}
For any $\gamma\in \cuL$, the solution $\Phi=\Phi^{\gamma}$ constructed above is continuous on $[0,1]\times \reals$, and further
satisfies the following regularity properties for any $\eps>0$
\begin{itemize}
\item[$(a)$] $\partial_x^j\Phi \in L^{\infty}([0,1-\eps];L^2(\reals)\cap L^{\infty}(\reals))$ for $j\ge 2$.
\item[$(b)$] $\partial_t\Phi \in L^{\infty}([0,1]\times \reals)$ and $\partial_t\partial_x^j\Phi\in L^{\infty}([0,1-\eps];L^2(\reals)\cap L^{\infty}(\reals))$ for $j\ge 1$.
\end{itemize}
\end{lemma}
\begin{proof}
Continuity follows since $\Phi^{\gamma}$ is the uniform limit of continuous functions.
Point $(a)$ and $(b)$ follows from the same proof as Lemma 10 in  \cite{jagannath2016dynamic}, applied to the PDE \eqref{eq:GeneralPDE}
with boundary condition at $t=1-\eps$, whereby we use Lemma \ref{lemma:SecondDerivativePhi} to initiate the bootstrap procedure.
\end{proof}

As a consequence of the stated regularity properties of $\Phi$, we can solve the SDE \eqref{eq:SDE}.
\begin{lemma}\label{lemma:DPhiX}
For any $\gamma\in \cuL$, let $\Phi= \Phi^{\gamma}$ be the PDE solution defined above. Then,  the stochastic differential equation 
\eqref{eq:SDE} has unique strong solution on $(X_t)_{t\in [0,1]}$, which is almost surely continuous. Further, for any $t\in [0,1]$
\begin{align}
\partial_x\Phi(t,X_t) = \int_0^t \sqrt{\xi''(s)}\, \partial_{x}^2\Phi(s,X_s)\, \de B_s\, .\label{eq:DxIntegral}
\end{align}
\end{lemma}
\begin{proof}
Existence and uniqueness for $t\in [0,1-\eps)$ follow because $\partial_x\Phi(t,\,\cdot\,)$ is Lipschitz continuous and $\xi''\gamma$ is bounded on such interval
(see, e.g., \cite[Chapter 5]{oksendal2013stochastic}.) 
By letting $\eps\downarrow 0$, we obtain existence and uniqueness on $[0,1)$. Further $X_t$ can be extended at $t=1$, letting
\begin{align*}
X_1 = \int_{0}^1\xi''(t)\gamma(t) \partial_x\Phi(t,X_t) \de t+ \int_0^1\sqrt{\xi''(t)} \de B_t\, .
\end{align*}
It is easy to check that this extension is almost surely continuous at $t=1$, since
\begin{align*}
\big|X_1-X_t\big|\le \int_t^1\xi''\gamma(s)\de s +\int_t^1\sqrt{\xi''(t)} \de B_t\, .
\end{align*}
The first integral vanishes as $t\to 1$ since $\int_0^1\xi''\gamma(t) \,  \de t<\infty$, while the second vanishes by continuity of the Brownian motion.

Next notice that, since $\Phi_x = \partial_x\Phi$ smooth in space and weakly differentiable in time for $t\in[0,1)$ by Lemma \ref{lemma:Smoothness}, it is a weak solution of 
\begin{align*}
\partial_t\Phi_x(t,x)+\frac{1}{2}\xi''(t) \Big(\partial_x^2\Phi_x(t,x)+\gamma(t) \partial_x(\Phi_x(t,x))^2\Big) = 0\, .
\end{align*}
More precisely, for any $x\in\reals$ and any $h\in C_c^((0,1))$, we have 
\begin{align}
\int \left\{h(t)\partial_t\Phi_x(t,x)+\frac{\xi''(t)}{2} h(t)\, \Big(\partial_x^2\Phi_x(t,x)+\gamma(t) \partial_x(\Phi_x(t,x))^2\Big) \right\}\de t= 0\, .\label{eq:WeakFirstDerivative}
\end{align}
Equation~\eqref{eq:DxIntegral} is then obtained by It\^o formula (see Proposition 22 in \cite{jagannath2016dynamic})
\begin{align*}
\partial_x\Phi(t,X_t) = &\int_0^t \sqrt{\xi''(s)}\, \partial_{x}^2\Phi(s,X_s)\, \de B_s\\  &+\int_0^t
\left(\partial_s \Phi_x(s,X_s)+\frac{1}{2}\xi''(s) \Big(\partial_x^2\Phi_x(s,X_s)+\gamma(s) \partial_x(\Phi_x(s,X_s))^2\Big) \right\} \de s\, ,
\end{align*}
The second term vanishes by Eq.~\eqref{eq:WeakFirstDerivative}. 
\end{proof}

\begin{corollary}\label{coro:ED2}
For any $\gamma\in\cuL$ and any $0\le t_1<t_2<1$, 
\begin{align*}
\E\{\partial_x\Phi(t_2,X_{t_2})^2\}-\E\{\partial_x\Phi(t_1,X_{t_1})^2\} = \int_{t_1}^{t_2} \xi''(s)\, \E\big\{\big(\partial_{x}^2\Phi(s,X_s)\big)^2\big\}\, \de s\, .
\end{align*}
In particular, $t\mapsto \E\{\partial_x\Phi(t,X_{t})^2\}$ is Lipschitz continuous on $[0,1-\eps)$ for any $\eps>0$.
\end{corollary}
\begin{proof}
This follows from Lemma \ref{lemma:DPhiX}, using the regularity properties of Lemma \ref{lemma:Smoothness}. 
\end{proof}

\begin{lemma}\label{lemma:DxxCont}
For any $\gamma\in\cuL$, the function $t\mapsto \E\{\partial_{x}^2\Phi(t,X_t)^2\}$ is continuous on $[0,1)$.
\end{lemma}
\begin{proof}
The function is continuous by an application of bounded convergence (using the continuity of
$t\mapsto X_t$ and the regularity of Lemma \ref{lemma:Smoothness}).
\end{proof}

\begin{proposition}\label{propo:DerivativeParisi}
Let $\gamma\in \cuL$, and $\delta: [0,1)\to \reals$ be such that $\|\xi''\delta\|_{\sTV[0,t]}<\infty$ for all $t\in [0,1)$,
$\|\xi''\delta\|_1<\infty$, and $\delta(t) = 0$ for $t\in (1-\eps,1]$, $\eps>0$. Further assume that $\gamma+s\delta\ge 0$ for all $s\in [0,s_0)$ for some positive $s_0$.
Then
\begin{align}
\left. \frac{\de \Par}{\de s}(\gamma+s\delta)\right|_{s=0+} = \frac{1}{2}\int_{0}^1 \xi''(t) \delta(t) \big(\E\{\partial_x\Phi(t,X_t)^2\}\, -\, t\big)\, \de t\, .\label{eq:ParisiVariation}
\end{align}
(Here  $(X_t)_{t\in [0,1]}$ is the solution of the SDE \eqref{eq:SDE}.)
\end{proposition}
\begin{proof}
Let $\gamma^s\equiv \gamma+s\delta$, $s\in [0,\eps)$, and denote by $\Phi^s$ the corresponding solution of the Parisi PDE. 
Following the proof of Lemma 14 in \cite{jagannath2016dynamic}, we get
\begin{align}
\Phi^s(0,0)-\Phi^0(0,0) = \frac{s}{2}\int_{0}^1 \xi''(t) \delta(t) \E\{\partial_x\Phi^0(t,Y^s_t)^2\}\, \de t\, ,\label{eq:Difference}
\end{align}
where $Y^s_t$ is the solution of the SDE
\begin{align}
\de Y^s_t = \frac{1}{2}\xi''(t) \gamma^s(t)\big[\partial_x\Phi^0(t,Y^s_t)+\partial_x\Phi^s(t,Y^s_t)\big]\, \de t+ \sqrt{\xi''(t)}\, \de B_t\, ,\;\;\;\;\; Y^s_0=0\ .  \label{eq:dY}
\end{align}
We also obtain (by the same argument as in \cite[Lemma 14]{jagannath2016dynamic}, using Lemma \ref{lemma:Smoothness},
and noting that $\delta(t) = 0$ for $t>1-\eps$ and $\xi''\gamma$ is bounded on $[0,1-\eps)$)
\begin{align}
\|\partial_x\Phi^s-\partial_x\Phi^0\|_{\infty}\le C(\eps,\gamma) \|\xi''\delta\|_1 \cdot s\,. \label{eq:PerturbDx}
\end{align}

Taking the difference between this Eqs.~\eqref{eq:dY}  and  \eqref{eq:SDE}, we get, for $t\in [0,1-\eps_0)$
\begin{align*}
|Y^s_t-X_t| \le&C \int_0^t\xi''(u)|\gamma^s(u)-\gamma(u)|\de u +C\int_0^t\xi''\gamma(u)
\big|\partial_x\Phi^0(u,Y^s_u)-\partial_x\Phi^s(u,Y^s_u)\big|\, \de u\\
&+ C\int_0^t\xi''\gamma(u)
\big|\partial_x\Phi^0(u,X_u)-\partial_x\Phi^0(u,Y^s_u)\big|\, \de u\\
&\le C\|\xi''(\gamma^s-\gamma^0)\|_1 + C(\eps,\gamma)\|\xi''(\gamma^s-\gamma^0)\|_1\|\xi''\gamma\|_1+
C(\eps_0) \int_0^t\xi''\gamma(u)
\big|Y^s_u-X_u\big|\, \de u\, .
\end{align*}
In the second inequality we used Eq.~\eqref{eq:PerturbDx}, and the fact that $\partial_{x}^2\Phi$ is bounded for $t\in [0,1-\eps_0)$, see Lemma \ref{lemma:Smoothness}.
Since $\xi''\gamma(u)\le \|\xi''\gamma\|_{\sTV[0,1-\eps_0]}$ for $u\in [0,1-\eps_0)$, we finally obtain
\begin{align*}
|Y^s_t-X_t| \le C(\gamma,\eps)\, s\|\xi''\delta\|_1 + C(\gamma,\eps_0) \int_0^t \big|Y^s_u-X_u\big|\, \de u\, .
\end{align*}
Therefore, we conclude by Gronwall lemma that
\begin{align*}
\sup_{t\le 1-\eps_0}\big|Y^s_t-X_t\big| \le C(\eps,\eps_0,\gamma)\|\xi''\delta\|_1\, s
\end{align*}
Using this in Eq.~\eqref{eq:Difference}, together with the fact that $\partial_x\Phi^0$ is bounded and Lipschitz, and $\delta(t)=0$ for $t>1-\eps$, we get
\begin{align*}
\Phi^s(0,0)-\Phi^0(0,0) = \frac{s}{2}\int_{0}^1 \xi''(t) \delta(t) \E\{\partial_x\Phi^0(t,X_t)^2\}\, \de t +O(s^2)\, ,
\end{align*}
whence Eq.~\eqref{eq:ParisiVariation} immediately follows.
\end{proof}

For any $\gamma \in \cuL$, we have $\|\gamma\|_{\sTV[0,t]}<\infty$ for any $t\in [0,1)$.
We can therefore modify $\gamma$ in  (at most) countably many points to obtain a right-continuous function. 
Since this modification does not change the solution $\Phi^{\gamma}$, by Proposition \ref{propo:SF}, we will hereafter assume 
that any $\gamma \in\cuL$ is right-continuous.

For $\gamma\in\cuL$, we denote by $S(\gamma)\equiv\{t\in [0,1):\, \gamma(t)>0\}$, and by $\oS(\gamma)$ the closure
of $S(\gamma)$ in $[0,1)$ (in particular, note that $1\not\in \oS(\gamma)$).
\begin{lemma}\label{lemma:Intervals}
The support is a disjoint union of countably many intervals $S(\gamma) = \cup_{\alpha\in A}I_\alpha$, where
$I_{\alpha} = (a_{\alpha},b_{\alpha})$ or $I_{\alpha} = [a_{\alpha},b_{\alpha})$, $a_{\alpha}<b_{\alpha}$, and $A$ is countable.
\end{lemma}
\begin{proof}
If $t_0\in S(\gamma)$, then by right continuity there exists $\delta>0$ such that $[t_0,t_0+\delta)\subseteq S(\gamma)$.
This implies immediately the claim.
\end{proof}

\begin{corollary}\label{coro:Stationarity}
Assume $\gamma_*\in \cuL$ is such that $\Par(\gamma_*) = \inf_{\gamma\in\cuL}\Par(\gamma)$. Then
\begin{align}
t\in \oS(\gamma_*) &\;\;\;\Rightarrow\;\;\; \E\{\partial_x\Phi^{\gamma_*}(t,X_t)^2\} =t \, ,\label{eq:Support_1}\\
t\in [0,1)\setminus \oS(\gamma_*) &\;\;\;\Rightarrow\;\;\; \E\{\partial_x\Phi^{\gamma_*}(t,X_t)^2\} \ge t\, . \label{eq:Support_2}
\end{align}
\end{corollary}
\begin{proof} 
First consider Eq.~\eqref{eq:Support_1}.
For any $0\le t_1<t_2< 1$, set $\delta(t) =\gamma_*(t)\ind(t\in [t_1,t_2))$. Clearly $\gamma_*+s\delta\in \cuL$
for $s\in (-1,1)$. By the optimality of $\gamma_*$, and using Proposition \ref{propo:DerivativeParisi}, we have 
\begin{align*}
0 = \left. \frac{\de \Par}{\de s}(\gamma_*+s\delta)\right|_{s=0} = \frac{1}{2}\int_{t_1}^{t_2} \xi''(t) \gamma_*(t) \big(\E\{\partial_x\Phi^{\gamma_*}(t,X_t)^2\}\, -\, t\big)\, \de t
\end{align*}
Since $t_1,t_2$ are arbitrary, and $\xi''(t)>0$ for $t\in (0,1)$ this implies
$\gamma_*(t)  (\E\{\partial_x\Phi^{\gamma_*}(t,X_t)^2\}-t)=0$ for almost every $t\in [0,1)$.
Since $\gamma_*(t)$ is right-continuous and $\E\{\partial_x\Phi^{\gamma_*}(t,X_t)^2\}$ is continuous (see Corollary \ref{coro:ED2}),
it follows that $\gamma_*(t)  (\E\{\partial_x\Phi^{\gamma_*}(t,X_t)^2\}-t)=0$ for every $t\in [0,1)$.
This in turns implies $\E\{\partial_x\Phi^{\gamma_*}(t,X_t)^2\}=t$ for every $t\in S(\gamma_*)$.
This can be extended to $t\in \oS(\gamma_*)$ again by continuity of $t\mapsto \E\{\partial_x\Phi^{\gamma_*}(t,X_t)^2\}$.

Next consider Eq.~(\ref{eq:Support_2}). Notice that, by Lemma \ref{lemma:Intervals}, $[0,1)\setminus \oS(\gamma_*)$
is a disjoint union of open intervals. Let $J$ be such an interval, and consider any $[t_1,t_2]\subseteq J$. 
Set $\delta(t) = \ind(t\in(t_1,t_2])$, and notice that $\gamma_*+s\delta \in\cuL$ for $s\ge 0$.
By Proposition \ref{propo:DerivativeParisi}, we have
\begin{align*}
0\le \left. \frac{\de \Par}{\de s}(\gamma+s\delta)\right|_{s=0} = \frac{1}{2}\int_{t_1}^{t_2} \xi''(t) \big(\E\{\partial_x\Phi(t,X_t)^2\}\, -\, t\big)\, \de t\, .
\end{align*}
Since $t_1,t_2$ are arbitrary, $\xi''(t)>0$ for $t\in(0,1)$ and $t\mapsto \E\{\partial_x\Phi(t,X_t)^2\}$ is continuous,
this implies $\E\{\partial_x\Phi(t,X_t)^2\}\ge t$ for all $t\in J$, and hence all $t\in [0,1)\setminus \oS(\gamma_*)$.
\end{proof}
\begin{corollary}\label{coro:DX2_supp}
Assume $\gamma_*\in \cuL$ is such that $\Par(\gamma_*) = \inf_{\gamma\in\cuL}\Par(\gamma)$. Then
\begin{align*}
t\in \oS(\gamma_*) &\;\;\;\Rightarrow\;\;\; \xi''(t)\E\{\partial_{x}^2\Phi^{\gamma_*}(t,X_t)^2\} =1\, .
\end{align*}
\end{corollary}
\begin{proof}
Set $\Phi(t,x)= \Phi^{\gamma_*}(t,x)$.
By Lemma \ref{lemma:Intervals}, $\oS(\gamma_*)$ is a disjoint union of closed intervals with non-empty interior.
Let $K$ be one such intervals. Then, for any $[t_1,t_2]\in K$, we have, by Lemma \ref{coro:Stationarity}
\begin{align*}
t_2-t_1 = \E\{\partial_x\Phi(t_2,X_{t_2})^2\} - \E\{\partial_x\Phi(t_1,X_{t_1})^2\} =
\int_{t_1}^{t_2}\xi''(t)\E\{\partial_{x}^2\Phi(t,X_t)^2\}  \de t\, .
\end{align*}
Since $t_1,t_2$ are arbitrary, we get $\xi''(t) \E\{\partial_{x}^2\Phi(t,X_t)^2\}=1$ for  almost every $t\in K$.
Using  Lemma \ref{lemma:DxxCont} we get $\xi''(t) \E\{\partial_{x}^2\Phi(t,X_t)^2\}=1$ 
for every $t\in \oS(\gamma_*)$.
\end{proof}

\begin{lemma}\label{lemma:DensityLB}
Assume $\gamma\in\cuL$ to be such that $\gamma(t)=0$ for all $t\in (t_1,1)$, where $t_1<1$. Then, for any 
$t_*\in (t_1,1)$, the probability distribution of $X_{t_*}$ has a density $p_{t_*}$ with respect to the Lebesgue measure. Further, for any
$t_*\in (t_1,1)$  and any $M\in \reals_{\ge 0}$, there exists $\eps(t_*,M,\gamma)>0$ such that
\begin{align*}
\inf_{|x|\le M, t\in [t_*,1]} p_{t}(x) \ge \eps(t_*,M,\gamma)\, .
\end{align*}
\end{lemma}
\begin{proof}
Since the SDE \eqref{eq:SDE} has strong solutions, $X_{t_1}$  is a well defined random variable taking values in $\reals$. Therefore, there exists 
$C_1= C_1(\gamma)<\infty$ such that $\prob(|X_{t_1}|\le C_1) \ge 1/2$. For $t\in (t_1,1)$, $X_t$ satisfies $\de X_t = \sqrt{\xi''(t)}\, \de B_t$
and therefore the law of $X_t$ is the convolution of a Gaussian (with variance $\theta(t)^2 \equiv\xi'(t)-\xi(t_1)>0$) with the law of $X_{t_1}$, and therefore has a density.
To prove the desired lower bound on the density, let $f_G(x) = \exp(-x^2/2)/\sqrt{2\pi}$ denote the standard Gaussian density. Note that, for any $|x|\le M$,
\begin{align*}
p_t(x) &= \E\Big\{\frac{1}{\theta(t)}f_G\Big(\frac{x-X_{t_1}}{\theta(t)}\Big)\Big\}\\
& \ge \E\Big\{\frac{1}{\theta(t)}f_G\Big(\frac{x-X_{t_1}}{\theta(t)}\Big)\ind_{|X_{t_2}|\le C_1}\Big\}\\
& \ge \frac{1}{\theta(t)}f_G\Big(\frac{M+C_1}{\theta(t)}\Big)\, \prob(|X_{t_1}|\le C_1) \ge \frac{1}{2\theta(t)}f_G\Big(\frac{M+C_1}{\theta(t)}\Big)\, .
\end{align*}
The latter expression is lower bounded by $\eps(t_*,M,\gamma)>0$ for any $t\in [t_*,1]$, as claimed.
\end{proof}

\begin{lemma}\label{lemma:DPsi}
For any $\gamma\in\cuL$, let $\Phi=\Phi^{\gamma}$ be the solution of the Parisi PDE constructed above.  
Then, the following identities hold (as weak derivatives in $[0,1)$)
have
\begin{align}
\frac{\de\phantom{t}}{\de t}\E\{\Phi(t,X_{t})\}&=\frac{1}{2}\, \xi''(t)\gamma(t)\E\{\partial_{x}\Phi(t,X_t)^2\}\, \label{eq:DEPhi}\\
\frac{\de\phantom{t}}{\de t}\E\{X_{t}\partial_x\Phi(t,X_{t})\}&=\xi''(t)\gamma(t)\E\{\partial_{x}\Phi(t,X_t)^2\}+\xi''(t) \E\{\partial_{x}^2\Phi(t,X_t)\}\, . \label{eq:DEPhiX}
\end{align}
\end{lemma}
\begin{proof}
We will write $\Phi_t = \partial_t \Phi$, $\Phi_x = \partial_x\Phi$ and $\Phi_{xx} = \partial_{x}^2\Phi$. 
For the first identity, using  the regularity properties of Lemma \ref{lemma:Smoothness} and It\^o's formula, we get
\begin{align*}
\de \Phi(t,X_{t}) =& \Phi_t(t,X_{t}) \, \de t+ \xi''(t)\gamma(t) \Phi_x(t,X_t)^2 \de t+   \sqrt{\xi''(t)}\,\Phi_x(t,X_{t})\, \de B_t
+\frac{1}{2}\Phi_{xx}(t,X_{t})\, \xi''(t) \de t\\
=&   \frac{1}{2}\xi''(t)\gamma(t) \Phi_x(t,X_t)^2 \de t+   \sqrt{\xi''(t)}\,\Phi_{x}(t,X_{t})\, \de B_t\, ,
\end{align*}
where the equalities hold after integrating over a test function $h\in C_c^{\infty}([0,1))$ and
in the second step we used the fact that $\Phi$ is a weak solution of Eq.~\eqref{eq:GeneralPDE}. The claim \eqref{eq:DEPhi} follows 
by taking expectations. 

We proceed analogously for the second identity. Using Lemma \ref{lemma:DPhiX}, and the fact that the $(X_t)_{t\in [0,1)}$ solved the SDE \eqref{eq:SDE},
we get
\begin{align*}
\de\big(X_t \Phi_x(t,X_t)\big) = &\Phi_x(t,X_t)\de X_t+ X_t \de\big(\Phi_x(t,X_t)\big) + \xi''(t) \Phi_{xx}(t,X_t) \, \de t\\
= &\xi''(t) \gamma(t) \Phi_x(t,X_t)^2\de t+\sqrt{\xi''(t)} \Phi_x(t,X_t)\de B_t +\sqrt{\xi''(t)}  X_t \Phi_{xx}(t,X_t)\de B_t \\
&+\xi''(t) \Phi_{xx}(t,X_t) \, \de t\, .
\end{align*}
The claim \eqref{eq:DEPhi} follows again by taking expectations.
\end{proof}

\begin{theorem}\label{thm:FullSupport}
Consider the case $f_0(x) = |x|$.
Assume $\gamma_*\in \cuL$ is such that $\Par(\gamma_*) = \inf_{\gamma\in\cuL}\Par(\gamma)$. Then $\oS(\gamma_*)= [0,1)$.
\end{theorem}
\begin{proof}
Throughout this proof $\Phi(t,x) =\Phi^{\gamma_*}(t,x)$.

By Lemma \ref{lemma:Intervals}, $\oS^c(\gamma_*) = [0,1)\setminus\oS(\gamma_*)$ is a countable union of disjoint  intervals,
open in $[0,1)$. 
First assume that at least one of these intervals is of the form $(t_1,t_2)$ with $0<t_1<t_2<1$, or $[t_1=0,t_2)$, $t_2<1$. By Corollary
\ref{coro:Stationarity} and Corollary \ref{coro:DX2_supp} we know that
\begin{align}
  \E\{\partial_x\Phi(t_1,X_{t_1})^2\} = t_1\, ,\;\;\;\;  \xi''(t_2)\E\{\partial_{x}^2\Phi(t_2,X_{t_2})^2\} = 1\, ,\;\;\;\; i\in\{1,2\}\, , \label{eq:Ti}\\
\E\{\partial_x\Phi(t,X_t)^2\} \ge t\, \;\;\;\;\; \forall t\in(t_1,t_2)\, . \label{eq:Ti2}
\end{align}
(Notice that the first identity in Eq.~\eqref{eq:Ti} holds also for $t_1=0$ since $\partial_x\Phi(0,0)=0$ by a symmetry argument.)
Further, for $t\in (t_1,t_2)$, $\Phi$ solves the PDE $\Phi_t+(\xi''(t)/2)\partial_{x}^2\Phi=0$ which coincides with the 
heat equation, apart from a time change. We therefore obtain, for $t\in (t_1,t_2]$
\begin{align*}
\Phi(t,x) =\E\big\{\Phi(t_2,x+\sqrt{\xi'(t_2)-\xi'(t)} \, G)\big\}\, ,\;\;\;\;\; G\sim\normal(0,1)\, .
\end{align*}
Differentiating this equation, and using dominated convergence (thanks to the fact that $\partial_{x}^2\Phi(t_2,x)$
is bounded by Lemma \ref{lemma:Smoothness}), 
we get $\partial_{x}^2\Phi(t,x) = \E\big\{\partial_{x}^2\Phi(t_2,x+\sqrt{\xi'(t_2)-\xi'(t)} \, G)\big\}$.
Notice also that the SDE \eqref{eq:SDE} reads, for $t\in (t_1,t_2)$, 
$\de X_t = \sqrt{\xi''(t)}\, \de B_t$, and therefore we can rewrite the last equation as
\begin{align*}
\partial_{x}^2\Phi(t,X_t) =\E\big\{\partial_{x}^2\Phi(t_2,X_{t_2})|X_t\big\}\, .
\end{align*}
By Jensen inequality, we have
\begin{align}
\E\{\partial_{x}^2\Phi(t,X_t)^2\} \le \E\big\{\partial_{x}^2\Phi(t_2,X_{t_2})^2\} = \frac{1}{\xi''(t_2)}\, ,
\end{align}
where in the last step we used Eq.~\eqref{eq:Ti}. Using Corollary \ref{coro:ED2} we get, for $t\in [t_1,t_2]$
\begin{align*}
\E\{\partial_{x}\Phi(t,X_t)^2\}& = \E\{\partial_{x}\Phi(t_1,X_{t_1})^2\} +
\int_{t_1}^{t}  \xi''(s)\E\{\partial_{x}^2\Phi(s,X_{s})^2\} \,\de s\\
& \le t_1+\int_{t_1}^{t}  \frac{\xi''(s)}{\xi''(t_2)}\,\de s< t\, ,
\end{align*}
where in the last step we used the fact that $t\mapsto \xi''(t)$ is
monotone increasing.
The last equation is in contradiction with Eq.~\eqref{eq:Ti}, and therefore $\oS^c(\gamma_*)$
can be either empty, of consist of a single interval $(t_1,1)$.

In order to complete the proof, we need to rule out the case $\oS^c(\gamma_*)= (t_1,1)$.
Assume by contradiction that indeed $\oS^c(\gamma_*)= (t_1,1)$. For $t\in (t_1,1)$, let $r= r(t) =\xi'(1)-\xi'(t)$, and notice that 
$r(t)$ is monotone decreasing with $r(t) = \xi''(1)(1-t)+O((1-t)^2)$ as $t\to 1$. 
By solving the Parisi PDE in the interval $(t_1,1)$, we get $\partial_x \Phi(t,x) = \E\,\sign(G+x/\sqrt{r(t)})$, where $G\sim\normal(0,1)$,
whence, for $t\in(t_1,1)$,
\begin{align*}
1-\E\{\partial_x\Phi(t,X_t)^2\} &=\E Q\Big(\frac{X_t}{\sqrt{r(t)}}\Big) \, ,\\
Q(x) &\equiv 1- \E\big\{\sign(x+G)\big\}^2\, .
\end{align*}
Note that $0\le Q(x)\le 1$ is continuous, with $Q(0) = 1$. Hence, there exists a numerical constant $\delta_0\in (0,1)$ such that $Q(x) \ge 1/2$ for $|x|\le \delta_0$.
Therefore, fixing $t_*\in (t_1,1)$, for any $t\in(t_*,1)$
\begin{align*}
1-\E\{\partial_x\Phi(t,X_t)^2\} &\ge \frac{1}{2}\prob\big(|X_t|\le \delta_0\sqrt{r(t)}\big) \\
& \stackrel{(a)}{\ge} \delta_0\eps(t_*,1,\gamma)\, \sqrt{r(t)}\stackrel{(b)}{\ge} C \sqrt{1-t}\, ,
\end{align*}
where $(a)$ follows by Lemma \ref{lemma:DensityLB} and $(b)$ holds for some $C=C(\gamma)>0$.
We therefore obtain $\E\{\partial_x\Phi(t,X_t)^2\}\le 1-C\sqrt{1-t}$, which contradicts Corollary \ref{coro:Stationarity} for $t$ close enough to $1$.
\end{proof}

\subsection{Proof of Theorem \ref{thm:VarPrinciple}}

Before passing to the actual proof, we state and prove a simple lemma.
\begin{lemma}\label{lemma:SimpleTV}
Let $g:[a,b]\times\reals\to \reals$ be bounded and Lipschitz continuous in its first argument, i.e. $|g(t_1,x)-g(t_2,x)|\le L|t_1-t_2|$ for all $x\in\reals$, $t_1,t_2\in [a,b]$,
and $h:[a,b]\to \reals$ have bounded total variation. Then $f= gh$ has bounded \emph{strong} total variation.
\end{lemma}
\begin{proof}
Fix $a\le t_0<\dots<t_n\le b$ and $x_1,\dots,x_n\in\reals$. Then
\begin{align*}
\sum_{i=1}^n |f(t_i,x_i) - f(t_{i-1},x_i)| &= \sum_{i=1}^n |h(t_i)g(t_i,x_i) - h(t_{i-1})g(t_{i-1},x_i)|\\
&\le \sum_{i=1}^n |h(t_i)| |g(t_i,x_i) - g(t_{i-1},x_i)|+\sum_{i=1}^n |h(t_i)-h_(t_{i-1})| |g(t_{i-1},x_i)|\\
&\le \sum_{i=1}^n |h(t_i)| L|t_i - t_{i-1}|+\|g\|_{\infty}\sum_{i=1}^n |h(t_i)-h(t_{i-1})|\\
& \le L(b-a) \|h\|_{\infty} + \|g\|_{\infty}\|h\|_{\sTV}\, .
\end{align*}
The claim follows since $\|h\|_{\infty}\le |h(a)|+\|h\|_{\sTV}<\infty$.
\end{proof}

\begin{proof}[Proof of Theorem \ref{thm:VarPrinciple}]
Let $\gamma\in\cuL$ be such that $\Par(\gamma) = \inf_{\tilde{\gamma}\in\cuL}\Par(\tilde{\gamma})$. We denote by $\Phi(t,x)=\Phi^{\gamma}(t,x)$ 
the corresponding solution of the Parisi  PDE, as constructed in Section \ref{sec:PropertiesVariational},  and fix $t_*\in [0,1)$.
We apply Theorem \ref{thm:value_message_passing} whereby $u$, $v$ are defined as follows for $t\in[0,t_*]$:
\begin{align}\label{eq:choice_uv}
v(t,x) \equiv \xi''(t)\gamma(t)\partial_x\Phi(t,x)\, ,\;\;\;\; u(t,x) \equiv \partial_{x}^2\Phi(t,x)\, .
\end{align}
For $t\in (t_*,1]$, we simply set $v(t,x) = v(t_*,x)$, $u(t,x) = u(t_*,x)$. Notice that this choice is immaterial since the algorithm of
Theorem \ref{thm:value_message_passing} never uses $v(t,x)$, $u(t,x)$ for $t>t_*$.
We define $(X_t)_{t\in [0,1]}$ by solving the SDE \eqref{eq:firstSDE}, which coincides, for $t\in [0,t_*]$ with the SDE \eqref{eq:SDE}.

We next check that these choices satisfy Assumption \ref{ass:uv}. Notice that, by construction, it is sufficient to consider $t\in [0,t_*]$.
\begin{itemize}
\item[\bf{(A1)}] $v$ is bounded, since $\|\partial_x\Phi\|_{\infty}\le 1$ by Lemma \ref{lemma:FirstDerivativePhi}, and therefore,
for $t\in [0,t_*]$, $x\in \reals$, $|v(t,x)|\le \|\xi''\gamma\|_{\sTV[0,t_*]} <\infty$. Further, $u$ is bounded because
$\|\partial_{x}^2\Phi(t,\,\cdot\,)\|_{\infty}\le C(t_*)$ for almost all $t\le t_*$ (by Lemma \ref{lemma:Smoothness}.$(a)$), and that we can choose a 
representative of $\partial_{x}^2\Phi$ which is continuous in time by Lemma \ref{lemma:Smoothness}.$(b)$.
\item[\bf{(A2,3)}] $v$ is Lipschitz continuous in space, because 
$|v(t,x_1)-v(t,x_2)|\le \xi''\gamma(t) \|\partial_{x}^2\Phi(t,\,\cdot\,)\|_{\infty} |x_1-x_2|\le \|\xi''\gamma\|_{\sTV[0,t_*]}C(t_*)|x_1-x_2|\le 
C'(t_*)|x_1-x_2|$ where we used the fact that
$\|\partial_{x}^2\Phi(t,\,\cdot\,)\|_{\infty}\le C(t_*)$ for almost all $t\le t_*$ (by Lemma \ref{lemma:Smoothness}.$(a)$), and that we can choose a 
representative of $\partial_{x}^2\Phi$ which s continuous in time by Lemma \ref{lemma:Smoothness}.$(b)$.

Analogously $u$ is Lipschitz continuous in space, because $|u(t,x_1)-u(t,x_2)|\le\|\partial_{x}^3\Phi(t,\,\cdot\,)\|_{\infty} |x_1-x_2|$,
and using Lemma \ref{lemma:Smoothness}.
\item[\bf{(A4)}] $v$ has bounded strong total variation by applying Lemma \ref{lemma:SimpleTV}. Indeed $\xi''\gamma$ has 
bounded total variation on $[0,t_*]$, and $\partial_x\Phi$ is bounded by Lemma \ref{lemma:FirstDerivativePhi} and Lipschitz
by Lemma \ref{lemma:Smoothness} as discussed above.

Further, $u$ has bounded strong total variation because  $\partial_{x}^2\Phi$ is Lipschitz continuous on $[0,t_*]\times\reals$, again
by Lemma \ref{lemma:Smoothness}. 
\end{itemize}

Let us next check the other assumptions in  Theorem \ref{thm:value_message_passing}.
 By Lemma \ref{lemma:DPhiX}, we have $M_{t_*} = \partial_x\Phi(t_*,X_{t_*})$ and therefore, using Lemma \ref{lemma:FirstDerivativePhi},
$|M_{t_*}|\le 1$ almost surely.

Further $\E[M_t^2] = \E[\partial_x\Phi(t,X_{t})^2] = t$
by Corollary \ref{coro:Stationarity} and Theorem \ref{thm:FullSupport}.

We are left with the task of computing the value achieved by the algorithm. By Theorem \ref{thm:value_message_passing}, this is given by
\begin{align}
\cuE(u,v) = \int_{0}^{t_*}\xi''(t) \E[u(t,X_t)]\, \de t = \int_{0}^{t_*}\xi''(t) \E[\partial_{x}^2\Phi(t,X_t)]\, \de t  \, . \label{eq:EnergyAchievedProof}
\end{align}

Define $\Psi:[0,1)\times \reals\to \reals$ by $\Psi(t,x) = \Phi(t,x) -x\partial_x\Phi(t,x)$. By Lemma \ref{lemma:Smoothness}, we can assume this
to be continuous, and hence $\lim_{t\to 0}\E\Psi(t,X_t) =\E\Psi(0,X_0) =\Phi(0,0)$. We therefore get, using Lemma \ref{lemma:DPsi},
\begin{align*}
\Phi(0,0) = \E\Psi(t_*,X_{t_*})+\frac{1}{2}\int_0^{t_*}\xi''(t) \gamma(t) \E\{\partial_x\Phi(t,X_t)^2\}\, \de t
+\int_0^{t_*} \xi''(t) \E\{\partial_{x}^2\Phi(t,X_t)\}\, \de t\, .
\end{align*}
Comparing this with Eq.~\eqref{eq:EnergyAchievedProof}, we get
\begin{align*}
\Par(\gamma)-\cuE(u,v) &=  \E\Psi(t_*,X_{t_*})+\frac{1}{2}\int_0^{t_*}\xi''(t) \gamma(t) \left(\E\{\partial_x\Phi(t,X_t)^2\}-t\right)\, \de t\\
&= \E\Psi(t_*,X_{t_*})\, .
\end{align*}
where in the second step we used Corollary \ref{coro:Stationarity} and Theorem \ref{thm:FullSupport}.

The proof is completed by showing that we obtain $\Par(\gamma)-\cuE(u,v) = \E\Psi(t_*,X_{t_*})\le \eps$ by taking $t_*$ close enough to one.
In order to show this, recall that $\Phi(t,\,\cdot\, )$ is convex, so $\Phi(t,x) -x\partial_x\Phi(t,x)\le \Phi(t,0)$. Moreover, $|\partial_x\Phi(t,x)|\le 1$. Whence
\begin{align*}
\Phi(t,0)-|x|\le \Psi(t,x) \le \Phi(t,0)\, .
\end{align*}
Notice that $\Phi(t,0)\to 0$ as $t\to 1$ (because $\Phi$ is continuous on $[0,1]\times \reals$, and $\Phi(1,x) = |x|$), and therefore 
\begin{align*}
\limsup_{t_*\to 1} \E\Psi(t_*,X_{t_*}) = \limsup_{t_*\to 1} \E\{\Psi(t_*,X_{t_*})\}-\Phi(t,0) \le 0\, .
\end{align*}
\end{proof}

\subsection{Proof of Corollary \ref{coro:NoGap}}

The key tool is provided by the following lemma, which is a variant of Corollary \ref{coro:Stationarity}, and of
results from earlier literature (the difference being that we focus on the zero-temperature case).
\begin{lemma}
  Assume the no-overlap gap assumption to hold for the mixture $\xi$,
  namely there exists $\gamma_*\in\cuU$ strictly increasing in $[0,1)$ such that  $\Par(\gamma_*) = \inf_{\gamma\in \cuU} \Par(\gamma)$.
  Then, for any $t\in [0,1)$,
  \begin{align}
    \E\{\partial_x\Phi^{\gamma_*}(t,X_t)^2\} =t \, .\label{eq:StationarityNoGap}
    \end{align}
  \end{lemma}
  \begin{proof}
    Fix $0<t_1<t_2<1$, an define $\delta(t) = [\gamma_*(t_1)-\gamma_*(t)]\ind_{(t_1,t_2)}(t)$. It is easy to see that
    this satisfies the assumptions of Proposition \ref{propo:DerivativeParisi}, with $s_0=1$, whence, letting
    $\gamma^s=\gamma_*+s\delta$,
\begin{align*}
  \left. \frac{\de \Par}{\de s}(\gamma^s)\right|_{s=0+} = -\frac{1}{2}\int_{t_1}^{t_2} \xi''(t) \big(\gamma_*(t)-\gamma_*(t_1))\, 
  \big(\E\{\partial_x\Phi(t,X_t)^2\}\, -\, t\big)\, \de t\, .
\end{align*}
(Here $\Phi=\Phi^{\gamma_*}$.)
On the other hand, $\gamma^s\in \cuU$ for $s\in [0,1]$ (since $\gamma_*$ is strictly increasing),  whence
\begin{align*}
\int_{t_1}^{t_2} \xi''(t) \big(\gamma_*(t)-\gamma_*(t_1))\, 
  \big(\E\{\partial_x\Phi(t,X_t)^2\}\, -\, t\big)\, \de t \le 0\, .
\end{align*}
for all $t_1<t_2$. Since $\gamma_*(t)-\gamma_*(t_1)>0$ strictly for all $t>t_1$, this implies $\E\{\partial_x\Phi(t,X_t)^2\} - t\le 0$
for almost every $t$, and therefore for every $t$ by Lemma \ref{lemma:DxxCont}.

The $\E\{\partial_x\Phi(t,X_t)^2\} - t\ge 0$ is proved in the same way, by using $\delta(t) = [\gamma_*(t_2)-\gamma_*(t)]\ind_{(t_1,t_2)}(t)$.
  \end{proof}

  Let $\gamma_*$ be a strictly increasing minimizer of $\Par(\,\cdot\,)$ in $\cuU$, namely
  $\Par(\gamma_*) = \inf_{\gamma\in \cuU} \Par(\gamma)$.  We claim that $\gamma_*$ minimizes $\Par(\,\cdot\,)$ over the larger
  space $\cuL$, i.e. $\Par(\gamma_*) = \inf_{\gamma\in \cuL} \Par(\gamma)$, thus proving the corollary.
  
  By the last lemma, $\gamma_*$ verifies the stationarity condition \eqref{eq:StationarityNoGap}. 
  Since $\Par:\cuL\to\reals$ is convex
  (this follows by exactly the same proof as \cite[Theorem 20]{jagannath2016dynamic}),
  the function $s\mapsto\Par((1-s)\gamma_*+s\gamma)$ is convex over the interval $[0,1]$ for any $\gamma\in\cuL$,
  whence
  \begin{align*}
   \Par(\gamma)- \Par(\gamma_*)&\ge \left.\frac{\de\Par}{\de s}(\gamma_*+s(\gamma-\gamma_*))\right|_{s=0} \\
                               &= \frac{1}{2}\int_{0}^1 \xi''(t) \big(\gamma(t)-\gamma_*(t)\big)\, \big(\E\{\partial_x\Phi^{\gamma_*}(t,X_t)^2\}\, -\, t\big)\, \de t = 0\, .
  \end{align*}
  We thus conclude that $\gamma_*$ minimizes $\Par$ over $\cuL$.
  
\section{Proof of Theorem \ref{thm:stoch_control}}
\label{sec:ProofControl}

\subsection{A candidate solution}
  
We produce a solution to HJB~\eqref{eq:HJB} via a change of variables by taking the Legendre transform of the solution to the Parisi PDE~\eqref{eq:PDEFirst} which we redisplay here: 
\begin{align}\label{eq:PDE_Parisi_nu}
\begin{split}
\partial_t \Phi_{\gamma}(t,x)+\frac{1}{2}\xi''(t) \Big(\partial_x^2\Phi_{\gamma}(t,x)+\gamma(t) (\partial_x\Phi(t,x))^2\Big) &= 0,~~ (t,x) \in [0,1)\times \reals,\\
\Phi_{\gamma}(1,x) &= |x|,~~ x \in \reals.
\end{split}
\end{align}
Since $\gamma$ is piecewise constant, the PDE~\eqref{eq:PDE_Parisi_nu} can be solved via the Cole-Hopf transform and the solution is highly regular in space as shown in Proposition~\ref{propo:SF}. 
 We define (the negative of) the Legendre transform of $\Phi_{\gamma}$ as
 \[\Phi_{\gamma}^*(t,z) := \inf_{x \in \reals} \big\{\Phi_{\gamma}(t,x) - xz\big\},\]  
 and define a candidate solution to HJB as
 \begin{equation}\label{eq:candidate}
 V(t,z) := \Phi_{\gamma}^*(t,z) - \frac{1}{2}\nu(t)z^2 - \frac{1}{2} \int_t^1 \nu(s)\de s,
 \end{equation}
where we recall that $\nu(t) = \int_t^1 \xi''(s)\gamma(s) \de s$.
\begin{proposition}\label{prop:equality}
For all $(t,z) \in [0,1]\times (-1,1)$, $ \VF_{\gamma}(t,z)=V(t,z) $, where $\VF_{\gamma}$ is defined in~\eqref{eq:valueBellman}. 
\end{proposition}
In particular the value at $(0,0)$ is 
\begin{align*}
\VF_{\gamma}(0,0) &= \inf_{x}\Phi_{\gamma}(0,x) - \frac{1}{2} \int_0^1 \nu(s)\de s\\
&= \Phi_{\gamma}(0,0) - \frac{1}{2} \int_0^1 s\xi''(s)\gamma(s) \de s
= \Par(\gamma).
\end{align*}
 The second equality follows since $\Phi_{\gamma}(t,\,\cdot\, )$ is convex and even. This proves Proposition~\ref{prop:value_at_zero}.
 
 \subsection{Verification}
 We dedicate this section to the proof of Proposition~\ref{prop:equality}. We collect in the next lemma the regularity properties of
 $\Phi_{\gamma}$ which will be used in what follows. 
\begin{lemma}\label{lem:regularity_Phi}
For $\gamma \in \SF_+$, we have the following. 
\begin{itemize}
\item[$(a)$] $\partial_x^j \Phi_{\gamma} \in C([0,1)\times \reals)$ for all $j\ge 0$.
\item[$(b)$] $\partial_t\partial_x^{j}\Phi_{\gamma} \in C([a,b)\times \reals)$ for all $j\ge 0$ and for any interval $[a, b)$ on which
  $\gamma$ is constant.
\end{itemize}

Further, for all $t \in [0,1)$:
\begin{itemize}
\item[$(c)$] The range of the map $x \mapsto \partial_x\Phi_{\gamma}(t,x)$ is the open interval $(-1,1)$. In particular $|\partial_x\Phi_{\gamma}| < 1$.
\item[$(d)$]  $\partial_x\Phi_{\gamma}(t,\cdot)$ is strictly increasing.
\item[$(e)$]  For all $x \in \reals$, $0 < \partial_x^2\Phi_{\gamma}(t',x) \le C(t,\gamma)$ for all $t'\in [0,t]$ and some constant $C(t,\gamma)<\infty$.  
\end{itemize}
\end{lemma}
\begin{proof}
  Set $\Phi=\Phi_{\gamma}$,
  All of these claims can be proved by direct calculus using the explicit expression for the Cole-Hopf solution. Given
  $\gamma(t)=\sum_{i=1}^m\gamma_i\ind_{[t_{i-1},t_i)}$,  $0=t_0<t_1<\cdots<t_m=1$, we let $r(t) = \xi'(1)-\xi'(t)$.
  The Cole-Hopf solution is then constructed recursively as follows. For each $i\in\{1,\dots,m\}$ and each $t\in[t_{i-1},t_{i})$, let
  \begin{align}
    \Phi(t,x) = \frac{1}{\gamma_i}\log \E\exp\big\{\gamma_i\Phi(t_i,x+\sqrt{r(t)-r(t_i)} G)\big\}\, \;\;\;\; G\sim\normal(0,1)\, ,
    \label{eq:ColeHopf}
  \end{align}
  (with $\Phi(t,x)=|x|$.)
  Claims $(a)$, $(b)$ follow by standard properties of convolutions (they are also a special case of  Lemma \ref{lemma:Smoothness}).

  Claim $(c)$, $(d)$, $(e)$ can be proved by differentiating \eqref{eq:ColeHopf}. For $t\in [t_{i-1},t_{i})$ define $\sP_{t,x}$ to the probability distribution
  with density
  \begin{align*}
    \sp_{t,x}(x') \equiv \frac{1}{\E\big\{ e^{\gamma_i\Phi(t_i,x+\sqrt{r(t_i)-r(t)} G)}\big\} }\, \exp\Big\{-\frac{(x'-x)^2}{2(r(t)-r(t_i))} +\gamma_i\Phi(t_i,x')\Big\}
    \, .
  \end{align*}
  Let $\sE_{t,x}$, and $\sVar_{t,x}$ denote expectation and variance with respect to this density.
Consider first  $t\in [t_{m-1},t_m=1)$,
  \begin{align*}
    \partial_x\Phi(t,x) &= \sE_{t,x}\sign(X)\, ,\\
    \partial^2_x\Phi(t,x) &= 2\, \sp_{t,x}(0)+ \gamma_m\big\{1-\sE_{t,x} (\sign(X))^2\big\}\, .
  \end{align*}
  The last expression yields $0<\partial^2_x\Phi(t,x)<C(t_*,\gamma)$ for all $t<t_*<1$ (notice indeed that $p_{t,x}(0)$
 is bounded and non-negative for all $t<t_*$), which is claim $(e)$. In particular, this implies that $x\mapsto \partial_x\Phi(t,x)$ is strictly
  increasing (claim $(d)$). Further  $|\partial_x\Phi(t,x)|<1$, because $p_{t,x}$ is strictly positive mass on $(-\infty,0)$ and on $(0,+\infty)$.
  Finally, $\lim_{x\to\pm \infty}\partial_x\Phi(t,x)=\pm 1$ because $\sP_{t,x}((-\infty,a])\to 0$ for all $a\in\reals$ as $x\to +\infty$,
    $\sP_{t,x}([a,+\infty])\to 0$ for all $a\in\reals$ as $x\to -\infty$.

  Next, for $t\in [t_{i-1},t_i)$, $i<m$, we have
  \begin{align*}
    \partial_x\Phi(t,x) &= \sE_{t,x}\partial_x\Phi(t_i,X)\, ,\\
    \partial^2_x\Phi(t,x) &= \sE_{t,x}\partial_{x}^2\Phi(t_i,X)+ \gamma_i \Var_{t,x} (\partial_x\Phi(t_i,X))\, ,
    \end{align*}
    Claims $(c)$-$(e)$ are proved buy induction using arguments similar to the above. In particular, if $0<\partial_{x}^2\Phi(t_i,x)<C_{i+1}$
    the last equation implies $0<\partial_{x}^2\Phi(t,x)<C_{i+1}+\gamma_i$ for $t\in [t_{i-1},t_i)$.    
\end{proof}
We now prove that $\VF_{\gamma}$ is a solution to HJB~\eqref{eq:HJB}. 
 
\begin{lemma}\label{lem:HJB-sol}
For $\gamma \in \SF_+$, the function $V$ defined in Eq.~\eqref{eq:candidate} is a solution to the HJB equation~\eqref{eq:HJB} on $[0,1] \times (-1,1)$.
\end{lemma}
\begin{proof}
First, since $\nu(1)=0$, it is clear that $V$ satisfies the terminal condition $V(1,z) = 0$ for $|z| < 1$. Next, let $t<1$. Since $\Phi_{\gamma}(t,\cdot)$ is twice continuously differentiable and strictly convex, there exists a continuous strictly increasing map $z \in (-1,1) \mapsto x_t^*(z)$ defined as the unique root $x$ of the equation $\partial_x\Phi_{\gamma}(t,x) = z$. Furthermore, the envelope theorem implies that $\partial_z \Phi^*_{\gamma}(t,z) = -x_t^*(z)$ and $\partial_z^2\Phi^*_{\gamma}(t,z) = -1\big/\partial_x^2\Phi_{\gamma}(t,x_t^*(z))$ for all $z \in (-1,1)$. 

Exploiting Eq.~\eqref{eq:candidate}, we have
\begin{align*}
\partial_t V(t,z) &= \partial_t \Phi_{\gamma}(t,x_t^*(z)) + \frac{1}{2}\xi''(t)\gamma(t) z^2 + \frac{1}{2}\nu(t),\\
\partial^2_z V(t,z) &= -\frac{1}{\partial_x^2\Phi_{\gamma}(t,x_t^*(z))} - \nu(t).
\end{align*}
Given that $\Phi_{\gamma}$ satisfies the Parisi PDE, we have for all $z \in (-1,1)$
\[\partial_t V(t,z)  - \frac{1}{2}\xi''(t)\gamma(t) z^2 - \frac{1}{2}\nu(t) + \frac{\xi''(t)}{2}\Big( \gamma(t) z^2 - \frac{1}{\partial^2_z V(t,z)  + \nu(t)}\Big) = 0.\]
Simplifying the quadratic term in $z$, we obtain
\[\partial_t V(t,z)  - \frac{1}{2}\nu(t)  - \frac{\xi''(t)}{2\big(\partial^2_z V(t,z)  + \nu(t)\big)} = 0.\]
Since $\partial_x^2\Phi_{\gamma}>0$ we have $\partial^2_z V(t,z)  + \nu(t)<0$ hence 
\[\sup_{\lambda \in \reals}\Big\{  \lambda + \frac{\lambda^2}{2} \big(\nu(t)+\partial_z^2V(t,z)\big) \Big\} =-\frac{1}{2\big(\partial^2_z V(t,z)  + \nu(t)\big)}.\]
Therefore $V$ is a solution to HJB~\eqref{eq:HJB} on $[0,1) \times (-1,1)$ with the right terminal condition at $t=1$, for any function $\gamma \in \SF_+$.
\end{proof}

\begin{proof}[Proof of Proposition~\ref{prop:equality}]
We closely follow the proof of Theorem 4.1 in the textbook~\cite{touzi2012optimal}.  
We recall the expression of $\VF_{\gamma}$:
\begin{align}\label{eq:valueBellman2}
\begin{split}
\VF_{\gamma}(t,z) := \sup_{u\in D[t,1]}~~& \E\left[ \int_t^1 \xi''(s) u_s \de s  + \frac{1}{2}  \int_t^1 \nu(s) \big( \xi''(s) u_s^2 - 1\big) \de s\right],\\
\mbox{s.t.} ~~&~ z+\int_t^1 \sqrt{\xi''(s)}u_s \de B_s  \in (-1,1) ~~\mbox{a.s.},
\end{split}
\end{align}
where $\nu(t) := \int_t^1 \xi''(s)\gamma(s) \de s$.

Let us first prove the bound $V \ge \VF_{\gamma}$. Lemma~\ref{lem:regularity_Phi} implies that $V \in C^{1,2}([a,b)\times (-1,1))$ whenever $\gamma$ is constant on $[a,b)$. 

We momentarily assume that $\gamma$ is constant on $[0,1]$. Let $(t,z) \in [0,1) \times (-1,1)$, and let $(u_{s})_{s \ge t} \in D[t,1]$.
Consider the process $M^{u}$ defined by $\de M_s^{u} = \sqrt{\xi''(s)} u_s \de B_s$, $s \ge t$ with initial condition $M^{u}_t = z$, and recall that $M^{u}_1 = z+\int_t^1 \sqrt{\xi''(s)} u_s \de B_s \in (-1,1)$ a.s. Since $(M^{u}_s)_{s \ge t}$ is a martingale (w.r.t.\ the filtration of Brownian motion $\mathcal{F}_t$ we have $M^{u}_t = \E[M^{u}_1 | \mathcal{F}_t]$, and therefore $M^{u}_s \in (-1,1)$ for all $s \in [t,1]$ a.s. 

By It\^o's formula we have for $ t \le \theta<1$,
\begin{align}
\E_{t,z}[V(\theta,M^{u}_{\theta})] - V(t,z) &= \E_{t,z} \int_t^{\theta} \Big(\partial_{z}V(s,M^{u}_s) + \frac{1}{2}\xi''(s) u_s^2 \partial_{z}^2V(s,M^{u}_s)\Big)\de s \nonumber\\
&\le   \E_{t,z} \int_t^{\theta}  \Big(\partial_tV(s, M^{u}_s) + \xi''(s) \sup_{u \in \reals}\big\{  u + \frac{u^2}{2} \big(\nu(s)+\partial_z^2V(s,M^{u}_s)\big)\big\}\Big) \de s \label{eq:ineq_HJB}\\
&~~~~~~ - \E_{t,z} \int_t^{\theta} \big(\xi''(s) u_s + \frac{1}{2}\xi''(s)\nu(s)u_s^2 \big) \de s \nonumber\\
&= \E_{t,z} \int_t^{\theta} \big(\frac{1}{2}\nu(s) - \xi''(s) u_s - \frac{1}{2}\xi''(s)\nu(s)u_s^2 \big) \de s. \nonumber
\end{align}
The first inequality follows by taking a supremum over $u_s \in \reals$, and the inequality follows since $V = \VF_{\gamma}$ is a solution to HJB~\eqref{eq:HJB} as shown in Lemma~\ref{lem:HJB-sol}.

Next we have $\E[V(\theta,M^{u}_{\theta})] \to 0$ as $\theta \to 1$.  Indeed notice that $M^{u}$ is continuous, $M^{u}_1 \in (-1,1)$ almost surely,
and $V(\theta,x)$ is continuous on $[0,1]\times (-1,1)$. Therefore, for  $W_\theta \equiv V(\theta,M^u_{\theta})$, we have
$W_{\theta}\to W_1=0$ almost surely as $\theta\to 1$. Further, we claim that $W_{\theta}$ is bounded, whence the
claim $\E[W_{\theta}]=\E[V(\theta,M^{u}_{\theta})] \to \E[W_1]=0$ follows by dominated convergence. In order to show that
$W_{\theta}$ is bounded, note that $\Phi_{\gamma}(t,x)\ge |x|$ for $t\in [0,1]$
by Eq.~\eqref{eq:ColeHopf} and Jensen inequality. This implies that $0\le \Phi^{*}_{\gamma}(t,z) \le\Phi_{\gamma}(t,0)$,
and therefore $V(\theta,z)$ bounded in $[0,1]\times (-1,1)$.

Since $u$ in $L^1 \cap L^2$ we obtain that
 \[V(t,z) \ge \E_{t,z}  \int_t^{1} \big(\frac{1}{2}\nu(s)\big(\xi''(s) u_s^2 -1)  + \xi''(s) u_s \big)\de s,\] 
 for all processes $u \in D[t,1]$ satisfying $M^{u}_1 \in (-1,1)$ a.s. Therefore $V(t,z) \ge \VF_{\gamma}(t,z)$. 
 
 Returning to the general case, if $\gamma$ has $0<t_1<\cdots<t_m<1$ points of discontinuity then It\^o's formula and the above argument can be applied inside every interval $[t_i,\theta_i]$ with $\theta_i <t_{i+1}$. Letting $\theta_i \to t_{i+1}$ and applying the dominated convergence theorem, then summing over $i$, the left-hand side in Eq.~\eqref{eq:ineq_HJB} telescopes and we obtain the desired result.  
 
 Now we show the converse bound. Fix $(t,z) \in [0,1) \times (-1,1)$ and consider the control process  
 \[u^*_s :=  \partial_x^2\Phi_{\gamma}(s, X_s)~~~\mbox{for}~ s \in [t,1),~~~\mbox{and}~ u^*_1 := 0,\] 
where $(X_s)_{s \ge t}$ solves the SDE  
\[\de X_s = \xi''(s)\gamma(s) \partial_x\Phi_{\gamma}(s,X_s) \de s+ \sqrt{\xi''(s)} \de B_s,\]
with initial condition $X_{t} = x$.  
This is the same SDE as in Eq.~\eqref{eq:firstSDE} with drift $v(t,x) =  \xi''(t)\gamma(t) \partial_x\Phi_{\gamma}(t,x)$ which is bounded and Lipschitz in space for $\gamma \in \SF_+$, therefore a strong solution exists. Further, since $\frac{\de }{\de s} \E\big[\partial_x\Phi_{\gamma}(s, X_s)^2\big] = \xi''(s)\E\big[\partial_x^2\Phi_{\gamma}(s, X_s)^2\big]$ (Corollary~\ref{coro:ED2}) and $|\partial_x\Phi_{\gamma}|\le 1$ then $u^*$ is an admissible control on $[t,1]$: $u^* \in D[t,1]$.

Legendre duality implies that $u^*_s$ can also be written as
\[u^*_s =  -\frac{1}{\big(\partial^2_z V(s,M^*_s)  + \nu(s)\big)},~~~\mbox{with}~~~M^*_s := \partial_{x}\Phi_{\gamma}(s,X_s).\]
Since $\Phi_{\gamma}$ is a solution to the Parisi PDE, an application of It\^o's formula reveals that $M^*$ is a martingale which is represented by the stochastic integral
\[\de M^*_s = \sqrt{\xi''(s)} \partial_x^2\Phi_{\gamma}(s,X_s)\de B_s = \sqrt{\xi''(t)} u^*_s \de B_s,\]
with initial condition $M^*_{t} = \partial_x\Phi_{\gamma}(t,x)$. 
Further, observe that $|M^*_1| \le 1$ a.s.\ and that by surjectivity of $\partial_{x}\Phi_{\gamma}(t,\cdot)$, we can choose $x$ such that $M^*_t = z$. 
We repeat the above execution of It\^o's formula with $M^*$ and $u^*$ replacing $M^{u}$ and $u$ respectively. We see that the crucial step~\eqref{eq:ineq_HJB} holds with equality, as $u^*_s$ achieves the supremum displayed inside the integral. Hence equality $V(t,z)  = \VF_{\gamma}(t,z)$, and this conclude our proof.    
\end{proof}

\section*{Acknowledgements}

This work was partially supported by grants NSF CCF-1714305, IIS-1741162, and ONR N00014-18-1-2729.

\appendix



\section{State evolution: Proof of Proposition \ref{prop:state_evolution}}
\label{sec:ProofSE}

In this and the following appendices we prove Proposition
\ref{prop:state_evolution}. Throughout, we denote by $\bW^{(k)}\in (\reals^N)^{\otimes k}$, $k\ge 2$ a sequence of standard Gaussian tensors as defined in
Section \ref{sec:Introduction}. We also write $\bA^{(k)} = c_k \bW^{(k)}$ for the rescaled tensors, and $\xi(t) = \sum_{k\ge 2}c_k^2t^k$.
Recall the notation $\bA^{(p)}\{\bu\}\in\reals^N$, for a symmetric tensor $\bA^{(p)}\in(\reals^N)^{\otimes p}$:
\begin{align}
\bA^{(p)}\{\bu\}_i = \frac{1}{(p-1)!}\sum_{1\le i_1,\cdots,i_{p-1}\le N}A^{(p)}_{i,i_1,\cdots,i_{p-1}}u_{i_1}\cdots u_{i_{p-1}}.
\end{align}
Analogously, if $\bT\in(\reals^N)^{\otimes (p-1)}$,  $\bA^{(p)}\{\bT\}\in\reals^N$ is the vector with components
\begin{align}
\bA^{(p)}\{\bT\}_i = \frac{1}{(p-1)!}\sum_{1\le i_1,\cdots,i_{p-1}\le N}A^{(p)}_{i,i_1,\cdots,i_{p-1}}T_{i_1\dots i_{p-1}}.\label{eq:TensorTensor}
\end{align}
%

We will use the notation $\< \bv\>_N=N^{-1}\sum_{i\le N} v_i$ and $\< \bu,\bv\>_N= N^{-1}\sum_{i\le N}u_iv_i$ when $\bu,\bv\in\reals^N$ are vectors.
The corresponding norm is $\|\bu\|_{N}= \<\bu,\bu\>_N^{1/2}$. 
We will write $a_N\psim b_N$ to mean that $a_N-b_N$ converges in probability to $0$. 
Analogously, for two vectors $\bu_N, \bv_N$, we write $\bu_N\psim \bv_N$ when $\|\bu_N-\bv_N\|_N$ converges in probability to $0$. 
When $f:\reals^{k+1}\to\reals$ is a function of $k+1$
 variables, and $\bv^0,\bv^1,\dots,\bv^k\in\reals^{N}$ are $k+1$, we define $f(\bv^0,\bv^1,\dots,\bv^k)\in\mathbb R^N$ component-wise via
\begin{align}
f(\bv^0,\bv^1,\dots,\bv^k)_i=f(v^0_i,\dots,v^k_i).
\end{align}
Finally, for a sequence of vectors $\bx^0,\bx^1,\dots$, we write $\bx^{\le t} = (\bx^0,\bx^1,\dots,\bx^t)$.

To deduce the state evolution result for mixed tensors, we analyze a slightly more general iteration where each homogenous
$p$-tensor is tracked separately, while restricting ourselves to the case where the mixture $\xi$ has finitely many components: $c_k = 0$ for all $k \ge \km +1$ for some fixed $\km \ge 2$. We then proceed by an approximation argument to extend the convergence to the general case $\km = \infty$.

We begin by introducing the Gaussian process that captures the asymptotic behavior of AMP.
For each $t\in\naturals$, let $f_t:\reals^{t+1}\to \reals$ be a Lipschitz function. 
Let  $(U^{0,k})_{2\le k\le \km}$ a collection of random variables with bounded second moment, and $(U^{t,k})_{1\le t\le T,k\le \km}$ 
a centered Gaussian process, independent of $(U^{0,k})_{2\le k\le \km}$, with covariance defined by:
\begin{enumerate}
    \item $U^{t,k},U^{s,k'}$ are independent whenever $k\neq k'$ 
    \item For each $k$, the covariance of $(U^{t,k})_{t\le T}$ is defined recursively via
\begin{align}
\E[U^{t+1,p}U^{s+1,p}]& = pc^2_p\E\left\{ f_t\left(X^0,\dots,X^t\right)f_s\left(X^0,\dots,X^s\right)\right\}^{p-1}\, ,\\
X^t & \equiv  \sum_{k=2}^{\km} U^{t,k}\, .
\end{align}
\end{enumerate}

We are now in position to define the AMP algorithm. 
For each iteration $t$, the state of the algorithm is given by vectors
$\bx^t\in \reals^N$, and $\bz^{t,k}\in\reals^N$, with $k\in\{2,\dots,\km\}$. (In the following we will often omit mentioning explicitly that $k$ starts from $2$ and simply write $k\le \km$.)
We define the AMP mapping via 
\begin{align}
\AMP_t\left(\bx^0,\dots,\bx^t\right)_{p}&:=\bA^{(p)}\{f_t(\bx^0,\dots,\bx^t)\}-\sum_{s\leq t} d_{t,s,p} f_{s-1}(\bx^0,\dots, \bx^{s-1})\, , \label{eq:AMP-def1}\\
d_{t,s,p} :=c^2_p \cdot p(p-1) &\E\left\{ f_t\left(X^0,\dots,X^t\right)f_{s-1}\left(X^0,\dots,X^{s-1}\right)\right\}^{p-2} \E\Big\{ \frac{\partial f_t}{\partial x^s}(X^0,X^1,\dots, X^t)\Big\}\, .\label{eq:AMP-def2}
\end{align}
The \emph{tensor AMP iteration} then reads
\begin{align}
\bx^t=\sum_{p=2}^{\km} \bz^{t,k}\, ,\;\;\;\;\;\;
\bz^{p,t+1}=\AMP_t\left(\bx^0,\dots,\bx^t\right)_{p}\, .\label{eq:TensorAMP}
\end{align}
\begin{theorem}[State Evolution for AMP]
\label{thm:mixedAMP}
Let $\{\bW^{(k)}\}_{k\ge 2}$ be independent standard  Gaussian tensors with $\bW^{(k)}\in(\reals^N)^{\otimes k}$, and define $\bA^{(k)} = c_k\bW^{(k)}$,
$\xi(t)=\sum_{k=2}^{\km} c^2_kt^k$. Let $f_0,f_1,\dots,$ be a sequence of Lipschitz functions $f_k:\reals^{k+1}\to\reals$.
 Let $\bz^{0,2},\cdots \bz^{0,\km}\in\reals^N$ be deterministic vectors and $\bx^0 =\sum_{k=2}^{\km} \bz^{0,k}$.
Assume that, the empirical distribution of the vectors  $(z_i^{0,2},\cdots z_i^{0,\km})$, $i\le N$ converges 
in $W_2$ distance to the law of the vector $(U^{0,k})_{2\le k\le \km}$.

Let $\bx^{t}, \bz^{t,k}$, $t\ge 1$ be given by the \emph{tensor AMP} iteration. Then, for any  $T\ge 1$ and for any pseudo-Lipschitz function $\psi:\reals^{T\times \km}\to \reals$, we have
\begin{align}
\plim_{N\to\infty}\frac{1}{N}\sum_{i=1}^N\psi((z_i^{t,k})_{t\le T, k\le \km}) = \E\big\{\psi\big((U^{t,k})_{t\leq T,k\leq \km}\big)\big\}\, .
\end{align}
where $(U^{t,k})_{1\le t\le T,k\le \km}$ is a centered Gaussian process, independent of $(U^{0,k})_{2\le k\le \km}$, with covariance defined above.
\end{theorem}
In the above Proposition, $W_2$ refers to the Wasserstein, or optimal transport, distance between probability measures on $\R^{D}$ with quadratic cost $c(\bx,\by) = \|\bx -\by\|_2^2$.    

Proposition \ref{prop:state_evolution} in the special case $c_k =0$ for all $k\ge \km+1$ follows immediately from this theorem by considering the special case in which
$\psi((z^{t,k})_{t\le T, k\le \km})$ is only a function of $(\sum_{k\le \km}z^{t,k})_{t\le T}$. We extend Proposition \ref{prop:state_evolution} to the general case $D = \infty$ in Section~\ref{sec:extension_infinite_mixture}. 

\subsection{Further definitions}

We now define the notations
\begin{align*}
\bX_t&=[\bx_0|\bx_1|\;\cdots\;|\bx_t]\, ,\\
\bZ_{p,t}^k&=[\bz_{p,0}^{\otimes k}|\bz_{p,1}^{\otimes k}|\cdots\;|\bz_{p,t}^{\otimes k}]\, .
\end{align*}
Given a $N\times (t+1)$ matrix, such as $\bX_t$, and a tensor $\bA^{(p)}\in(\reals^{N})^{\otimes p}$, we write 
$\bA^{(p)}\{\bX_t\}$ for the $N\times (t+1)$ matrix with columns $\bA^{(p)}\{\bx_0\}$, \dots, $\bA^{(p)}\{\bx_t\}$:
\begin{align*}
\bA^{(p)}\{\bX_t\}&=\Big[\bA^{(p)}\{\bx_0\}\Big|\bA^{(p)}\{\bx_1\}\Big|\;\cdots\;\Big|\bA^{(p)}\{\bx_t\}\Big]\, .
\end{align*}

When $k=1$ we omit $k$, e.g. $\ZZ_{p,t}^1=\ZZ_{p,t}$. We will write $f_t(\X_t)=f_t(\x^0,\dots,\x^t)$, and we also set 
\begin{align}
\by_{p,t+1}(\ZZ_{p,t})&=\A_p\left\{f_t(\ZZ_{p,t})\right\} = \bz^{p,t+1}+\sum_{s\leq t} d_{t,s,p} f_{s-1}(\bx^0,\dots, \bx^{s-1})\, ,\label{eq:Ydef1}\\
\bY_{p,t} &= [\by_{p,1}|\;\cdots\;|\by_{p,t}]\, ,\;\;\;\;\;\; \by_{t}(\ZZ_{p,t})=\sum_p \by_{p,t}(\ZZ_{p,t})\, . \label{eq:Ydef2}
\end{align}
%

For any positive integer $k$ and $p\times T$ matrix $\M$ of length $n$ vectors we define $\F_t^{k}(\M)$ to be the length $t+1$ vector of $k$-tensors 
\begin{align}
\F_t^k(\M)=[f_0(\M)^{\otimes k}|f_1(\M)^{\otimes k}|\;\cdots\;|f_t(\M)^{\otimes k}].
\end{align}
We also define an associated $(t+1)\times (t+1)$ Gram matrix $\bG_t^k=\bG_t^k(\bM)$ via $(\bG_t^k(\bM))_{i,j}=\< f_i(\bM),f_j(\bM)\>^k_N$.
The matrix $\bG_t^k$ can be represented by the following tensor network diagram:
\begin{center}
\begin{tikzpicture}[thick]
    \node[draw,shape=circle] (v0) at (0,0) {$\F_t^{\otimes k}$};
    \node[draw,shape=circle] (v1) at (3,0) {$\F_t^{\otimes k}$};
    \draw  (v0.east) -- (1.5,0) node  {}; 
    \draw (v1) -- (1.5,0) node [label=above :{$N_{[k]}$}] {};
   
      \draw (v0.west) -- (-1.3,0) node [midway,label= above:{$t$}] {};

        \draw (v1.east) -- (4.5,0) node [midway,label= above:{$t$}] {};
         \node (Gt) at (-2,0) {$\bG_t^k$};

         \draw [double equal sign distance] (Gt.east) to (-1.45,0);
 
\end{tikzpicture}
\end{center}
We recall that in tensor networks, tensors correspond to vertices, and edges joining them to indices contracted between tensors. 
We use the convention of labeling vertices by the corresponding tensors, and edges by the dimension of the corresponding index.
Since we often have indices with dimension $N$, we label the edges by $N_1,N_2,\dots$ and so on. When two tensors are contracted along multiple indices of the same dimension (say $N$), we draw a single line between them labelled $N_{S}$ where $S$ is the set of contracted indices. For example, the middle edge in the above figure represents $k$ edges with labels $N_1,\cdots,N_k$. 

Finally, we let $\cF_t$ denote the $\sigma$-algebra generated by all iterates up to time $t$:
\begin{align}
\cF_t=\sigma\big(\{\bz_{p,s}\}_{p\le \km,s\le t}\big)=\sigma(\{\bz_{p,s},\bx_s,\f_s\}_{p\le \km,s\le t})\, .
\end{align}

\subsection{Preliminary lemmas}

\begin{lemma}\label{lem:4}
For any deterministic $\bu,\bv\in \reals^N$ and standard Gaussian symmetric $p$-tensor $\bW^{(p)}\in(\reals^N)^{\otimes p}$ we have:
\begin{enumerate}
    \item Letting $g_0\sim\normal(0,1)$ independently of $\bg\sim\normal(0,\id_N)$, we have
\begin{align}
\bW^{(p)}\{\bu\}\ed \sqrt{p} \|\bu\|_{N}^{p-1} \, \bg+  \sqrt{p(p-1)}  \|\bu\|_{N}^{p-2}\frac{\bu}{\sqrt{N}} \, g_0\, .
\end{align}
\item Letting $g_0,g_1\sim\normal(0,1)$ independent,  we have
\begin{align}
\sqrt{N}\<\bv,\bW^{(p)}\{\bu\}\>_N\ed \sqrt{p}  \|\bu\|_{N}^{p-1}\|\bv\|_{N} \, g_1+  \sqrt{p(p-1)}  \|u\|_{N}^{p-2}\<\bu,\bv\>_N \, g_0\, .
\end{align}
\item $\<\bW^{(p)}\{\bu\},\bW^{(p)}\{\bv\}\>_N\psim p\< \bu,\bv\>_N^{p-1}$.
\item For a deterministic symmetric tensor $\bT\in(\reals^N)^{\otimes p-1}$, the vector
$\bW^{(p)}\{\bT\}$ is Gaussian, with zero mean and covariance
\begin{align}
  \E\{\bW^{(p)}\{\bT\}_i\bW^{(p)}\{\bT\}_j\} = \frac{p}{N^{p-1}}\, \|\bT\|^2_{F} +\frac{p(p-1)}{N^{p-1}}\, \sum_{i_1,\dots,i_{p-2}=1}^NT_{i,i_1,\dots,i_{p-1}}T_{j,i_1,\dots,i_{p-1}}\, . 
\end{align}
\item Let $\bP\in\reals^{N\times N}$ be the orthogonal projection onto a $d$-dimensional subspace $S\subseteq\reals^N$.
$\|\bP\bW^{(p)}\{\bu\} - \bW^{(p)}\{u\}\|_2 /\|\bW^{(p)}\{u\}\|_2\psim 0$.
\item Recall that the operator (injective) norm of a tensor is given by $\|\bW^{(p)}\|_{\op}\equiv \max_{|\bu\|\le 1} \<\bW^{(p)},\bu^{\otimes p}\>$
or, equivalently,  by $\|\bW^{(p)}\|_{\op}\equiv \max_{\|\bu_1\|\le 1,\dots, \|\bu_p\|\le 1} \<\bW^{(p)},\bu_1\otimes \cdots\otimes \bu_p\>$.
If $\xi(t)<\infty$ for some $t>1$, then there exists a constant $C=C(\xi)$ such that, with probability at least $1-2e^{-N}$, 
\begin{align}
\|\bA\|_{\op}\equiv\sum_{k=2}^{\infty}\frac{N^{k/2}}{k!}\|\bA^{(k)}\|_{\op} =
\sum_{k=2}^{\infty}\frac{c_k N^{k/2}}{k!}\|\bW^{(k)}\|_{\op}\le C N\, .\label{eq:UpperBoundOpNorm}
\end{align}
\end{enumerate}
\end{lemma}
\begin{proof}
All of these statements are the elementary Gaussian calculations. The only exception is the upper bound \eqref{eq:UpperBoundOpNorm},
which follows from the  concentration bound
\[\P\Big(N^{(k-2)/2} \cdot \|\W^{(k)}\|_{\textup{op}}  \ge k! \sqrt{\log k} + \frac{k!}{\sqrt{k}} s\Big) \le e^{-N s^2/2k} ~~~ \forall s \ge 0.\]
The above is a restatement of~\cite[Lemma 2]{richard2014statistical}. We conclude by using the fact $|c_k|\le c_*\alpha^k$ for some $\alpha<1$ and letting $s = k$.
\end{proof}

We next develop  a formula for the conditional expectation of a Gaussian tensor $\bA^{(p)}$ given
a collection of  linear observations.
We set $\bD$ to be the $t\times t\times t$ tensor with entries $D_{ijk}=1$ if $i=j=k$ and $D_{ijk}=0$ otherwise.
\begin{lemma}\label{lem:symregression}
Let $\E\{\bA^{(p)}|\cF_t\}$ be the conditional expectation of $\bA^{(p)}$ given the $\sigma$-algebra 
$\cF_t=\sigma(\{\bz_{p,s},\bx_s,\f_s\}_{p\le \km,s\le t})$ generated by observations up to time $t$.
Equivalently $\E\{\bA^{(p)}|\cF_t\}$ is the conditional expectation of $\bA^{(p)}$ given the $t$ linear (in $\bA^{(p)}$) observations
\begin{align}
\bA^{(p)}\{\f_s\}=\by_{p,s+1} \, \;\;\; \mbox{ for $s\in\{0,\dots,t-1\}$.} \label{eq:LinearConstraint}
\end{align}
Then we have for $i_1,i_2,\dots,i_p\leq n$,  
%
\begin{align}
\E[\bA^{(p)}|\cF_t]_{i_1,i_2,\dots,i_p}= 
\frac{1}{p}\sum_{j=1}^p \sum_{0\leq r,s\leq t-1} (\hZZ_{p,t})_{i_j,s}\cdot (\bG^{-1}_{p-1,t-1})_{s,r}\cdot (\f_{r,i_1}\cdots \f_{r,i_{j-1}}\f_{r,i_{j+1}}\cdots \f_{r,i_{p}})\, .\label{eq:SymmRegression}
\end{align}
Here, the matrix $\hZZ_{p,t}\in\reals^{N\times t}$ is defined as the solution of a system of linear equations as follows.
Define the linear operator $\cT_{p,t}:\reals^{N\times t}\to\reals^{N\times t}$ by letting, for $i\leq N$, $0\leq s\leq t-1$:
%
\begin{align}
[\cT_{p,t}(\bZ)]_{i,s}&= \sum_{j\le N}\sum_{0\leq r,r'\leq t-1} (\f_{r'})_{i} (\f_{s})_{j} (\bG_{p-1,t-1}^{-1})_{r',r} (\bG_{p-2,t-1})_{r',s} (\bZ)_{j,r}\, ,\label{eq:Tdef}
\end{align}
Then $\hZZ_{p,t}$ is the unique solution of the following linear equation (with $\bY_{p,t}$ defined as per Eq.~\eqref{eq:Ydef1})
\begin{align}
\hZZ_{p,t}+(p-1)\cT_{p,t}(\hZZ_{p,t}) =\bY_{p,t}.\label{eq:ZZ-eq}
\end{align}

(Here, $\hZZ_{p,t} = [\hat{\bz}_{p,0},\cdots,\hat{\bz}_{p,t-1}]$ and $\bY_{p,t} = [\hat{\by}_{p,1},\cdots,\hat{\by}_{p,t}]$ have dimensions $N \times t$.)
\end{lemma}
%

The above formulas for $\E\{\bA^{(p)}|\cF_t\}$ and $\cT_{p,t}$ are somewhat difficult to parse. It is therefore useful
to draw the associated tensor networks 

\begin{center}
\begin{tikzpicture}[thick]
\node (Ap) at (-4,0) {\large$\mathbb E[\bA^{(p)}|\cF_t]=$} ;
\node (sum) at (-2,0) {\Large $\frac{1}{p}\sum_j$} ;
    \node[draw,shape=circle] (Xt) at (0,0) {$\hat \ZZ_{p,t}$};
    \node[draw,shape=circle] (Gp-1) at (2,0) {$\bG_{p-1,t-1}^{-1}$};
    \node[draw,shape=circle] (Vp) at (4,0) {$\F_{t-1}^{p-1}$};

  \draw (Xt.west) -- node[above] {$N_j$} ++ (-0.8,0);
  \draw  (Xt.east) -- node[above] {$t$} ++ (0.48,0) ;
  \draw  (Gp-1.east) -- node[above]{$t$} ++ (0.39,0) ;
  \draw  (Vp.east) -- node[above] {$N_{[p]\setminus j}$}  ++ (1.5,0);%
\end{tikzpicture}
\end{center}
The operator $\cT_{p,t}$ is represented by the following diagram, with input on the left and output on the right.
\begin{center}
\begin{tikzpicture}[thick]
\node (F) at (-5,-3) {\Large $\cT_{p,t}=$} ;
    \node[draw,shape=circle] (Gp-1) at (0,0) {$\bG_{p-1,t-1}^{-1}$};
    \node[draw,shape=circle] (D0) at (0,-2) {$\bD$};
    \node[draw,shape=circle] (Gp-2) at (0,-4) {$\bG_{p-2,t-1}$};
    \node[draw,shape=circle] (D1) at (0,-6) {$\bD_t$};
    \node[draw,shape=circle] (V0) at (2,-2) {$\F_{t-1}$};
    \node[draw,shape=circle] (V1) at (-2,-6) {$\F_{t-1}$};
  
  \draw  (Gp-1.west) -- node[above] {$t$} ++ (-1.6,0) ;
  \draw  (Gp-1.south) -- node[left]{$t$} ++ (0,-0.68) ;
  \draw  (D0.east) -- node[above] {$t$}  ++ (.98,0);
  \draw  (D0.south) -- node[left]{$t$} ++ (0,-0.68) ;
  \draw  (Gp-2.south) -- node[left]{$t$} ++ (0,-0.6) ;
  \draw  (D1.west) -- node[above]{$t$} ++ (-.9,0) ;
  \draw (V1.west) -- node[above]{$N$} ++ (-1.3,0);
  \draw (D1.east)--node[above]{$t$} ++ (1.3,0);
  \draw (V0.east)--node[above]{$N$} ++ (1.3,0);
\end{tikzpicture}
\end{center}

\begin{proof}[Proof of Lemma \ref{lem:symregression}]
Let $\cV_{p,t}$ be the affine space of symmetric tensors satisfying the constraint \eqref{eq:LinearConstraint}.
The conditional expectation $\E[\bA^{(p)}|\cF_t]$ is the tensor with minimum Frobenius norm  in the affine space $\cV_{p,t}$.
By Lagrange multipliers, there exist vectors $\bm_1,\dots,\bm_t\in\reals^N$ such that $\E[\bA^{(p)}|\cF_t]=\hbA^{(p)}$ takes the form
\begin{align}
\hbA^{(p)}_t :=\sum_{s=0}^{t-1}\sum_{j=1}^p \underbrace{\f_s\otimes \cdots \otimes \f_s}_{\mbox{$j-1$ times}}\otimes\bm_s \otimes
\underbrace{\f_s\otimes \cdots \otimes \f_s}_{\mbox{$p-j$ times}}\, .
\end{align}
Further, again by duality, if a tensor $\hbA^{(p)}$ of this form (i.e., a choice of vectors $\bm_1,\dots,\bm_t$) satisfies the constraints $\hbA^{(p)}\{\f_s\}=\by_{p,s+1}$
for $s< t$, then such a tensor is unique, and corresponds to $\E[\bA^{(p)}|\cF_t]$.
Without loss of generality, we write 
\begin{align}
\bm_r =  \sum_{s=0}^{t-1}(\bG^{-1}_{p-1,t-1})_{r,s}\hbz_{s}\, , \;\;\;\; \hZZ_{p,t}= [\hbz_1|\;\cdots\;|\hbz_t]\, .
\end{align}
By direct calculation we obtain
\begin{align}
\hbA^{(p)}_t\{\f_s\} &= \sum_{r=0}^{t-1} (\bG_{p-1,t-1})_{s,r}\bm_r+(p-1)  \sum_{r=0}^{t-1} (\bG_{p-2,t-1})_{s,r}\<\f_s,\bm_r\>\f_r\\
&= \hbz_s+(p-1)  \sum_{r=0}^{t-1} (\bG_{p-2,t-1})_{s,r}\<\f_s,\bm_r\>\f_r\, .\label{eq:SumT}
\end{align}
We next stack these vectors as columns of an $N\times t$ matrix. The first term obviously yields $\hZZ_{p,t}$. We claim that the second term
coincides with $(p-1)\cT_{p,t}(\hZZ_{p,t})$ so that overall we get
\begin{align}
\big[\hbA^{(p)}_t\{\f_{0}\},\cdots,\hbA^{(p)}_t\{\f_{t-1}\}\big] & =  \hZZ_{p,t} + (p-1)\cT_{p,t}(\hZZ_{p,t})\, .
\end{align}
This in turns implies that the equation determining $\hZZ_{p,t}$ takes the form \eqref{eq:ZZ-eq}.
 The desired claim is simply obtained by rearranging the order of sums in Eq.~\eqref{eq:SumT}.
%
%
\end{proof}

\subsection{Long AMP}
\label{sec:LAMP}

As an intermediate step towards proving Theorem \ref{thm:mixedAMP}, we introduce a new iteration that we call Long AMP (LAMP), following 
\cite{berthier2019state}. This iteration is less compact but simpler to analyze. For each $p\le \km$, let  
$\cS_{p,t}\subseteq (\reals^N)^{\otimes p}$ be the linear subspace of tensors $\bT$ that are symmetric and such that
$\bT\{\f_s\}=0$ for all $s<t$.   We denote by $\cP_t^{\perp}(\bA^{(p)})$ be the projection of $\bA^{(p)}$ 
onto $\cS_{p,t}$. 
We then define the LAMP mapping 
\begin{align}
\LAMP_t\left(\vv^{\le t}\right)_{p}&:=\cP_t^{\perp} (\bA^{(p)})\{f_t(\vv^0,\dots,\vv^t)\}+\sum_{0\leq s\leq t}  h_{t,s-1,p}\qq^{p,s},\label{eq:LAMP1}\\
h_{t,s,p}&:= \sum_{0\le r\le t-1}\big[\bG_{p-1,t-1}^{-1}\big]_{s,r}\big[\bG_{p-1,t}\big]_{r,t},~~~ h_{t,-1,p}=0. \label{eq:LAMP2}
\end{align}
Here we use the same notations $\f_t = f_t(\VV_t)$ and $\bG_{k,t}= \bG_{k,t}(\VV_t) = (\<\f_s,\f_r\>^k)_{s,r\le t}$ that we introduced for the case of AMP, however, these quantities are now different: they are computed using the vectors $\bv^0,\dots,\bv^t$. 
%
%
\begin{align}
\vv^t=\sum_{p=2}^D \qq^{p,t}\, ,\;\;\;\;\;\;\;
\qq^{p,t+1}=\LAMP_t\left(\vv^{\le t}\right)_{p}\, . 
\end{align}

Our proof strategy will be similar to the one of \cite{berthier2019state}, and proceed along the following steps:
\begin{enumerate}
  \item Prove state evolution for LAMP, under a non-degeneracy assumption.
  \item Deduce state evolution for AMP, under the previous non-degeneracy assumption.
  \item Deduce general state evolution for AMP, by perturbing the functions $f_t$ slightly to give a non-degenerate instance. 
\end{enumerate}

We will use notations analogous to the ones introduced for AMP. In particular:
\begin{align}
\VV_{t}&=[\vv_{1}|\vv_{2}|\dots|\vv_{t}]\\
\bQ_{p,t}&=[\qq_{p,1}^{\otimes p}|\qq_{p,2}^{\otimes p}|\dots|\qq_{p,t}^{\otimes p}].
\end{align}

\subsection{State Evolution for LAMP}

\begin{theorem}\label{thm:SELAMP}
  Under the assumptions of Theorem \ref{thm:mixedAMP}, let $\qq^{0,2},\cdots \qq^{0,\km}\in\reals^N$
be deterministic vectors and $\bv^0 =\sum_{k=2}^{\km} \bq^{0,k}$.
Assume that, the empirical distribution of the vectors  $(q_i^{0,2},\cdots, q_i^{0,\km})$, $i\le N$ converges 
in $W_2$ distance to the law of the vector $(U^{0,k})_{2\le k\le \km}$.

Further assume that there exist a constant $C<\infty$ such that, for all $t\le T$,
\begin{itemize}
\item[$(i)$] The matrices $\bG_{p,t}= \bG_{p,t}(\VV)$ are
  well-conditioned, i.e., $C^{-1}\le \sigma_{\min}(\bG_{p,t})\le \sigma_{\max}(\bG_{p,t})\le C$ for all $p\le \km$, $t\le T$.
\item[$(ii)$] Let the linear operator $\cT_{p,t}:\reals^{N\times t}\to\reals^{N\times t}$ be defined as per Eq.~\eqref{eq:Tdef}, with $\bG_{p,t} = \bG_{p,t}(\VV)$,
and  $\f_t=f_t(\VV)$, and define $\cL_{p,t} = {\boldsymbol 1}+(p-1)\cT_{p,t}$. Then $C^{-1}\le \sigma_{\min}(\cL_{p,t})\le \sigma_{\max}(\cL_{p,t})\le C$.
\end{itemize}

Then the following statements hold for any $t\le T$ and sufficiently large $N$:
\begin{enumerate}[label=(\alph*)]
\item Correct conditional law: 
\begin{equation}\label{eq:conditional}
\qq^{p,t+1}|_{\mathcal F_t}\ed \E[\qq^{p,t+1}|\mathcal F_t] +  \cP_t^{\perp}(\tbA^{(p)}) \{f_t(\VV_{t})\}\, .
\end{equation}
where  $\tbA^{(p)}$ is a symmetric tensor distributed identically to $\bA^{(p)}$ and independent of everything else,  and
$\cP_{t}^{\perp}$ is the projection onto the subspace $\cS_{p,t}$ defined in Section \ref{sec:LAMP}.
Further 
\begin{align}
\E[\qq^{p,t+1}|\mathcal F_t]= \sum_{0\leq s\leq t}  h_{t,s-1,p}\qq^{p,s}\, .\label{eq:CondExp}
\end{align}
Moreover, the vectors $(\qq^{p,t+1})_{p\leq \km}$ are conditionally independent given $\mathcal F_t$. 
\item Approximate isometry: we have
 \begin{align}
\< \qq^{p,r+1},\qq^{p,s+1}\>_N &\psim pc_p^2\< f_r(\VV_{r}),f_s(\VV_s)\>_N^{p-1}\, ,\label{eq:c1}\\
\<\vv^{r+1},\vv^{s+1}\>_N&\psim \xi'\left(\langle f_r(\VV^r),f_s(\VV^s)\rangle_N\right). \label{eq:c3}
\end{align}
with both sides converging in probability to constants as $N\to\infty$.
Moreover for $p\neq p'$,    
\begin{equation}\label{eq:c2}
\< \qq^{p,r+1},\qq^{p',s+1}\>_N\psim 0.
\end{equation}
\item  For any pseudo-Lipschitz function 
$\psi:\reals^{T\times \km}\to \reals$, we have
\begin{align}
\plim_{N\to\infty}\frac{1}{N}\sum_{i=1}^N\psi((q_i^{t,p})_{t\le T, p\le \km}) = \E\big\{\psi\big((U^{t,p})_{t\leq T,p\leq \km}\big)\big\}\, .\label{eq:ConvergenceLAMP}
\end{align}
where $(U^{t,p})_{1\le t\le T,p\le \km}$ is a centered Gaussian process, independent of $(U^{0,k})_{2\le k\le \km}$, 
as defined in the statement of  Theorem \ref{thm:mixedAMP}.
\end{enumerate}
\end{theorem}
Note that the conditional expectation, as given by Eqs.~\eqref{eq:LAMP2}, \eqref{eq:CondExp} can be represented by the following tensor network:
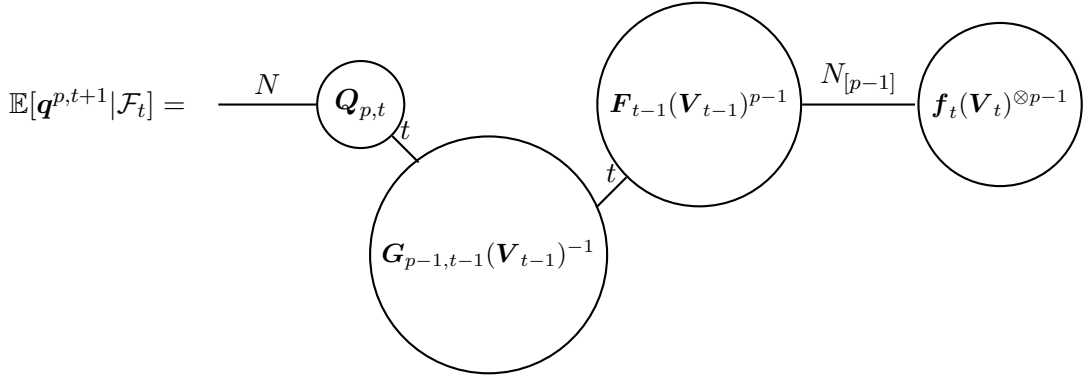
\begin{center}
    \begin{tikzpicture}[thick]
\node (xt+1) at (-4,0) {$\mathbb E[\qq^{p,t+1}|\mathcal F_t]=$} ;
    \node[draw,shape=circle] (V) at (-.5,0) {$\bQ_{p,t}$};
    \node[draw,shape=circle] (D) at (1.2,-2) {\small{$\bG_{p-1,t-1}(\VV_{t-1})^{-1}$}};
    \node[draw,shape=circle] (Vk-1) at (4,0) {\small{$\F_{t-1}(\VV_{t-1})^{p-1}$}};
      \node[draw,shape=circle] (vt+1) at (8,0) {\small{$\f_{t}(\VV_t)^{\otimes p-1}$}};
  \draw  (V.south east) -- node[above] {$t$} ++ (0.35,-0.35) ;
  \draw  (Vk-1.south west) -- node[above]{$t$} ++ (-0.4,-0.4) ;
    \draw  (V.west) -- node[above]{$N$} ++ (-1.3,0) ;
        \draw  (Vk-1.east) -- node[above]{$N_{[p-1]}$} ++ (1.5,0);
\end{tikzpicture}
\end{center}
In the next section, we will prove these statements by induction on $t$. The crucial point we exploit is the representation $(a)$.

 As a preliminary remark, we emphasize that the iteration number $t$ is bounded as $N\to\infty$, and therefore all numerical quantities not depending on $N$ (but possibly on $t$) will be treated as constants. Further we will refer to the condition 
$C_T^{-1}\le \sigma_{\min}(\bG_{k.t})\le \sigma_{\max}(\bG_{k,t})\le C_T$ simply by saying that the matrices $\bG_{k,t}$
are `well conditioned'.

\subsection{Proof of Theorem~\ref{thm:SELAMP}}

The proof will be by induction over $t$. The base case is clear, so we focus on the inductive step. We assume the statements above for $t-1$ and prove them for $t$. 

\subsubsection{Proof of $(a)$}

Note that $\cP_t^{\perp}(\bA^{(p)})$ is by construction independent of $\cF_t$, and therefore we can replace $\bA^{(p)}$
by a fresh independent matrix in Eq.~\eqref{eq:LAMP1}, whence we get the desired expression.

\subsubsection{Proof of $(b)$: Approximate isometry}

We will repeatedly apply Lemma~\ref{lem:4}. We start with  Eq.~\eqref{eq:c1}. As we are inducting on $t$, we may limit ourselves to considering inner products
 $\< \qq^{p,t+1},\qq^{p,u+1}\rangle_N$, for $u\le t$. Using Lemma~\ref{lem:4} (point 2), we get,  for $u<t$,
\begin{align*}
\< \qq^{p,t+1},\qq^{p,u+1}\>_N \psim \<\E[\qq^{p,t+1}|\cF_t],\qq^{p,u+1}\>_N. 
\end{align*}

We next use the formula in $(a)$ for $\E[\qq^{p,t+1}|\cF_t]$ (together with the expression in Eq.~\eqref{eq:LAMP1}):
\begin{align}
\<\mathbb E[\qq^{p,t+1}|\mathcal F_t],\qq^{p,u+1}\rangle_N &\stackrel{p}{\simeq} \left\langle \sum_{0\leq r,s\leq t-1} \qq^{p,s+1}(\bG_{p-1,t-1}^{-1})_{s,r} \langle \f_r,\f_{t}\rangle_N^{p-1},\qq^{p,u+1}\right\rangle_N\\
&= \sum_{0\leq r,s\leq t-1} \langle \qq^{p,s+1},\qq^{p,u+1}\rangle_N (\bG_{p-1,t-1}^{-1})_{s,r} \langle \f_r,\f_{t}\rangle_N^{p-1}\\
&\stackrel{p}{\simeq} pc_p^2\sum_{0\leq r,s\leq t-1} (\bG_{p-1,t-1})_{s,u} (\bG_{p-1,t-1}^{-1})_{s,r} \langle \f_r,\f_{t}\rangle_N^{p-1}\\
&=pc_p^2\langle \f_u,\f_t\rangle_N^{p-1}.
\end{align}

The third equality was obtained by the induction hypothesis. We next prove Eq.~\eqref{eq:c1} when $u=t$.
 We set $(\f_{t}^{\otimes p-1})_{\parallel}$ to be the projection of $\f_{t}^{\otimes p-1}$ onto ${\rm span}(\f_s^{\otimes p-1})_{s< t}$ and $
(\f_{t}^{\otimes p-1})_{\perp}=\f_{t}^{\otimes p-1}-(\f_{t}^{\otimes p-1})_{\parallel}$. We then have
\begin{align*}
\cP_t^{\perp} (\tbA^{(p)})\{\f_t\}=\cP_t^{\perp} (\tbA^{(p)})\{(\f_t^{\otimes p-1})_{\perp}\} \, ,
\end{align*}
where the right-hand side is defined according to Eq.~\eqref{eq:TensorTensor}.
Next we will use the following lemma.
\begin{lemma}\label{lem:LemmaPerp}
We have
\begin{align*}
 \cP_t^{\perp} (\tbA^{(p)})\{(\f_t^{\otimes p-1})_{\perp}\} \psim
  \tbA^{(p)}\{(\f_t^{\otimes p-1})_{\perp}\}\, .
  \end{align*}
  \end{lemma}
Using this result and Lemma~\ref{lem:4} (point 4), we have
\begin{align}
\big\|\cP_t^{\perp} (\tbA^{(p)})\{\f_t\}\big\|_{N}^2\psim \frac{pc_p^2}{N^{p-1}}\|(\f_{t}^{\otimes p-1})_{\perp}\|^{2}\, .\label{eq:PAnorm}
\end{align}
Further, again using $\cP_t^{\perp} (\tbA^{(p)})\{(\f_t^{\otimes
  p-1})_{\perp}\} \psim \tbA^{(p)}\{(\f_t^{\otimes p-1})_{\perp}\}$,
and Lemma~\ref{lem:4} (point 2)
we obtain
\begin{align}
\<\cP_t^{\perp} (\tbA^{(p)})\{\f_t\},\E[\qq^{p,t+1}|\cF_t]\>_N \psim 0\, .\label{eq:ApproxOrth}
\end{align}
We next claim that 
\begin{align}
\big\|\E[\qq^{p,t+1}|\mathcal F_t]\big\|_N^2\stackrel{p}{\simeq} \frac{pc^2_p}{N^{p-1}} \|(\f_t^{\otimes p-1})_{\parallel}\|^2\, .
\end{align}
In order to prove this, recall  the expression for $\E[\qq^{p,t+1}|\cF_t]$ from part $(a)$,
and the corresponding tensor network diagram which we reproduce here

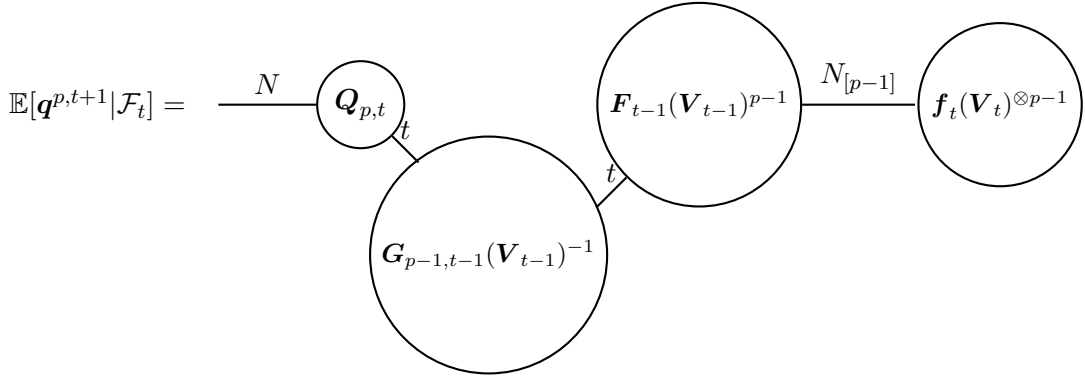
\begin{center}
    \begin{tikzpicture}[thick]
\node (xt+1) at (-4,0) {$\mathbb E[\qq^{p,t+1}|\mathcal F_t]=$} ;
    \node[draw,shape=circle] (V) at (-.5,0) {$\bQ_{p,t}$};
    \node[draw,shape=circle] (D) at (1.2,-2) {\small{$\bG_{p-1,t-1}(\VV_{t-1})^{-1}$}};
    \node[draw,shape=circle] (Vk-1) at (4,0) {\small{$\F_{t-1}(\VV_{t-1})^{p-1}$}};
      \node[draw,shape=circle] (vt+1) at (8,0) {\small{$\f_{t}(\VV_t)^{\otimes p-1}$}};
  \draw  (V.south east) -- node[above] {$t$} ++ (0.35,-0.35) ;
  \draw  (Vk-1.south west) -- node[above]{$t$} ++ (-0.4,-0.4) ;
    \draw  (V.west) -- node[above]{$N$} ++ (-1.3,0) ;
        \draw  (Vk-1.east) -- node[above]{$N_{[p-1]}$} ++ (1.5,0);
\end{tikzpicture}
\end{center}

Further, by the formula for simple linear regression, we have
\begin{align}
(\f_{t}^{\otimes p-1})_{\parallel}& =\sum_{0\leq s\leq t-1} \alpha_{s,t} \f_s^{\otimes p-1},\\
& \alpha_{s,t}=\sum_{0\leq r\le t-1} (\bG_{p-1,t-1}^{-1})_{s,r} \< \f_r,\f_{t}\>_N^{p-1}\, .
 \end{align}
This can be represented by a tensor network as follows:

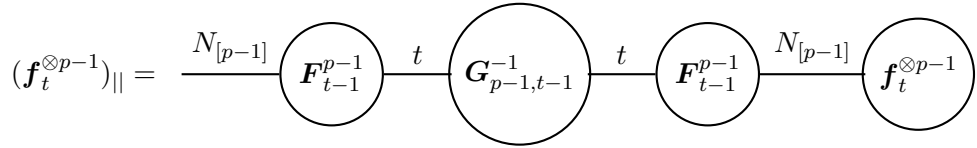
\begin{center}
\begin{tikzpicture}[thick]
\node (xt+1) at (-3.3,0) {$(\f_{t}^{\otimes p-1})_{||}=$} ;
    \node[draw,shape=circle] (V) at (0,0) {$\F_{t-1}^{p-1}$};
    \node[draw,shape=circle] (D) at (2.5,0) {$\bG_{p-1,t-1}^{-1}$};
    \node[draw,shape=circle] (Vk-1) at (5,0) {$\F_{t-1}^{p-1}$};
      \node[draw,shape=circle] (vt+1) at (7.8,0) {$\f_{t}^{\otimes p-1}$};

  \draw  (V.east) -- node[above] {$t$} ++ (.9,0) ;
  \draw  (Vk-1.west) -- node[above]{$t$} ++ (-.9,0) ;
    \draw  (V.west) -- node[above]{$N_{[p-1]}$} ++ (-1.3,0) ;
        \draw  (Vk-1.east) -- node[above]{$N_{[p-1]}$} ++ (1.4,0);
\end{tikzpicture}
\end{center}

However by part $(b)$ of the inductive step, $\sqrt{p}c_p\F_{t-1}^{\otimes p-1}$ and $\bQ_{p,t}$ are approximately unitarily equivalent in that 
$pc_p^2\langle \f_r,\f_s\rangle_N^{p-1}\psim \langle \qq_{p,r+1},\qq_{p,s+1}\rangle_N$. Therefore the above expressions have approximately the same norm 
up to the factor $p^{1/2}c_p$, since they are linear combinations with the same coefficients:
\begin{align}
\big\|\E[\qq^{p,t+1}|\mathcal F_t]\big\|_N^2\stackrel{p}{\simeq} \frac{pc_p^2}{N^{p-1}} \|(\f_{t}^{\otimes p-1})_{\parallel}\|^2. \label{eq:CEnorm}
\end{align}
Using together Eqs.~\eqref{eq:PAnorm}, \eqref{eq:ApproxOrth}, and \eqref{eq:CEnorm}, we get
\begin{align*}
\< \qq^{p,t+1},\qq^{p,t+1}\>_N &\psim \|\E[\qq^{p,t+1}|\mathcal F_t]\|_N^2+ \frac{pc_p^2}{N^{p-1}}\|(\f_t^{\otimes p-1})_{\perp}\|_N^2\\
&\stackrel{p}{\simeq} \frac{pc_p^2}{N^{p-1}}\< \f_t^{\otimes p-1},\f_t^{\otimes p-1}\>_N\\
&=pc^2_p\< \f_t,\f_t\>_N^{p-1}\, 
 \end{align*}
finishing the proof of Eq.~\eqref{eq:c1}.

Next consider Eq.~\eqref{eq:c2}, i.e., approximate orthogonality of $\qq^{p,r}$ and $\qq^{p',r}$ for $p\neq p'.$ This follows easily from the representation
in point $(a)$ which, together with Lemma~\ref{lem:4}, inductively implies that the iterates $\qq^{s,p}$ for different $p$ are approximately orthogonal.
Finally, Eq.~\eqref{eq:c3}  follows directly from Eq.~\eqref{eq:c1} and \eqref{eq:c2}.
We now prove Lemma~\ref{lem:LemmaPerp}.

 \begin{proof}[Proof of Lemma~\ref{lem:LemmaPerp}]
    For simplicity of notation, we let $\tbA= \tbA^{(p)}$. By Lagrange multipliers, there exists
    $(\blambda_s)_{s \le t-1}$ vectors in $\reals^N$ such that $\cP_t^{\perp} (\tbA) = \tbA - \bQ$, where 
\begin{align*}
  \bQ = \frac{(p-1)!}{N^{p-1}}\sum_{s=0}^{t-1} \sum_{j=1}^p  \underbrace{\f_s\otimes \cdots \otimes \f_s}_{\mbox{$j-1$ times}}\otimes \blambda_s \otimes \underbrace{\f_s\otimes \cdots \otimes \f_s}_{\mbox{$p-j$ times}}.
\end{align*}
The vectors $(\blambda_s)_{s \le t-1}$ are determined by the set of equations $\cP_t^{\perp} (\tbA) \{\f_s\} = 0$ for all $s\le t-1$
which are equivalent to
\begin{align*}
  \sum_{r<t} (\bG_{p-1,t-1})_{s,r}\blambda_r+(p-1)\sum_{r<t}(\bG_{p-2,t-1})_{s,r}\<\f_s,\blambda_r\>_N \f_r= \tbA\{\f_s\}\, .
\end{align*}
Multiplying these equations by $\bG_{p-1,t-1}^{-1}$ (recall that we assume $\bG_{p-1,t-1}$ well conditioned with high probability),
we obtain
\begin{align}
  \blambda_s+(p-1) \sum_{r',r<t}(\bG_{p-1,t-1}^{-1})_{s,r'} (\bG_{p-2,t-1})_{r',r}\<\f_s,\blambda_r\>_N \f_r=
  \sum_{r<t}(\bG_{p-1,t-1}^{-1})_{s,r}\tbA\{\f_r\}\, . \label{eq:LambdaEq}
\end{align}
This in particular implies that
\begin{align*}
  \blambda_s &=   \blambda^0_s+ \blambda^{\parallel}_s\, ,\;\;\;\;  \blambda^0_s\equiv\sum_{r<t}(\bG_{p-1,t-1}^{-1})_{s,r}\tbA\{\f_r\} \, ,
\end{align*}
where $\blambda^{\parallel}_s\in \spn((\f_r)_{r<t})$. We claim that $\|\blambda^{\parallel}\|_N\psim 0$, i.e.,
$\blambda_s\psim   \blambda^0_s$. Indeed, letting $\bLambda\in\reals^{N\times t}$ be the matrix with columns
$(\blambda_{s})_{s<t}$, and $\bLambda^0$ the matrix with columns $(\blambda^0_{s})_{s<t}$
Eq.~\eqref{eq:LambdaEq} can be written as
\begin{align*}
  \cL_{p,t}^{\sT}(\bLambda) =\bLambda^0\, .
\end{align*}
Here we recall $\cL_{p,t}=\bfone+(p-1)\cT_{p.t}$ and $\cT_{p,t}\in\reals^{Nt\times Nt}$ is defined in Eq.~\eqref{eq:Tdef}.
Substituting the decomposition $\bLambda = \bLambda^0+\bLambda^{\parallel}$ in the above, we obtain
\begin{align*}
  \cL_{p,t}^{\sT}(\bLambda^{\parallel}) =-(p-1)\cT^{\sT}_{p,t}(\bLambda^0)\, .
\end{align*}
Since by assumption $\cL_{p,t}$ is well conditioned, it is sufficient to prove that $\cT^{\sT}_{p,t}(\bLambda^0)\psim 0$. 
Since $\cT^{\sT}_{p,t}(\bLambda^0)\in  \spn((\f_r)_{r<t})$, and the Gram matrix $\bG_{1,t-1}=(\<\f_r,\f_s\>)_{r,s<t}$ is well conditioned,
it is sufficient to check that $\<\f_s,\cT^{\sT}_{p,t}(\bLambda^0)\>_N\psim 0$ for each $s<t$. This is in turn equivalent to
\begin{align*}
  \sum_{r',r<t}(\bG_{p-1,t-1}^{-1})_{s,r'} (\bG_{p-2,t-1})_{r',r}\<\f_s,\blambda^0_r\>_N \<\f_s,\f_r\>_N\psim 0\, .
\end{align*}
Finally, this last claim follows by substituting the expression  for $\blambda^0_r$, and using the fact that  $\<\f_s,\tbA\{\f_q\}\>_N\psim 0$
for all $r,q\le t$, by Lemma \ref{lem:4}.

We are now in position to prove the claim of this lemma.
Note that $\tbA\{(\f_t^{\otimes p-1})_{\perp}\} -\cP_t^{\perp} (\tbA)\{(\f_t^{\otimes p-1})_{\perp}\} \psim \bQ\{(\f_t^{\otimes p-1})_{\perp}\}$ and
\begin{align*}
  \bQ\{(\f_t^{\otimes p-1})_{\perp}\} = \frac{(p-1)}{N^{p-1}}\sum_{s<t}\<\lambda_s\otimes \f_s^{\otimes (p-2)},(\f_t^{\otimes (p-1)})_{\perp}\>\, \f_s
  \equiv(p-1)\sum_{s\le t}c_s\f_s\, .
\end{align*}
Since the Gram matrix $\bG_{1,t-1} = (\<\f_s,\f_r\>)_{s,r<t}$ is well conditioned, in order to show $\|\bQ\{(\f_t^{\otimes p-1})_{\perp}\}\|_N\psim 0$,
  it is sufficient to check that each of the coefficients $c_s\psim 0$ for each $s$. Notice
  that $(\f_t^{\otimes (p-1)})_{\perp}  = \sum_{r\le t}\beta_r\f_r^{\otimes (p-1)}$, where the $\beta_s$ are bounded
  thanks to the fact that $\bG_{p-1,t-1}$ is well conditioned. Using $\blambda_s\psim\blambda^0_s$, we get
\begin{align*}
  c_s &\psim  \frac{1}{N^{p-1}} \<\blambda^0_s\otimes \f_s^{\otimes (p-2)},(\f_t^{\otimes p-1})_{\perp}\>\\
      & =  \frac{1}{N^{p-1}}\sum_{r\le t}\sum_{q<t}\beta_r (\bG_{p-1,t-1}^{-1})_{s,r'}\<\tbA\{\f_{q}\}\otimes \f_s^{\otimes(p-2)}, \f_r^{\otimes(p-1)}\>_N\\
  & = \sum_{r\le t}\sum_{q<t}\beta_r (\bG_{p-1,t-1}^{-1})_{s,r'}\<\tbA\{\f_{q}\},\f_t\>_N\<\f_s, \f_r\>_{N}^{p-2} \psim 0\, ,
\end{align*}
    where in the last step we used $\<\tbA\{\f_{q}\},\f_t\>_N\psim 0$, thanks to Lemma \ref{lem:4}.
\end{proof}

\subsubsection{Proof of $(c)$: State evolution}

Recall that the process $(U^{p,t})_{t\ge 1}$ is Gaussian by construction, and independent of $U^{p,0}$, 
Define $C_{r,s} \equiv \E\{U^{p,r}U^{p,s}\}$ and $\bC_{\le t}\equiv (C_{r,s})_{r,s\le t}$. We then have
\begin{align}
\E[U^{p,t+1}| U^{p,0},\dots,U^{p,t}]&=\sum_{s=1}^t \tilde{\alpha}_s U^{p,s}\, ,\\
\tilde\alpha_s & = \sum_{r=1}^t(\bC^{-1}_{\le t})_{s,r}C_{r,t+1}\, .
\end{align}

On the other hand, from point $(a)$, we know that
\begin{align}
\E[\qq^{p,t+1}|\mathcal F_t]& = \sum_{1\leq s\leq t}  \alpha_s \qq^{s,p}\, ,\\
\alpha_s & = \sum_{r=1}^t(\bG^{-1}_{p-1,t-1})_{s-1,r-1} (\bG_{p-1,t})_{r-1,t}\, .
\end{align}
Moreover, by the induction hypothesis we know that, for $r,s\le t$
\begin{align*}
(\bG_{p-1,t})_{r,s} &\psim \E\{f_{r}(X^0,\dots,X^{r}) f_{t}(X^0,\dots,X^{s})\}^{p-1}\, ,
\end{align*}
where we recall that $X^t \equiv \sum_{p\le \km} U^{p,t}$. 
Therefore, using the definition of the process $(U^{p,t})_{t\ge 0}$ we obtain $(\bG_{p-1,t})_{r,s} \psim C_{r+1,s+1}/(pc_p^2)$ for $r,s\le t$, whence
$\alpha_s\psim \tilde\alpha_s$ (where we used the fact that $\bG_{p-1,t}$ is well conditioned by assumption). 
Therefore we also have 
\begin{align}
\Big\|\E[\qq^{p,t+1}|\mathcal F_t] - \sum_{s=1}^t \tilde{\alpha}_s \qq^{p,s}\Big\|_N^2 &= \Big\|\sum_{s=1}^t (\alpha_s-\tilde{\alpha}_s) \qq^{p,s}\Big\|_N^2 \nonumber\\ 
&= \sum_{s,r=1}^t (\alpha_s-\tilde{\alpha}_s) (\alpha_r-\tilde{\alpha}_r) \langle \qq^{p,s},\qq^{p,r}\rangle_N \nonumber\\
&\psim \sum_{s,r=1}^t (\alpha_s-\tilde{\alpha}_s) (\alpha_r-\tilde{\alpha}_r) C_{r,s} \psim 0.
\end{align}
Moreover, Lemma~\ref{lem:4} (point 4) shows that $\cP^{\perp}_t(\tbA^{(p)})\{\f_t\}\psim\tbA^{(p)}\{(\f_{t}^{\otimes p-1})_{\perp}\}$ 
has entries which are approximately independent Gaussian with variance $\sigma^2_t\equiv pc^2_p \|(\f_{t}^{\otimes p-1})_{\perp}\|^2/N^{p-1}$, even conditionally on $\cF_t$. Therefore
\begin{align}
  \qq^{p,t+1} &\ed \sum_{s=1}^t \tilde{\alpha}_s \qq^{p,s} +\sigma_t \bg + \bfe^{p,t+1}\, ,\label{eq:ReprSE}
\end{align}
where $\|\bfe\|_N\psim 0$ and $\bg\sim\normal(\bfzero,\id_N)$ is independent of everything else.
From here on, the rest of the argument for state evolution for pseudo-Lipschitz functions is exactly the same as in Lemma 5 (b) in \cite{berthier2019state}.
As proved in the previous point, for any $s\le t$,
\begin{align*}
  \<\qq^{p,t+1} ,\qq^{p,s+1}\>_N^2&\psim pc_p^2\<\f_t,\f_s\>_N^{p-1}\psim \E\{U^{p,t+1}U^{p,s+1}\}\, .
\end{align*}
Therefore, in order to  prove Eq.~\eqref{eq:ConvergenceLAMP},it is sufficient to consider $\psi:\reals^{D(t+1)}\to\reals$  Lipschitz.
Using the representation \eqref{eq:ReprSE}, and focusing for simplicity on a single $p$, we get
\begin{align*}
  \frac{1}{N}\sum_{i=1}^N\psi(\qq_i^{p,\le t}, q_i^{p,t+1}) &\psim \frac{1}{N}\sum_{i=1}^N\psi\left(\qq_i^{p,\le t},
  \sum_{s=1}^t \tilde{\alpha}_s \qq^{p,s} +\sigma_tg_i\right)\\
  &\psim \frac{1}{N}\sum_{i=1}^N\E\psi\left(\qq_i^{p,\le t},
  \sum_{s=1}^t \tilde{\alpha}_s \qq^{p,s} +\sigma_tG\right)\, ,
\end{align*}
where the second equality follows by Gaussian concentration. At this point we apply the induction hypothesis.

\subsection{Asymptotic equivalence of Tensor AMP and Tensor LAMP}

Here we show that tensor AMP and tensor LAMP produce approximately the same iterates. 
\begin{lemma}\label{lem:ampequalslamp}
Let $\{\bW^{(p)}\}_{p\le \km}$ be standard Gaussian tensors, and $\bA^{(p)} = c_p \bW^{(p)}$ for $p\ge 2$. Consider the corresponding AMP
iterates $\ZZ_{t}\equiv (\bz^{p,s})_{p\le\km,s\le t}$ and LAMP iterates $\bQ_{t}\equiv (\qq^{p,s})_{p\le\km,s\le t}$,
from the same initialization  initialization $\ZZ_0=\bQ_0$ satisfying the assumptions of Theorem \ref{thm:mixedAMP}
and Theorem \ref{thm:SELAMP}.

Let $\f_t = f_t(\VV_t)$, $t\ge 0$ be the nonlinearities applied to LAMP iterates and $(\bG_{p,t}(\VV))_{r,s}=\<\f_t,\f_s\>^p$ be the corresponding Gram matrices.
Further assume that there exist a constant $C<\infty$ such that, for all $t\le T$,
\begin{itemize}
\item[$(i)$] The LAMP Gram matrices $\bG_{p,t} = \bG_{p,t}$ are
  well-conditioned, i.e., $C^{-1}\le \sigma_{\min}(\bG_{p,t})\le \sigma_{\max}(\bG_{p,t})\le C$ for all $p\le \km$, $t\le T$.
\item[$(ii)$] Let the linear operator $\cT_{p,t}:\reals^{N\times t}\to\reals^{N\times t}$ be defined as per Eq.~\eqref{eq:Tdef}, with $\bG_{p,t} = \bG_{p,t}(\VV)$,
and  $\f_t=f_t(\VV)$, and define $\cL_{p,t} = {\boldsymbol 1}+(p-1)\cT_{p,t}$. Then $C^{-1}\le \sigma_{\min}(\cL_{p,t})\le \sigma_{\max}(\cL_{p,t})\le C$.
\end{itemize}
Then, for any $t\le T$, we have
\begin{align}
\|\ZZ_{t} - \bQ_{t}\|_N\psim 0 \, .
\end{align}
\end{lemma}
\begin{proof}
Throughout the proof we will write $f_t(\bX_t)$ or $f_t(\bV_t)$ to distinguish AMP and LAMP iterates, and analogously for
$\bG_{p,t}(\bX_t)$ or $\bG_{p,t}(\bV_t)$.
The proof is by induction over the iteration number, so we will assume it to hold at iteration $t$, 
and prove it for iteration  $t+1$. We prove the induction step by establishing the following two facts:
\begin{align}
\big\|\AMP_{t+1}(\ZZ_{t})_p-\AMP_{t+1}(\bQ_{t})_p\big\|_N&\psim 0 \, ,\label{eq:LAMPapprox1}\\
\big\|\AMP_{t+1}(\bQ_{t})_p-\LAMP_{t+1}(\bQ_{t})_p\big\|_N &\psim 0\, . \label{eq:LAMPapprox2}
\end{align}

Let us first consider the claim \eqref{eq:LAMPapprox1}, and note that
\begin{align*}
\AMP_{t+1}(\ZZ_{t})_p-\AMP_{t+1}(\bQ_{t})_p = \bA^{(p)}\{f_t(\bX_t)\}-\bA^{(p)}\{f_t(\bV_t)\} - \sum_{s\leq t} d_{t,s,p} \big[f_{s-1}(\bX_{s-1})-f_{s-1}(\bV_{s-1})\big]\, ,
\end{align*}
where we wrote $d_{t,s,p}$ for the coefficients of Eq.~\eqref{eq:AMP-def2}, with AMP iterates replaced by LAMP iterates. 
We then have 
\begin{align}
\big\|\AMP_{t+1}(\ZZ_{t})_p&-\AMP_{t+1}(\bQ_{t})_p\big\|_N \le D_{1,t}+D_{2,t}\, ,\\
D_{1,t} & \equiv \big\| \bA^{(p)}\{f_t(\bX_t)\}-\bA^{(p)}\{f_t(\bV_t)\}\big\|_N\, ,\\
D_{2,t} &\equiv \sum_{s\leq t} |d_{t,s,p}|\|f_{s-1}(\bX_{s-1})- f_{s-1}(\bV_{s-1})\|_N\, .
\end{align}
Notice that, by the induction assumption (and recalling that $f_t$ is Lipschitz continuous and acts component-wise): 
\begin{align}
\big\|f_t(\bX_t)-f_t(\bV_t)\big\|_N\le C_T\sum_{s\le t,p\le \km}\|\bx^{p,s}-\bv^{p,s}\|_N \psim 0\, .\label{eq:InductionF}
\end{align}
Further, for any tensor $\bT\in(\reals^{N})^{\otimes p}$, and any vectors $\bv_1,bv_2\in\reals^N$,
\begin{align}
\|\bT\{\bv_1\}-\bT\{\bv_{2}\}\|_N\le (N^{\frac{p-2}{2}}\|\bT\|_{\op})  (\|\bv_1\|_N+\|\bv_2\|_N)^{p-2} \|\bv_1-\bv_2\|_N
\end{align}
Using Lemma~\ref{lem:4}, this implies that the following bound holds with high probability for a constant $C$:
\begin{align}
D_{1,t} &\le C  (\|f_t(\bX_t)\|_N+\|f_t(\bV_t)\|_N)^{p-2} \|f_t(\bX_t)-f_t(\bV_t)\|_N\\
& \le C  (2\|f_t(\bV_t)\|_N+\|f_t(\bX_t) -f_t(\bV_t)\|_N)^{p-2} \|f_t(\bX_t) -f_t(\bV_t)\|_N\psim 0
\end{align}
Where the last step follows from Eq.~\eqref{eq:InductionF} and Theorem~\ref{thm:SELAMP}, which implies (using the fact that $f_t$ is Lipschitz)
$\|f_t(\bV_t)\|_N\le C$ with high probability. Notice that the same argument implies $\|f_t(\bX_t)\|_N \le C$ with high probability.

Similarly, $D_{2,t}\psim 0$ follows since $\|f_{s-1}(\bX_{s-1})- f_{s-1}(\bV_{s-1})\|_N\psim 0$ and $|d_{t,s,p}|\le C_T$ by construction, thus yielding the desired claim
\eqref{eq:LAMPapprox1}.

We now turn to proving Eq.~\eqref{eq:LAMPapprox2}. Comparing Eq.~\eqref{eq:AMP-def2} and \eqref{eq:LAMP1}, and letting 
$\cP_t^{\parallel} = \bfone-\cP_t^{\perp}$ we obtain 
\begin{align}
\AMP_{t+1}(\bQ_{t})_p-\LAMP_{t+1}(\bQ_{t})_p &= \cP_t^{\parallel} (\bA^{(p)})\{f_t(\bV_t)\}
-\ons_{p,t+1}-\sum_{0\leq s\leq t-1}  h_{t,s,p}\qq^{p,s+1}\, ,\label{eq:AMP-LAMP}\\
\ons_{p,t+1} & = \sum_{s\leq t} d_{t,s,p} f_{s-1}(\bV_{s-1})
\end{align}
Note that $\cP_t^{\parallel} (\bA^{(p)})=\E\{\bA^{(p)}|\cF_t\}$, where $\cF_t$ is the $\sigma$-algebra generated by 
$\{\bq ^{p,s}\}_{s\le t, p\le \km}$. Equivalently, this is the conditional expectation of $\bA^{(p)}$ given the linear constraints
\begin{align}
\A^{(p)}\{f_s(\bV_{s})\}&=\by_{p,s+1}\, ,\;\;\;\;\;\mbox{ for } s\in \{0,\dots, t-1\}\, ,
\end{align}
Also notice that, by the induction hypothesis, and the definition of $\by_{p,s}$, Eq.~\eqref{eq:Ydef1}, we have for all $s\le t$, 
\begin{align}
\by_{p,s} & \psim  \qq^{p,s}+\ons_{p,s}\, .\label{eq:Y-ons}
\end{align}
Lemma~\ref{lem:symregression} implies that $\cP_t^{\parallel} (\bA^{(p)})$ takes the form of Eq.~\eqref{eq:SymmRegression} for a suitable matrix
$\hZZ_{p,t}\in\reals^{N\times t}$. 
The key claim is that
\begin{align}
\hZZ_{p,t} \psim \bQ_t\, .\label{eq:ZequalsQ}
\end{align}
In order to establish this claim, we show that,  under the inductive hypothesis,
 \[(\bfone+(p-1)\cT_{p,t})\bQ_t\psim \bY_{p,t}.\] 
 Since $\cL_{p,t}=\bfone +(p-1)\cT_{p,t}$ is well-conditioned by assumption, Eq.~\eqref{eq:ZZ-eq} implies $\hZZ_{p,t}\psim \bQ_t$. Notice that, by Eq.~\eqref{eq:Y-ons}  in order to prove this claim, it is sufficient to show that 
$(p-1)\cT_t\bQ_t\psim \ONS_{p,t}:=[\ons_{p,1}|\cdots|\ons_{p,t}]$.

In order to prove this claim, we use Theorem~\ref{thm:SELAMP}.
Recall $C_{r,s} = \E\{U^{p,r}U^{p,s}\}$, $X^r = \sum_{p} U^{p,r}$ and $\bC_{\le t} = (C_{r,s})_{r,s\le t}$. By Theorem~\ref{thm:SELAMP}, $C_{r+1,s+1} \psim \<\qq^{p,r+1},\qq^{p,s+1}\> 
\psim pc_p^2 (\bG_{p-1,t}(\bV))_{r,s}$ for $r,s\le t$.  This implies for any $0 \le r\le t-1$,
\begin{align}
\sum_{j=0}^{t-1}(\bG_{p-1,t-1}^{-1})_{rj}\< \qq^{p,j+1},f_{t-1}(\VV_{t-1})\>_N &\psim pc_p^2 
\sum_{j=0}^{t-1} (\bC_{\le t}^{-1})_{r+1,j+1} \E\{U^{p,j+1}f_{t-1}(X^0,\dots,X^{t-1})\} \nonumber\\
& = pc_p^2\, \E\left\{\frac{\partial f_{t-1}}{\partial x^{r+1}} (X^0,\dots,X^{t-1})\right\} \bfone_{r\le t-2}\, , \label{eq:Stein}
\end{align}
%
where we used Stein's lemma in the second equality. 
Using this last expression and the definition \eqref{eq:AMP-def2} 
allows to check we conclude $(p-1)\cT_{p,t}\bQ_t\psim \ONS_{p,t}$ as claimed. Indeed we have
\begin{align*}
(p-1)\big[\cT_{p,t}\bQ_t\big]_t &= \sum_{r=0}^{t-1} (\bG_{p-2,t-1})_{r,t-1} \f_r  \Big(\sum_{r'=0}^{t-1} (\bG_{p-1,t-1}^{-1})_{r,r'} \langle \qq^{p,r'+1}, \f_{t-1}\rangle \Big)\\
&\psim p(p-1)c_p^2 \sum_{r=0}^{t-2}\langle \f_r, \f_{t-1}\rangle_N^{p-2} \f_r \cdot \E\left\{\frac{\partial f_{t-1}}{\partial x^{r+1}} (X^0,\dots,X^{t-1})\right\}\\
&= \ons_{p,t}.
\end{align*}

Having established  Eq.~\eqref{eq:ZequalsQ}, we can use the representation of  $\cP^{\parallel}_t(\bA^{(p)})=\E\{\bA^{(p)}|\cF_t\}$ given in Eq.~\eqref{eq:SymmRegression} to get
\begin{align}
 \cP^{\parallel}_{t}(\bA^{(p)})\{\f_t\}&\psim \sum_{s\leq t}\alpha_s\qq^{p,s}+(p-1) \sum_{s\leq t}\beta_s \f_s\, ,\label{eq:AParallel}\\
 \alpha_s &= \sum_{0\leq r\le t-1} (\bG_{p-1,t-1}^{-1})_{s,r} \langle f_r(\VV_r),f_{t}(\VV_t)\rangle_N^{p-1}\, ,\\
\beta_s &= \Big(\sum_{0\le r\leq t-1} (\bG_{p-1,t-1}^{-1})_{s,r} \langle \qq^{p,r},\f_{t} \rangle_N\Big) \langle \f_s,\f_{t}\rangle_N^{p-2} .
\end{align}

On the other hand, using again Eq.~\eqref{eq:Stein}, we obtain
\begin{align}
(p-1) \sum_{s\leq t}\beta_s \f_s &\psim  \sum_{s \le t-1}  d_{t,s,p} \f_{s-1}  =\ons_{p,t+1},\\
\mbox{and}~~~\sum_{s\leq t}\alpha_s\qq^{p,s} &\psim \sum_{0\leq s\leq t-1}  h_{t,s,p}\qq^{p,s+1}.
\end{align}
%
%
%
%
%
%
%

We therefore conclude, from Eq.~\eqref{eq:AMP-LAMP}, that $\|\AMP_{t+1}(\bQ_{t})_p-\LAMP_{t+1}(\bQ_{t})_p \|_N\psim 0$, and this finishes our proof.
\end{proof}

\subsection{Reduction to the well-conditioned case}

Theorem \ref{thm:SELAMP} and Lemma \ref{lem:ampequalslamp} imply the conclusion of the main statement Theorem \ref{thm:mixedAMP},
under the additional assumptions in points $(i)$ and $(ii)$ of Lemma \ref{lem:ampequalslamp}. 
Here we show how to approximate an arbitrary AMP algorithm with one satisfying those conditions, completing the proof of Theorem \ref{thm:mixedAMP}.
This strategy was already employed in \cite{javanmard2013state,berthier2019state}, and we refer to these references for further background.
\begin{lemma}\label{lemma:Perturbation}
  Let $(f_t)_{t\ge 0}$, with $f_t:\reals^{t+1}\to \reals$, be any sequence of Lipschitz functions.
  Then for any $\eps>0$ there exists a sequence of smooth  functions $\varphi_{t}:\reals^{t+1}\to \reals$, with $\|\varphi_t\|_{L^\infty}\le 1$,
  $\|\nabla \varphi_t\|_{L^\infty}\le 1$,  such that the following holds.
  Defining $f^{\eps}_t =f_t+\eps\varphi_t$, the sequence of functions  $(f^{\eps}_t)_{t\ge 0}$ 
  satisfies conditions $(i)$ and $(ii)$ of Lemma \ref{lem:ampequalslamp}.
\end{lemma}
The proof of this lemma is presented in the next two subsections, considering first condition $(i)$, and then condition $(ii)$.
Before presenting this proof, we show that this lemma indeed allows to prove Theorem~\ref{thm:mixedAMP}.
\begin{proof}[Proof of Theorem~\ref{thm:mixedAMP}]
  Let $(f^\eps_t)_{t\in\naturals}$ be a sequence of functions as per Lemma \ref{lemma:Perturbation}, and denote by $\bz^{\eps,p,t}$
  the corresponding iterates, and $\bZ^{\eps}_t =(\bz^{\eps,p,s})_{p\le \km,s\le t}$. We instead use $\bZ_t = (\bz^{p,s})_{p\le \km,s\le t}$
  for the unperturbed AMP iteration. 
  Using the same argument as in the proof of Lemma \ref{lem:ampequalslamp} (in particular, the argument to prove
  Eq.~\eqref{eq:LAMPapprox1})  we obtain, for every fixed $t$,
  \begin{align}
    \plim_{\eps\to 0}\limsup_{N\to\infty}\|\bZ_t-\bZ^{\eps}_t\|_N=0\, .\label{eq:Zpert}
  \end{align}
  On the other hand, for any $\eps>0$, the iterates satisfy the non-degeneracy
  conditions $(i)$ and $(ii)$ of Lemma \ref{lem:ampequalslamp}. We can therefore apply this lemma, and Theorem \ref{thm:SELAMP}
  to conclude that, for any test pseudo-Lipschitz function $\psi:\reals^{T\times \km}\to\reals$, we
  have
\begin{align}
  \plim_{N\to\infty}\frac{1}{N}\sum_{i=1}^N\psi((z_i^{\eps,p,t})_{t\le T, p\le \km}) = \E\big\{\psi\big((U^{\eps,t,p})_{t\leq T,p\leq \km}\big)\big\}\, .
  \label{eq:LimUeps}
\end{align}
Here $(U^{\eps,t,p})_{t\ge 0, p\le \km}$ is the Gaussian process associated to the nonlinearities $(f^{\eps}_t)_{t\ge 0}$, namely with covariance
determined recursively via
\begin{align}
\E[U^{\eps,t+1,p}U^{\eps,s+1,p}]& = pc^2_p\E\left\{ f^{\eps}_t\left(X^{\eps,0},\dots,X^{\eps,t}\right)f^{\eps}_s\left(X^{\eps,0},\dots,X^{\eps,s}\right)\right\}^{p-1}\, ,\\
X^{\eps,t} & \equiv  \sum_{k=2}^{\km} U^{\eps,t,k}\, .
\end{align}
Recalling that $f_t^{\eps}=f_t+\eps\varphi_t$ with $\varphi_t$ bounded, with bounded gradient, it is immediate to show by
induction that $\E[U^{\eps,t,p}U^{\eps,s,p}]\to \E[U^{t,p}U^{s,p}]$ as $\eps\to 0$. In particular,
it is possible to couple $(U^{\eps,t,p})_{t\ge 0, p\le \km}$ and $(U^{t,p})_{t\ge 0, p\le \km}$ so that $\E\{(U^{\eps,t,p}-U^{t,p})^2\}\to 0$
for any $t,p$. 
We thus conclude that
\begin{align*}
  \plim_{N\to\infty}\frac{1}{N}\sum_{i=1}^N\psi((z_i^{p,t})_{t\le T, p\le \km}) \stackrel{(a)}{=}
  \lim_{\eps\to 0}\plim_{N\to\infty}\frac{1}{N}\sum_{i=1}^N\psi((z_i^{\eps,p,t})_{t\le T, p\le \km}) \\
  \stackrel{(b)}{=}
  \lim_{\eps\to 0}\E\big\{\psi\big((U^{\eps,t,p})_{t\leq T,p\leq \km}\big)\big\}
\stackrel{(c)}{=}\E\big\{\psi\big((U^{\eps,t,p})_{t\leq T,p\leq \km}\big)\big\}
  \,
\end{align*}
where $(a)$ follows from Eq.~\eqref{eq:Zpert}, $(b)$ from Eq.~\eqref{eq:LimUeps}, and $(c)$ from the remark
that $\E\{(U^{\eps,t,p}-U^{t,p})^2\}\to 0$.
\end{proof}

\subsubsection{Condition $(i)$: Control of $\bG_{p,t}$}
\label{sec:Condition1}

We begin with condition $(i)$  which requires  $C^{-1}\le \sigma_{\min}(\bG_{p,t})\le \sigma_{\max}(\bG_{p.t})\le C$ with high probability for
some constant $C$ independent of $N$. Note that 
Lemma \ref{lem:ampequalslamp}  requires these bounds to hold for a finite collections of values of $p$, $t$. Since this collection is fixed independently of $N$,
it is sufficient to consider a single pair  $(p,t)$.
By Theorem~\ref{thm:SELAMP}, we know that 
\begin{align}
  \plim_{N\to\infty} (\bG_{p,t})_{r,s} = (\bG_{p,t}^{\infty})_{r,s}\:= \big(\E\{f_r(X_0,\dots,X_r)f_s(X_0,\dots,X_s)\}\big)^p\,.
\end{align}
It is therefore sufficient to prove $\sigma_{\min}(\bG_{p,t}^{\infty})>0$ for all $p,t$.
Note that $\sigma_{\min}(\bG_{p,t}^{\infty})<\infty$ is immediate since $\bG_{p,t}^{\infty}$ has finite entries, and is a matrix
of fixed dimensions $t+1\times t+1$.

Recall that Hadamard product preserves positive-semidefinite (PSD) ordering:  if $\bA_1\succeq \bB_1\succeq\bfzero$ and
$\bA_2\succeq \bB_2\succeq\bfzero$, then $\bA_1\odot\bA_2\succeq \bB_1\odot\bB_2$.
(This follows from decomposing any PSD matrices as a sum of rank-one PSD matrices.)
In particular, $\bG_{1,t}^{\infty}\succeq C\id$ implies $\bG_{p,t}^{\infty}\succeq C^p\id$. It is therefore
sufficient to prove $\sigma_{\min}(\bG_{1,t}^{\infty})>0$, which we do in the next lemma
\begin{lemma}\label{lem:perturbinvertible}
  Under the assumptions of Lemma \ref{lemma:Perturbation}, there exist functions
  $\varphi_{t}:\reals^{t+1}\to \reals$, with $\|\varphi_t\|_{L^\infty}\le 1$,
  $\|\nabla \varphi_t\|_{L^\infty}\le 1$, and an $\eps_0>0$ such that the following holds.
  Letting $\bG_{1,t}^{\infty}$ denote  the Gram matrices associated to $(f^{\eps}_t)_{t\ge 0}$
  we have $\sigma_{\min}(\bG_{1,t}^{\infty})>0$ for $\eps<\eps_0$.
\end{lemma}
\begin{proof}
  We construct $\varphi_t$ satisfying the claim inductively in $t$. The base case is clear: $\bg_{1,0}^{\infty} = \E\{f_0(X_0)\}^2 >0$ for $f_0$ non vanishing.  Assuming we have constructed
  these functions up to $\varphi_{t-1}$, we know that the vector $(X^{\eps}_1,X^{\eps}_2,\dots,X^{\eps}_t)$
  defined by state evolution (for nonlinearities $f^{\eps}_t$) is a non-degenerate Gaussian.
  
  In order to prove our claim, we need to construct $\varphi_{t}$ so that the vector
  $\{f^{\eps}_s(X_0^{\eps},\dots,X_{s}^{\eps})\}_{s\le t}$ has non-degenerate covariance. Since we know already that
  $\{f^{\eps}_s(X_0^{\eps},\dots,X_{s}^{\eps})\}_{s\le t-1}$ is non-degenerate, it is sufficient to show that,
  for any coefficients $(\alpha_s)_{s \le t}$,
  \begin{align}
    \E\Big\{\Big(f^{\eps}_t(X_0^{\eps},\dots,X_{t}^{\eps})-\sum_{s\le t-1}\alpha_s f^{\eps}_s (X_0^{\eps},\dots,X_{s}^{\eps})\Big)^2\Big\}>0\, .\label{eq:NonDegG}
  \end{align}
  It is always possible to choose $\varphi_t$ so that this is the case. Indeed, the space of functions spanned by $f^{\eps}_s$ for $s\le t$
  has dimension at most $t$. Therefore, we can take any $t+1$ linearly independent bounded smooth functions of $x_t$ only,
and choose $\varphi_t$ to be a linear combination of these that is outside the span of $(f^{\eps}_s)_{s\le t-1}$.
Since non-degenerate Gaussians have full support, this implies the non-degeneracy condition \eqref{eq:NonDegG}
and therefore the induction claim.
\end{proof}

In preparation for the next part, we argue that when the Gram matrices $\bG_{1,t}^{\infty}$ are non-degenerate, we can
perturb the nonlinearities $(f_t)_{t\ge 0}$ to induce any desired small change in $\bG_{1,t}^{\infty}$. (Below
$\Sym_m$ denotes the space of $m\times m$ symmetric matrices.)
\begin{lemma}\label{lem:perturbsurjective}
  Under the assumptions of Lemma \ref{lemma:Perturbation}, assume the nonlinearities $(f_t)_{t\ge 0}$ are such that $\bG_{1,t}^{\infty}$
  is non-degenerate. Then there exists finite sets of functions $A_s= \{\varphi_{s,1},\dots,\varphi_{s,n(s)}\}$ of
  smooth functions $\varphi_{s,j}:\reals^s\to\reals$, with $\|\varphi_{s,j}\|_{L^\infty}\le 1$,
  $\|\nabla \varphi_{s,j}\|_{L^\infty}\le 1$, such that the following is true.
  For $\beps =(\eps_{s,j})_{j\le n(s), s\le t}\in\reals^{n_*}$, $n_*:=\sum_{s\le t} n(s)$, consider the nonlinearities
$(f_s^{\beps})_{s\le t}$ defined by $f^{\beps}_s=f_s+\sum_{j\le n(s)} \eps_{s,j}\varphi_{s,j}$, and let
  $\bG^{\infty}_{1,t}(\beps)$ to be the corresponding (asymptotic) Gram matrix.
  If $\bcG_t:\reals^{n_*}\to \Sym_{t}$ is the mapping $\bcG_t:\beps \mapsto \bG^{\infty}_{1,t}(\beps)$,
  then its derivative $D\bcG_t|_{\beps=\bfzero}$ is surjective. 
\end{lemma}
\begin{proof}
 Note that $\Sym_{t}\cong \reals\times \reals^2\times\cdots\times \reals^t$,
 by identifying $\bM\in\Sym_t$ which  a list of columns $M_{11}$, $(M_{1,2},M_{2,2})$,  \dots, $(M_{j,t})_{j\le t}$.
 Also $\reals^{n_*}\cong \reals^{n(1)}\times\cdots\times \reals^{n(t)}$, by identifying $\beps = (\beps_1,\dots,\beps_t)$,
 $\beps_s=(\eps_{s,j})_{j\le n(s)}$. The matrix $D\bcG_t|_{\beps=\bfzero}$ is block-triangular with respect to this decomposition.
 By an induction argument, it is therefore sufficient to show that $A_t$ can be constructed so that  the last diagonal block
 $D\bcG_t|_{\beps=\bfzero}:\reals^{n(t)}\to\reals^t$ is surjective.

 Note that $\bcG_t$ is  the map that takes as input $\eps_t$, and outputs the last column of the
 asymptotic Gram matrix corresponding to the nonlinearities $f_1,\dots, f_{t-1}$ and
 $f^{\beps}_t=f_t+\sum_{j\le n(t)} \eps_{t,j}\varphi_{t,j}$.
 Since by assumption $\bG_{t,1}$ is non-degenerate, the  functions $f_1,\dots,f_t$ are linearly independent
 (viewed as vectors in the $L^2$ space associated to the joint  distribution of $(\bX_s)_{s\le t}$).
 We can therefore construct functions $(\varphi_{t,s})_{s\le t}$ such that $\E\{\varphi_{t,s}(X_0,\dots,X_t)f_r(X_0,\dots,X_r)\} = 0$
 if $r\neq s$, and $>0$ if $r=s$. It is then immediate to show that the resulting map $D\bcG_t|_{\beps=\bfzero}$ is surjective.
\end{proof}

\subsubsection{Condition $(ii)$: Control of $\cL_{p,t}$}

We are left with the task of showing that --after a small perturbation of the nonlinearities $(f_t)_{t\ge 0}$-- condition $(ii)$ of
Lemma \ref{lem:ampequalslamp} holds, namely  $C^{-1}\le \sigma_{\min}(\cL_{p,t})\le \sigma_{\max}(\cL_{p,t})\le C$ for all $p\le\km$,
$t\le T$, with high probability.
Given the results of the previous section \ref{sec:Condition1}, we can assume without loss of generality that
$C^{-1}\le \sigma_{\min}(\bG_{p,t}^{\infty})\le \sigma_{\max}(\bG_{p,t}^{\infty})\le C$ for all $p,t$. Indeed, if this is not the case,
we can modify the nonlinearities as described above, as to satisfy this condition. Also, as before, we can consider a single pair $(p,t)$
since we only are interested in a finite (independent of $N$) collection of such pairs.

Recall that $\cL_{p,t} = \bfone+(p-1)\cT_{p,t}$, and, by Eq.~\eqref{eq:Tdef},
\begin{align}
(\cT_{p,t})_{is;jr} = \sum_{r'=0}^{t-1}F_{ir'}F_{js}(\bG_{p-1,t-1}^{-1})_{r',r} (\bG_{p-2,t-1})_{r',s}\, ,
\end{align}
where  $F_{is} = (\bF_{t-1})_{is}= (\f_s)_i$ for $0\le s\le t-1$, $\bF_{t-1}\in\reals^{N\times t}$ (for consistency,
we index the columns of $\bF_{t-1}$ as $0,\dots,t-1$). This implies that $\cT_{p,t}$ has rank at most $t^2$ since
\begin{align}
  (\cT_{p,t})_{is;jr} &= \sum_{a,b=0}^{t-1} (\cU_{p,t})_{as;br}F_{ir'}F_{js}\, ,\\
  (\cU_{p,t})_{as;br} & := (\bG_{p-1,t-1}^{-1})_{ra} (\bG_{p-2,t-1})_{sa}\delta_{b,s}\, ,
\end{align}
or, in matrix notation
\begin{align}
  \cT_{p,t} = (\id_t\otimes \bF_{t-1})\cU_{p,t}(\id_t\otimes \bF^{\sT}_{t-1})\, .
\end{align}
It follows that the $(N-t)t$ singular values of $\cL_{p,t}$ are equal to $1$, and the other $t^2$ singular values coincide with the
ones of $\tcL_{p,t} = \bfone_{t^2} + (p-1)\tcT_{p,t}$, where
\begin{align}
  \tcT_{p,t} = (\id_t\otimes \bG_{1,t-1}^{-1/2}) \cU_{p,t} (\id_t\otimes \bG_{1,t-1}^{-1/2})\, .
\end{align}
Indeed $\tcT_{p,t}$ is unitarily equivalent to $\cT_{p,t}$ (when the latter is restricted to its range), using the fact that
$\bF_{t-1}^{\sT}\bF_{t-1}/N = \bG_{1,t-1}$.

We now proceed by induction over the iteration number. Assuming the claim to hold up to iteration $t-1$, we need to
to show that (for a suitable perturbation of the nonlinearities)
$C^{-1}\le \sigma_{\min}(\tcL_{p,t})\le\sigma_{\max}(\tcL_{p,t})\le C$ with high probability.
By using the induction hypothesis Theorem \ref{thm:SELAMP} and  Lemma \ref{lem:ampequalslamp} we know that $\bG_{p,t}$
converges in probability to the deterministic limit $\bG^{\infty}_{p,t}$ which is non-degenerate. Therefore,
it is sufficient to prove that (again, for a suitable perturbation of the nonlinearities)
$C^{-1}\le \sigma_{\min}(\tcL^{\infty}_{p,t})\le\sigma_{\max}(\tcL^{\infty}_{p,t})\le C$, where
$\tcL^{\infty}_{p,t} = \bfone_{t^2} + (p-1) \tcT^{\infty}_{p,t}$, and $\tcT^{\infty}_{p,t}$ is obtained from $\tcT_{p,t}$
by replacing $\bG_{k,s}$ by its asymptotic version $\bG_{k,s}^{\infty}$ everywhere. Since the resulting matrix
$\tcL^{\infty}_{p,t}$ is finite (and of dimension independent of $N$), it is sufficient to prove that
$\sigma_{\min}(\tcL^{\infty}_{p,t})>0$. Since $\bG_{1,t-1}^{\infty}$ is non-degenerate,
it is sufficient to prove $\sigma_{\min}(\cW^{\infty}_{p,t})>0$, where
\begin{align}
  \cW^{\infty}_{p,t}&:= \id_t\otimes \bG^{\infty}_{1,t-1}+(p-1) \cU^{\infty}_{p,t}\, , \label{eq:Wdef1}\\
   (\cU^{\infty}_{p,t})_{as;br} & := ((\bG^{\infty}_{p-1,t-1})^{-1})_{ra} (\bG^{\infty}_{p-2,t-1})_{sa}\delta_{b,s}\, .\label{eq:Wdef2}
  \end{align}

  In order to prove the desired non-degeneracy bound for $\cW^{\infty}_{p,t}$, it is useful to introduce a piece of terminology.
\begin{definition}
We say a subset $S\subseteq\reals^d$ is \emph{locally full} if for any open set $U\subseteq R^d$ with $U\cap S\neq\emptyset$ we have $\lambda(U\cap S)>0$ (with $\lambda$ denoting the Lebesgue measure on $\reals^d$).
\end{definition}
For instance, a full-dimensional convex set is locally full.

\begin{lemma}
\label{lem:AG}
Let $K\subseteq \mathbb R^d$ be locally full and $R:\reals^d\to\reals$ a rational function which is not identically zero or infinity.
For any $\eps>0$ and $\bx\in K$ there is $\bx'\in K$ with $\|\bx-\bx'\|\le \eps$ and $R(\bx')\not\in\{0,\pm\infty\}$. 
\end{lemma}
\begin{proof}
Simply recall that any nontrivial polynomial vanishes on a measure zero set.
\end{proof}

We are now in position to show that the nonlinearities $(f_s)_{0\le s\le t}$ can be modified so that the
resulting matrix $\cW^{\infty}_{p,t}$ has $\sigma_{\min}(\cW^{\infty}_{p,t})>0$, thus completing the proof.
\begin{lemma}\label{lem:perturbdone}
  Under the assumptions of Lemma \ref{lemma:Perturbation}, further assume the nonlinearities $(f_s)_{s\ge 0}$
  to be such that $\sigma_{\min}(\bG_{p,t}^{\infty})>0$ for all $p\le \km$, $t\le T$. Then, for any $\eps>0$  there exist functions
  $\varphi_{s}:\reals^{s+1}\to \reals$, with $\|\varphi_s\|_{L^\infty}\le 1$,
  $\|\nabla \varphi_s\|_{L^\infty}\le 1$, such that the following holds.

  Let $\cW^{\infty}_{p,t} (\eps)$ the matrix defined in Eqs.~\eqref{eq:Wdef1}, \eqref{eq:Wdef2},
  for nonlinearities $f^{\eps}_s = f_s+\eps\varphi_s$, $s\le t$. Then, for  any $p\le\km$ and $t \le T$,
  $\sigma_{\min}(\cW^{\infty}_{p,t}(\eps))>0$.
\end{lemma}
\begin{proof}
  Notice that $\cW^{\infty}_{p,t}$ is a function of the matrix $\bG^{\infty}_{1,t}$ (the matrices $\bG^{\infty}_{p,t}$
  being themselves Hadamard
  powers of $\bG_{1,t}$). With a slight abuse of notation, we will write $\cW^{\infty}_{p,t}=\cW_{p,t}^{\infty}(\bG^{\infty}_{1,t})$.
  Define  $R: \Sym_{t+1}\to\reals$ to be the function that takes as input a $t+1\times t+1$ symmetric matrix $\bG$
  and outputs
  \begin{align}
    R(\bG) \equiv\prod_{p\le \km}\det(\cW_{p,t}^{\infty}(\bG))\, .
  \end{align}
  By checking Eqs.~\eqref{eq:Wdef1}, \eqref{eq:Wdef1}, we see that this is a rational function on
  $\Sym_t\cong \reals^{\binom{t+1}{2}}$. Further, it is not identically zero or infinity, as it can be checked by
  computing $\cW_{p,t}^{\infty}(\id)$. Applying Lemma \ref{lem:AG}  to the set of PSD matrices,
  which is locally full in $\reals^{\binom{t}{2}}$, and the rational function $R$, we obtain that, for any $\xi>0$, there exists
  $\bG_*\succeq\bfzero$, with $\|\bG_*-\bG^{\infty}_{p,t}\|_F\le \xi$, and  $R(\bG_*)\not\in \{0,\pm\infty\}$,
  which implies $\sigma_{\min}(\cW_{p,t}^{\infty}(\bG_*))>0$ for all $p\le \km$.

  Finally, using Lemma \ref{lem:perturbsurjective} and the implicit function theorem, we
  conclude that we can find a perturbation $(\varphi_s)_{s\le t}$, and $\eps_0>0$ such that $\bG_{1,t}(\eps)= \bG_*$.
  By taking $\xi$ sufficiently small, we can ensure that $\eps$ can also be arbitrarily small.  
\end{proof}

\subsection{Extension to the case $\km = \infty$}
\label{sec:extension_infinite_mixture}
Here we extend the state evolution result proved for finite mixtures to the general case where $\xi$ has infinitely many components. The proof proceeds by induction over the number of iterations, and is similar to previous arguments. 
Let us write $\tilde{\xi}(x) := \sum_{ k \le D} c_k^2 x^k$ while $\xi(x) = \sum_{ k =2}^{\infty} c_k^2 x^k$. Denote by $(\tilde{X}^{0},\cdots,\tilde{X}^{\ell})$ the state evolution Gaussian process  corresponding to $\tilde{\xi}$, and $(X^{0},\cdots,X^{\ell})$ the one based on $\xi$. 
First, using the fact that $f_\ell$ is Lipschitz, it is easy to show
by induction over $\ell$ that there exists a coupling such that
$\E[(\tilde{X}^\ell - X^\ell)^2] = o_D(1)$ (throughout this section, $o_D(1)$ is a term independent of $N$ that vanishes as $\km\to\infty$). We deduce from this that $\tilde{d}_{\ell,j} - d_{\ell, j} = o_D(1)$ for all $\ell, j$. (Here, $\tilde{d}_{\ell,j}$ is defined similarly to $d_{\ell ,j}$, based on the mixture $\tilde{\xi}$.)

Next we show that the AMP iterates are close. Let $\tilde{\bz}^0,\cdots \tilde{\bz}^\ell$ be the AMP iterates based on $\tilde{\xi}$ and $\bz^0,\cdots \bz^\ell$ those based on $\xi$. Let $\tilde{\bz}^0 = \z^0 = \bfzero$ and assume $\lim_{D \to \infty} \plim_{N\to \infty} \|\tilde{\bz}^j - \bz^j\|_N = 0$ for all $j \le \ell$. Further let $\tilde{\f}_\ell = f_{\ell}(\tilde{\bz}^0,\cdots,\tilde{\bz}^\ell)$. Then
\begin{align}
\big\|\tilde{\bz}^{\ell+1} - \bz^{\ell+1}\big\|_N &\le \Big\| \sum_{p=2}^D \frac{c_p}{p!} \W^{(p)} \{\tilde{\f}_\ell\} - \sum_{p=2}^{\infty} \frac{c_p}{p!} \W^{(p)} \{\f_\ell\}\Big\|_N + \Big\| \sum_{j=0}^\ell \tilde{d}_{\ell,j}\tilde{\f}_{j-1} - d_{\ell, j} \f_j\Big\|_N\\
&=: E_1 + E_2. 
\end{align}
We have 
\[E_1 \le   \sum_{p \ge \km+1} \frac{c_p}{p!} \big\|\W^{(p)} \{\tilde{\f}_\ell\}\big\|_N + \sum_{p=2}^{\infty} \frac{c_p}{p!} \big\|\W^{(p)} \{\f_\ell\}-\W^{(p)} \{\tilde{\f}_\ell\}\big\|_N.\]
The first in the above is bounded by 
\[\sum_{p \ge \km+1} \frac{c_p}{p!} N^{(p-2)/2} \|\W^{(p)}\|_{\textup{op}} \cdot \|\tilde{\f}_\ell\|_N^{p-1}.\] 
Using Theorem~\ref{thm:SELAMP}, $\|\tilde{\f}_\ell\|_N \le C$ with high  probability. Lemma~\ref{lem:4} then implies that the above is $o_D(1)$  with high probability. Next, the second term in $E_1$ is similarly bounded by 
\[\sum_{p =2}^\infty \frac{c_p}{p!} N^{(p-2)/2} \|\W^{(p)}\|_{\textup{op}} \cdot (\|\tilde{\f}_\ell\|_N+\f_\ell\|_N)^{p-2}\|\tilde{\f}_\ell - \f_\ell\|_N.\]    
Since $f_\ell$ is Lipschitz, and using the induction hypothesis, similar considerations show that this term converges to zero in probability as $N \to \infty$. 
Next,
\[E_2 \le  \sum_{j=0}^\ell (\tilde{d}_{\ell,j}- d_{\ell, j})\|\tilde{\f}_{j-1}\|_N +   \sum_{j=0}^\ell  |d_{\ell, j}| \|\tilde{\f}_{j-1} - \f_{j-1} \|_N
\psim o_D(1).  \]
This implies \[\lim_{D \to \infty} \plim_{N\to \infty} \big\|\tilde{\bz}^{\ell+1} - \bz^{\ell+1}\big\|_N =0,\]
which concludes the inductive argument. 
Finally, for $\psi$ a pseudo-Lipschitz function, we have
\begin{align}
\plim_{N\to \infty} \frac{1}{N}\sum_{i=1}^N \psi(\bz^0,\cdots,\bz^\ell) &= \plim_{N\to \infty} \frac{1}{N}\sum_{i=1}^N \psi(\tilde{\bz}^0,\cdots,\tilde{\bz}^\ell) + o_D(1)\\
&= \E [\psi(\tilde{X}^{0},\cdots,\tilde{X}^{\ell})] + o_D(1)\\
&= \E [\psi(X^{0},\cdots,X^{\ell})] +o_D(1).
\end{align}
This concludes our proof of state evolution, Proposition~\ref{prop:state_evolution}.


\bibliographystyle{amsalpha}

\newcommand{\etalchar}[1]{$^{#1}$}
\providecommand{\bysame}{\leavevmode\hbox to3em{\hrulefill}\thinspace}
\providecommand{\MR}{\relax\ifhmode\unskip\space\fi MR }
\providecommand{\MRhref}[2]{%
  \href{http://www.ams.org/mathscinet-getitem?mr=#1}{#2}
}
\providecommand{\href}[2]{#2}


\end{document}